\documentclass[twoside,leqno,11pt]{article}

\usepackage{a4wide}

\usepackage{amsmath}

\usepackage{color}
\usepackage{amsmath,amssymb,enumerate}
\usepackage{chemarr}
\usepackage{pstricks}

\usepackage{epsfig}

\usepackage{bbm}

\usepackage{hyperref}
\usepackage{dsfont}
\usepackage[all]{xy}

\usepackage{graphics}
\usepackage{pst-3d,float}
\newtheorem{theorem}{Theorem}[section]

\newtheorem{definition}{Definition}[section]
\newtheorem{lemma}{Lemma}[section]
\newtheorem{proposition}[theorem]{Proposition}
\newtheorem{remark}{Remark}[section]
\newenvironment{proof}{\medskip\noindent{\bf Proof.}\;}{\null\hfill $\Box$\par\medskip }

\newcommand{\cf}{\ensuremath{\mathcal F}}
\newcommand{\supp}{{\rm supp\,}}

\newcommand{\Beas}{\begin{eqnarray*}}
\newcommand{\Eeas}{\end{eqnarray*}}

\title{Hyperbolic wavelet analysis of classical isotropic and anisotropic Besov-Sobolev spaces}
\author{Martin Sch\"afer$^a$\footnote{Corresponding author. Email:
		martin.schaefer@math.tu-chemnitz.de}, Tino Ullrich$^a$, B\'eatrice Vedel$^b$ \\\\
$^a$Chemnitz Technical University, 09111 Chemnitz, Germany\\
$^b$Universit\'e de Bretagne Sud, UMR 6205, LMBA, 56000 Vannes, France}

\date{\today}

\DeclareMathOperator{\sgn}{sgn}

\def \Z {\mathbb{Z}}
\def \N {\mathbb{N}}
\def \R {\mathbb{R}}
\def \C {\mathbb{C}}
\def \S {{\mathcal{S}}}
\newcommand{\E}{\mathbb{E}}

\def \B {\widetilde{B}}
\def \F {\widetilde{F}}

\def \bj{\bar{j}}
\def \br{\bar{r}}
\def \bk{\bar{k}}
\def \bm{\bar{m}}
\def \bw{\bar{w}}
\def \bl{\bar{\ell}}
\def \balpha{\bar{\alpha}}
\def \bgamma{\bar{\gamma}}
\def \sign{\mathrm{sign}}

\def\bbone{{\mathbbm 1}}

\newcommand{\amin}{\alpha_{\text{min}}}

\begin{document}

\maketitle

\begin{abstract}
In this paper we introduce new
function spaces which we call anisotropic hyperbolic Besov and Triebel-Lizorkin spaces. 
Their definition is based on a
hyperbolic Littlewood-Paley analysis involving an anisotropy vector only occurring in the smoothness weights. Such spaces provide a general and natural setting in order to understand what kind of anisotropic smoothness can be described using hyperbolic wavelets (in the literature also sometimes called tensor-product wavelets), a wavelet class which hitherto has been mainly used to characterize spaces of dominating mixed smoothness.

A centerpiece of our present work are characterizations of these new spaces based on the hyperbolic wavelet transform. Hereby we treat both, the standard approach using wavelet systems equipped with sufficient smoothness, decay, and vanishing moments, but also the very simple and basic hyperbolic Haar system.

The second major question we pursue
is the relationship between the novel hyperbolic spaces and the classical anisotropic Besov-Lizorkin-Triebel scales. 
As our results show, in general, both approaches to
resolve an anisotropy do not coincide.
However, in the Sobolev range this is the case, providing a link to apply the
newly obtained
hyperbolic wavelet characterizations to the classical setting.
In particular, this allows for detecting classical
anisotropies via the coefficients of a universal hyperbolic wavelet basis, without 
the need of adaption of the basis or a-priori knowledge on the anisotropy.
\end{abstract}

\noindent
{\bf Keywords :} Hyperbolic wavelet analysis, Hyperbolic Haar wavelet, Anisotropic Sobolev and Besov-Lizorkin-Triebel spaces.\\
{\bf 2010 Mathematics Subject Classification : } 42C40, 46E35, 42B25, 41A25.

\section{Introduction}
With the development of wavelet analysis from the beginning of the 1980s until the present time we nowadays have several powerful tools at
hand to perform signal analysis with the aim to extract important information out of a signal. The information is thereby
usually coded in objects easy to compute and handle -- the wavelet coefficients.

Wavelet methods have been used with the known success for the purpose of compression, denoising, inpainting, classification, etc., of data, to
mention just a few. Roughly speaking, the common underlying idea is the fact that a few wavelet coefficients contain a
rather complete information of the signal to be analyzed. However, due to their construction principle (dyadic dilations
and integer translates of a few basic ``mother'' functions) classical wavelets are not well-suited for the analysis of,
say, anisotropic signals. In fact, a signal which is rather smooth in $x$-direction but rough in $y$-direction
(such as layers in the earth, stripes on a shirt, etc.) can not be properly resolved by a classical multi-resolution analysis.
The respective wavelet coefficients do not contain the anisotropic smoothness information, they rather resolve a certain
minimal smoothness. That results in a bad decay of the sequence of wavelet coefficients or, in other words, a bad
compression rate.

Anisotropy is not a rare phenomenon since it arises whenever physics does not act the same in different directions,
e.g., geophysics, oceanography, hydrology, fluid mechanics, or medical image processing (see \cite{bonami:estrade:2003, SL87} among others) are some of the fields where it naturally appears. For this reason
wavelets have been adapted in many different ways in order to ``detect'' and resolve anisotropy. There is a vast
amount of literature dealing with this. For instance, there are wave atoms \cite{demanet:ying:2007} as well as
curvelets \cite{CanDo2000,CandDon02newtight,candes:demanet:2005},
shearlets \cite{KLLW05,Guo2006SparseMR,lakhonchai:sampo:sumetkijakan:2012}, anisets, and anisotropic wavelets \cite{triebel:2004,triebel:2006,hochmuth:2002a}. The latter concept represents a rather flexible construction since it can
be build (theoretically) for any present anisotropy. The theoretical basis of anisotropic wavelet analysis is the
equivalent characterization of corresponding anisotropic function spaces, like H\"older, Besov, Sobolev and
Triebel-Lizorkin spaces. The major shortcoming of the existing theory is the fact that one has to know the anisotropy
in advance, i.e., one has to adapt the wavelet accordingly. In other words, if physics does not provide the
anisotropy parameters of the signal we are not able to resolve the signal accordingly without ``trying out'' several
anisotropic bases. Such a method is, of course, hardly implementable in practice.

In Abry et.\ al.\ \cite{ACJRV:2014} it has been shown that any anisotropic Besov space -- defined with respect to the
cartesian axis -- can ``almost'' be characterized with the help of the so-called hyperbolic wavelet transform. The
anisotropy of the signal can then be detected using a uniform basis and is characterized by a special weight in the
wavelet coefficients. This has led to an efficient algorithm for image classification and anisotropy detection applied to both synthetic and real textures (see \cite{roux:clausel:vedel:jaffard:abry:2012,ARWM}).

In this paper we further develop this idea of describing anisotropy with the help of the
hyperbolic wavelet transform. For this reason we introduce a new family of anisotropic function spaces which are defined via a hyperbolic Littlewood-Paley analysis and for which we prove exact characterization with hyperbolic wavelets. The motivation behind this is to provide a general setting of anisotropic spaces characterized by one single basis of wavelets and thus to understand how one such fixed basis can help to describe anisotropic smoothness.

Concretely, we start with a hyperbolic Littlewood-Paley analysis defined as the usual tensor product 
\begin{equation}\label{f12}
   \Delta_{\bj}(f) := \cf^{-1}[\theta_{j_1}\otimes...\otimes \theta_{j_d}\cf f]\quad,\quad \bj=(j_1,\ldots,j_d)\in \N_0^d\,.
\end{equation}
This hyperbolic decomposition of the frequency space has been widely used for the Fourier analytic definition of the
well-known spaces with dominating mixed smoothness, see \cite{SchTr87, Vyb06} and the references therein. These
spaces represent a suitable framework for multivariate appoximation, see \cite{devore:konyagin:temlyakov:1998, Te93} and the recent survey article \cite{DuTeUl2015}. The main ideas have been developed over more than fifty
years
of intense research in the former Soviet Union such that it is  beyond
the scope of this paper to name all the relevant references (cf.\ \cite{DuTeUl2015}).

Based on the
decomposition~\eqref{f12}, 
we then define spaces $\widetilde{A}^{s,\balpha}_{p,q}(\R^d)$ with $A\in \{B,F\}$ of Besov-Lizorkin-Triebel type involving an anisotropy vector $\balpha=(\alpha_1,\ldots,\alpha_d) >0$ with $\sum_{i=1}^d \alpha_i = d$. As a special
case ($A=F$, $1<p<\infty$, $q=2$), these include the Sobolev type spaces
$\widetilde{W}^{s,\balpha}_{p}(\R^d):=\F^{s,\balpha}_{p,2}(\R^d)$, where
$$
      \|f\|_{\widetilde{F}^{s,\balpha}_{p,q}(\R^d)}:=\Big\|\Big(\sum\limits_{\bj\in
\N_0^d}2^{qs\|\bj/\balpha\|_{\infty}}|\Delta_{\bj}(f)(\cdot)|^q\Big)^{1/q}\Big\|_p \quad,\quad f\in\mathcal{S}^{\prime}(\R^d)\,.
$$
It is important to note that the anisotropy hereby only enters in the weight $2^{s\|\bj/\balpha\|_{\infty}}$, where we use the short-hand notation $\|\bj/\balpha\|_{\infty}:=\max\{j_1/\alpha_1,...,j_d/\alpha_d\}$, but not in the
choice of the Littlewood-Paley decomposition. 

One of the main results of this paper is the coincidence (with respect to
equivalent norms)
\begin{equation}\label{coinc}
    \widetilde{W}^{s,\balpha}_p(\R^d) = W^{s,\balpha}_p(\R^d)\,,\quad\text{if $1<p<\infty$}\,,
\end{equation}
where the space on the right-hand side represents the classical anisotropic Sobolev space defined in
\eqref{cls_iso} below. This relation has already been observed for isotropic (i.e. $\balpha = (1,...,1)$)
Hilbert-Sobolev spaces ($p=2$) on the $d$-torus, see \cite{GrKn09,DiUl12}, as well as on $\R^2$ in \cite{ACJRV:2014}. Our result extends this observation to all
$1<p<\infty$. Surprisingly, such a coincidence in the spirit of \eqref{coinc} is only possible in the
Sobolev case. To be more precise, it holds
$$
    \widetilde{A}^{s,\balpha}_{p,q}(\R^d) = A^{s,\balpha}_{p,q}(\R^d) \quad\text{if and only if\: $A=F$, $1<p<\infty$,\, and
    	$q=2$.}
$$

As an important consequence of this equality~\eqref{coinc}, we can further prove that it is possible  
to characterize (e.g.\ detect and classify) classical anisotropies described by the spaces $W^{s,\balpha}_p(\R^d)$ via the wavelet coefficients of a universal hyperbolic wavelet basis. Compared to the classical approach using anisotropic wavelets, this approach has the advantage that one does not need a-priori knowledge on the anisotropies, otherwise required for constructing 
the ``right'' basis.
In particular, our results entail that any sufficiently regular orthonormal basis  $(\psi_{\bj,\bk})_{\bj,\bk}$ of tensorized wavelets $\psi_{\bj,\bk}=\psi_{j_1,k_1} \otimes\cdots\otimes \psi_{j_d,k_d}$ constitutes an unconditional Schauder basis for $W^{s,\balpha}_{p}(\R^d)$, whose coefficients,  
measured in an appropriate corresponding sequence space,  
give rise to an equivalent norm on $W^{s,\balpha}_{p}(\R^d)$, i.e. for $ f\in W^{s,\balpha}_{p}(\R^d)$ with coefficients $\langle f, \psi_{\bj,\bk}\rangle$
$$
\|f\|_{W^{s,\balpha}_{p}(\R^d)} \asymp
\Big\|\Big(\sum\limits_{\bj\in \N_0^d}2^{2s\|\bj/\balpha\|_{\infty}}
\Big|\sum\limits_{\bk\in \Z^d}\langle f, \psi_{\bj,\bk}\rangle \chi_{\bj,\bk}(\cdot)\Big|^2\Big)^{1/2}\Big\|_p 
\,.
$$

This is stated in Theorem~\ref{thm:main_wave1}. A similar result, see Theorem~\ref{thm:haarSobolev}, holds true for the hyperbolic Haar system $\mathcal{H}_d=(h_{\bj,\bk})_{\bj,\bk}$, where $h_{\bj,\bk}=h_{j_1,k_1} \otimes\cdots\otimes h_{j_d,k_d}$, under the following restriction on the parameter $s$ of the space
$W^{s,\balpha}_{p}(\R^d)$, 
\begin{align*}
|s|/\alpha_{\text{min}}<
\min\Big\{\frac{1}{p},1-\frac{1}{p}\Big\}.
\end{align*}
In this direction, we would also like to mention the new and related findings of Oswald in~\cite{Oswald19} on the Schauder basis property of the hyperbolic Haar system 
in the classic isotropic Besov spaces defined via first-order moduli of smoothness.

At the center of our respective proofs, 
we will rely on discrete characterizations provided by hyperbolic wavelets for the spaces $\widetilde{A}^{s,\balpha}_{p,q}(\R^d)$, $A\in\{B,F\}$. These characterizations are fundamental and established in separate theorems, Theorem~\ref{main_wav} and Theorem~\ref{thm:haarmain}, whereby we follow two paths.
On the one hand, we use the usual methodology and consider orthonormal wavelet bases for which we
assume sufficient smoothness, decay, and vanishing moments.
As a byproduct, we thereby significantly extend the wavelet characterizations in \cite{Vyb06, RaUl10} for Besov-Lizorkin-Triebel spaces with dominating mixed smoothness.
On the other hand, we use a hyperbolic Haar system, which does not fulfill smoothness conditions as before but nevertheless allows for characterization in a certain restricted parameter range.

Let us remark that analysis with the Haar wavelet has a long tradition (see e.g.\ \cite{Golubov1972,Romanyuk2014,Romanyuk2016,Romanyuk2016b,Andrianov1999}), the Haar wavelet being the
oldest and simplest orthonormal wavelet, conceived as early as 1909~\cite{Haar1911}. Besides its elegance and simplicity, notably its connection to the Faber system~\cite{Faber1910} and other spline functions, such as e.g.\ the Chui-Wang wavelet~\cite{CWwavelet}, makes it interesting from a numerical perspective. In particular in imaging science it plays an important role in practical applications.
Recently, it has attracted renewed attention with a series of publications
\cite{SeU2015,SeUl2015,Garrigos2017,GaSeUl2019,GaSeUl2019b,SU2019,DerUll19}.

The paper has the following {\bf structure.} After having recalled in Section~2 some helpful Fourier analytic tools (in particular some classical maximal functions and associated inequalities) as well as the definition of the classical (anisotropic) function spaces $A^{s,\balpha}_{p,q}(\R^d)$, where $A\in\{B,F\}$, we introduce in Section~3 the notion of hyperbolic Littlewood-Paley analysis and the related 
Besov-Lizorkin-Triebel spaces $\widetilde{A}^{s,\balpha}_{p,q}(\R^d)$. Wavelet characterizations of these new hyperbolic spaces are the topic of Sections 4 and 5, whereby we 
first resort to standard wavelets with sufficient smoothness, decay, and vanishing moments in Section~4, while
in Section 5 we utilize a hyperbolic Haar basis.
The relationship of the new scale to the traditional spaces
is finally investigated in Sections 6 and 7. Specifically, in Section~6, we show 
the equality $\widetilde{W}^{s,\balpha}_p(\R^d) = W^{s,\balpha}_p(\R^d)$, i.e. $\widetilde{F}^{s,\balpha}_{p,2}(\R^d)=F^{s,\balpha}_{p,2}(\R^d)$, in the range $1<p<\infty$, from which we can 
then extract our main theorems concerning hyperbolic wavelet characterizations of
the classical $W^{s,\balpha}_p(\R^d)$.

Let us agree on the following general {\bf notation.} As usual $\N$ shall denote the natural numbers. We further put $\N_0:=\N\cup\{0\}$, and let
$\Z$ denote the integers,
$\R$ the real numbers,
and $\C$ the complex numbers. By $\mathbb{T}:=\R/2\pi \Z$ we refer to the torus identified with the interval $[0,2\pi]\subset\R$. We write $\langle x,y\rangle$ or $x\cdot y$ for the Euclidean inner product in $\R^d$ or
$\C^d$. The letter $d$ is hereby always reserved for the underlying dimension 
and by $[d]$ we mean the set
$\{1,...,d\}$. For $0<p\leq \infty$ and $x\in \R^d$ we define $\|x\|_p := (\sum_{i=1}^d |x_i|^p)^{1/p}$, with the
usual modification in the case $p=\infty$. If $1 \leq p\leq \infty$ we set $p'$ such that $1/p+1/p' = 1$. For $0<p,q\leq \infty$ we further denote
$\sigma_{p,q}:=\max\{1/p-1,1/q-1,0\}$ and
$\sigma_p:=\max\{1/p-1,0\}$. 
We also put $x_{+}:= ((x_1)_+,...,(x_d)_+)$, whereby $a_+:=\max\{a,0\}$ for $a\in\R$.
Analogously we define $x_-$. 
By $(x_1,\ldots,x_d)>0$ we shall mean that each coordinate is positive.
Finally, as usual, $a\in \R$ is decomposed into $a = \lfloor a
\rfloor + \{a\}$, where $0\leq \{a\} <1$ and $\lfloor a
\rfloor\in\Z$. In case $x\in\R^d$, $\{x\}$ and $\lfloor x
\rfloor$ are then meant component-wise.
Multivariate indices are typesetted with a bar, like e.g.\ $\bk,\bj,\bl$, or $\bm$, to indicate the multi-index.
In all the paper, the multi-index
$\balpha = (\alpha_1, ..., \alpha_d) > 0$ thereby stands for
an anisotropy and is such that $\alpha_1+ ... + \alpha_d = d$. In addition, we here use the abbreviations $\alpha_{\text{min}}:=\min\{\alpha_1,..,\alpha_d\}$ and
$\alpha_{\text{max}}:=\max\{\alpha_1,..,\alpha_d\}$. The notation $\balpha/\bj$ shall always stand for $(\alpha_1/j_1,\ldots,\alpha_d/j_d)$.
Given a positive real $a>0$, we further write $a^{\balpha}$ for the vector $(a^{\alpha_1}, \ldots,a^{\alpha_d})$ and let $f(a^{\balpha}x):=f(a^{\alpha_1}x_1,\ldots,a^{\alpha_d}x_d)$ be the anisotropically scaled version of the function $f:\R^d\to\C$. 
For two (quasi-)normed spaces $X$ and $Y$, the \mbox{(quasi-)norm}
of an element $x\in X$ will be denoted by $\|x\|_X$.
The symbol $X \hookrightarrow Y$ indicates that the
identity operator is continuous. For two sequences $a_n$ and $b_n$ we will write $a_n \lesssim b_n$ if there exists a
constant $c>0$ such that $a_n \leq c\,b_n$ for all $n$. We will write $a_n \asymp b_n$ if $a_n \lesssim b_n$ and $b_n
\lesssim a_n$.

\section{Classical spaces and tools from Fourier analysis}
\label{section2}
\noindent
Let $L_p=L_p(\R^d)$, $0 < p\le\infty$, be the Lebesgue space of all measurable functions
$f:\R^d\to\C$ such that
\[
\|f\|_p := \Big(\int_{\R^d} |f(x)|^p dx \Big)^{1/p} < \infty \,,
\]
with the usual modification if $p=\infty$.
We will also need $L_p$-spaces on compact domains $\Omega\subset\R^d$
instead of $\R^d$ and shall
write $\|f\|_{L_p(\Omega)}$ for the corresponding restricted \mbox{$L_p$-(quasi-)norms}.

For $k\in \N_0$, we denote by $C^{k}_0(\R^d)$ the collection of all compactly supported
functions $\varphi$ on $\R^d$ which have uniformly continuous derivatives $D^{\bar{\gamma}}\varphi$ on $\R^d$
whenever $\| \bar{\gamma}\|_{1} \leq k$.
Additionally, we define the spaces of infinitely differentiable functions $C^{\infty}(\R^d)$ and infinitely differentiable
functions with compact support $C_0^{\infty}(\R^d)$ as well as the Schwartz space $\S=\S(\R^d)$ of
all rapidly decaying infinitely differentiable functions on $\R^d$, i.e.,
\[
\S(\R^d) := \bigl\{\varphi\in C^{\infty}(\R^d)\colon \|\varphi\|_{k,\ell}<\infty
\;\text{ for all } k,\ell\in\N\bigr\}\,,
\]
and
$$
    \|\varphi\|_{k,\ell}:=\Big\|(1+|\cdot|)^k\sum_{\|\bar{\gamma}\|_1\leq
\ell}|D^{\bar{\gamma}}\varphi(\cdot)|\Big\|_{\infty}\quad,\quad k,\ell \in \N\,.
$$

The space $\mathcal{S}'(\R^d)$, the topological dual of $\mathcal{S}(\R^d)$, is also referred to as the space of tempered
distributions on $\R^d$. Indeed, a linear mapping $f:\mathcal{S}(\R^d) \to \C$ belongs
to $\mathcal{S}'(\R^d)$ if and only if there exist numbers $k,\ell \in \N$
and a constant $c = c_f$ such that
$$
    |f(\varphi)| \leq c_f\|\varphi\|_{k,\ell}
$$
for all $\varphi\in \mathcal{S}(\R^d)$. Any locally integrable function $f$ on $\R^d$ belongs to $\mathcal{S}'(\R^d)$
in the sense that
\begin{equation*}
   f(\varphi) := \int_{\R^d} f(x)\varphi(x)\,dx\quad,\quad \varphi\in S(\R^d)\,.
\end{equation*}
The space $\mathcal{S}'(\R^d)$ is equipped with the weak$^{\ast}$-topology.

For $f\in L_1(\R^d)$ we define the Fourier transform
\[
\cf f(\xi)
\,=\, (2\pi)^{-d/2}\int_{\R^d} f(y) e^{- i \xi\cdot y} dy, \qquad \xi\in\R^d,
\]
and the corresponding inverse Fourier transform $\cf^{-1}f(\xi)=\cf f(-\xi)$.
As usual, the Fourier transform can be extended to $\mathcal{S}'(\R^d)$
by $(\cf f)(\varphi) := f(\cf \varphi)$, where
$\,f\in \mathcal{S}'(\R^d)$ and $\varphi \in \mathcal{S}(\R^d)$.
The mapping $\cf:\S'(\R^d) \to \S'(\R^d)$ is a bijection. 

The convolution $\varphi\ast \psi$ of two
square-integrable
functions $\varphi, \psi$ is defined via the integral
\begin{equation*}
    (\varphi \ast \psi)(x) = \int_{\R^d} \varphi(x-y)\psi(y)\,dy\,.
\end{equation*}
If $\varphi,\psi \in \mathcal{S}(\R^d)$ then $\varphi \ast \psi$ still belongs to
$\mathcal{S}(\R^d)$.
In fact, we have $\varphi \ast \psi\in\S(\R^d)$ even if $\varphi \in \mathcal{S}(\R^d)$
and $\psi\in L_1(\R^d)$. The convolution can be extended to $\mathcal{S}(\R^d)\times \mathcal{S}'(\R^d)$ via
$(\varphi\ast \psi)(x) = \psi(\varphi(x-\cdot))$, which makes sense pointwise and is
a $C^{\infty}$-function on $\R^d$.

\subsection{Classical (an)isotropic Littlewood-Paley analysis}

Subsequently,
$\balpha = (\alpha_1, ..., \alpha_d) > 0$ will denote
an anisotropy and be such that $\alpha_1+ ... + \alpha_d = d$. 
Anisotropic Besov spaces may then be introduced using an anisotropic 
Littlewood-Paley analysis depending on $\balpha$.
Classical isotropic spaces -- as a particular case of anisotropic spaces -- will thereby be obtained for $\balpha = (1,1,...,1)$.

Let $\varphi_0^{\balpha}  \ge 0$ belong to the Schwartz class ${\mathcal{S}}(\mathbb{R}^d)$ and be such that, for
$\xi=(\xi_1,...,\xi_d) \in \mathbb{R}^d$,
$$
\varphi_0^{\balpha}(\xi) = 1 \quad \text{if} \quad \sup_{i=1,2, ..., d} \vert \xi_{i} \vert \le 1\;,
$$
$$
\text{and} \,   \, \, \varphi_0^{\balpha}(\xi)=0 \quad \text{if}  \quad \sup_{i=1,...,d} \vert 2^{-\alpha_{i}} \xi_i \vert
\ge 1\;.
$$
For $j\in\mathbb{N}$, further define
\begin{align*}
\varphi_j^{\balpha}(\xi) &:= \varphi_0^{\balpha}(2^{-j \balpha}\xi)-\varphi_0^{\balpha}(2^{-(j-1)\balpha}\xi) \\
&= \varphi_0^{\balpha}(2^{-j \alpha_1}\xi_1, \ldots, 2^{-j \alpha_d}\xi_d)-\varphi_0^{\balpha}(2^{-(j-1)\alpha_1}\xi_1, \ldots, 2^{-(j-1)\alpha_d}\xi_d)\,.
\end{align*}
Then
$\sum_{j=0}^{\infty} \varphi_j^{\balpha}\equiv 1$,
and $(\varphi_j^{\balpha})_{j\geq 0}$ is called an {\it anisotropic resolution of unity}. It satisfies
\begin{equation*}
\mathrm{supp}(\varphi_0^{\balpha}) \subset R_1^{\balpha}, \quad \mathrm{supp}(\varphi_j^{\balpha})
\subset R_{j+1}^{\balpha} \setminus R_{j-1}^{\balpha}\;,
\end{equation*}
\begin{align*}
\text{ where} \, \, \, R_j^{\balpha} = \Big\lbrace \xi=(\xi_1, \ldots, \xi_d) \in \mathbb{R}^d\,:\;\vert
\xi_{i} \vert \le 2^{\alpha_i j}
\;\text{for } i\in [d]=\{1,\ldots,d\}  \Big\rbrace \;.
\end{align*}

For $f \in \mathcal{S}'(\mathbb{R}^d)$, we then define
$$
\Delta^ {\balpha}_j f = {\mathcal{F}}^{-1} ( \varphi^{\balpha}_j {\mathcal{F}}f )\;.
$$
The sequence $(\Delta^ {\balpha}_j f )_{j \ge 0}$ is called an {\it anisotropic Littlewood-Paley analysis} of $f$. With this tool,
the anisotropic Besov spaces are now defined as follows (see~\cite{bownik:ho:2005,bownik:2005}).

\begin{definition}
For $0<p \le \infty$, $0<q\le \infty$, $s\in \mathbb{R}$, the Besov space $B^{s,\balpha}_{p,q}(\mathbb{R}^d)$  is
defined by
$$
B^{s,\balpha}_{p,q}(\mathbb{R}^d) = \Big\{ f \in \mathcal{S}'(\mathbb{R}^d)~:~ \, \Big( \sum_{j \ge 0} 2^{jsq} \Vert
\Delta^ {\balpha}_j f \Vert_p^q \Big)^{1/q} <\infty \Big\}\,,
$$
with the usual modification for $q=\infty$.

This definition does not depend on chosen resolution of unity $\varphi_0^{\balpha}$ and the quantity
$$
\Vert f \Vert_{B^{s,\balpha}_{p,q}} = \Big( \sum_{j \ge 0}  2^{jsq} \Vert \Delta^ {\balpha}_jf \Vert_p^q
\Big)^{1/q}
$$
is a norm (resp. quasi-norm) on $B^{s,\balpha}_{p,q}(\mathbb{R}^d)$ for $1 \leq p, \, q \leq \infty$ (resp. $0< \min
\{p,q\} <1$) and with the usual modification if $q=\infty$.
\end{definition}

As in the isotropic case, anisotropic Besov spaces encompass a large class of classical anisotropic function spaces
(see~\cite{triebel:2006} for details).
For example, when $p=q=2$, the Besov spaces coincide with the anisotropic Sobolev spaces
and, when $p=q=\infty$, the  spaces $B^{s,\balpha}_{\infty,\infty}(\mathbb{R}^d)$ are called anisotropic H\"{o}lder
spaces and are denoted by $\mathcal{C}^{s,\balpha}(\mathbb{R}^{d})$.

\begin{definition}
For $0<p \le \infty$, $0<q\le \infty$, $s\in \mathbb{R}$, the Triebel-Lizorkin space
$F^{s,\balpha}_{p,q}(\mathbb{R}^d)$ is defined by
$$
F^{s,\balpha}_{p,q}(\mathbb{R}^d) = \Big\{ f \in \mathcal{S}'(\mathbb{R}^d)~:~ \, \Big\Vert \Big( \sum_{j \ge 0}
2^{jsq} \vert \Delta^ {\balpha}_j f(\cdot) \vert^q \Big)^{1/q} \Big\Vert_p <\infty \Big\}\,,
$$
with the usual modification for $q=\infty$.

This definition does not depend on the chosen resolution of unity $\varphi_0^{\balpha}$ and the quantity
$$
\Vert f \Vert_{F^{s,\balpha}_{p,q}} = \Big\Vert \Big( \sum_{j \ge 0}  2^{jsq} \vert \Delta^ {\balpha}_j f(\cdot) \vert^q
\Big)^{1/q} \Big\Vert_p
$$
is a norm (resp. quasi-norm) on $F^{s,\balpha}_{p,q}(\mathbb{R}^d)$ for $1 \leq p <\infty$ and $1 \le q \leq \infty$
(resp. $0< \min\{p,q\} <1$) and with the usual modification if $q=\infty$.
\end{definition}

If $q=2$ and $1<p<\infty$, the anisotropic Triebel-Lizorkin space coincides with the anisotropic Sobolev space denoted
by $W^{s,\balpha}_p(\mathbb{R}^d)$ :
\begin{equation}\label{cls_iso}
W^{s,\balpha}_{p} = \Big\{ f \in \mathcal{S}'(\mathbb{R}^d)~:~ \, \Big\Vert {\mathcal{F}}^{-1} \Big[\Big(\sum_{i=1}^d(1 +\xi_i^2 )^{1/2\alpha_i}\Big)^{s} {\mathcal{F}}f(\xi)\Big] \Big\Vert_p<\infty \Big\}\,.
\end{equation}

\begin{remark} {\em (i)} As mentioned before, if $\balpha= (1,...,1)$, it is easy to check that the spaces
$B^{s,\balpha}_{p,q}(\R^d)$
(resp. $F^{s,\balpha}_{p,q}(\R^d)$) coincide with the classical spaces $B^{s}_{p,q}(\R^d)$ (resp.
$F^s_{p,q}(\mathbb{R}^d)$). In addition, we have $F^{0,\balpha}_{p,2}(\R^d)= L_p(\R^d)$ in the range $1<p<\infty$.

{\em (ii)} Our understanding of anisotropic spaces coincides with the one in Triebel \cite{triebel:2006} (see also the
references therein). There are different (but related) notions of anisotropy in the Russian literature, see
Nikolskij \cite[Chapt.\ 4]{Nik75} or Temlyakov \cite[II.3]{Te93}. A consequence of our Theorem \ref{comp} below is the
fact that in case of $W$-spaces the mentioned approaches coincide and lead to the same notion of anisotropy.
However, in case of H\"older-Nikolskij spaces this is in general not the case as for instance Theorem \ref{neg_comp}
shows.
\end{remark}

\subsection{Maximal inequalities}

Let us provide here the maximal inequalities for the Hardy-Littlewood and Peetre maximal functions, respectively. For
further details we refer to \cite[1.2, 1.3]{Vyb06} or \cite[Chapt.\ 2]{SchTr87}\,.

For a locally integrable function $f:\R^d\to \C$ we denote by $Mf(x)$ the Hardy-Littlewood maximal function defined by
\begin{equation}
  (Mf)(x) = \sup\limits_{x\in Q} \frac{1}{|Q|}\int_{Q}\,|f(y)|\,dy\quad,\quad x\in\R^d \label{maxfunc}
  \,,
\end{equation}
where the supremum is taken over all cubes with sides parallel to the coordinate axes containing~$x$.
A vector valued generalization of the classical Hardy-Littlewood maximal inequality is due to
Fefferman and Stein \cite{FeSt71}.
\begin{theorem}[\cite{FeSt71}]\label{feffstein}\rm For $1<p<\infty$ and $1 < q \leq \infty$ there exists a constant $c>0$, such
that
  \begin{equation*}
      \Big\|\Big(\sum\limits_{\ell \in I} |M f_{\ell}|^q\Big)^{1/q}\Big\|_p \leq c
      \Big\|\Big(\sum\limits_{\ell \in I} |f_{\ell}|^q\Big)^{1/q}\Big\|_p
  \end{equation*}
  holds for all sequences $\{f_\ell\}_{\ell\in I}$ of locally Lebesgue-integrable functions on $\R^d$.
\end{theorem}
\noindent We require a direction-wise version of (\ref{maxfunc})
\begin{equation*}
  (M_i f)(x) = \sup\limits_{s>0}\frac{1}{2s}\int_{x_i-s}^{x_i+s} |f(x_1,...,x_{i-1},t,x_{i+1},...,x_d)|\,dt
  \quad,\quad x\in\R^d.
\end{equation*}
We denote the composition of these operators by  $M_e=\prod_{i\in e}M_i $, where $e$ is a subset of $[d]=\{1,\ldots,d\}$ and
$M_{\ell}M_k$ has to be interpreted as $M_{\ell}\circ M_k$. The following version of the Fefferman-Stein maximal
inequality is due to St\"ockert \cite{Stoe}.

\begin{theorem}[\cite{Stoe}]\label{bagby}\rm For $1<p<\infty$ and $1 <q \leq \infty$ there exists a constant $c>0$ such that
for
any $i\in [d]$
  \begin{equation*}
      \Big\|\Big(\sum\limits_{\ell \in I} |M_i f_{\ell}|^q\Big)^{1/q}\Big\|_p \leq c
      \Big\|\Big(\sum\limits_{\ell \in I} |f_{\ell}|^q\Big)^{1/q}\Big\|_p
  \end{equation*}
  holds for all sequences $\{f_\ell\}_{\ell\in I}$ of locally Lebesgue-integrable functions on $\R^d$.
\end{theorem}
Iteration of this theorem yields a similar boundedness property for the operator $M_{[d]}$.\\

\noindent The following construction of a maximal function is due to Peetre, Fefferman, and Stein.
Let $\bar{b}=(b_1,...,b_d)>0$, $a>0$, and $f \in L_1(\R^d)$ with $\cf f$ compactly
supported. We define the Peetre maximal function $P_{\bar{b},a}f$ by
\begin{equation}
  P_{\bar{b},a}f(x) = \sup\limits_{z\in \R^d} \frac{|f(x-z)|}{(1+|b_1z_1|)^a\cdot...\cdot (1+|b_dz_d|)^a}
  \quad.\label{petfefste}
\end{equation}

\begin{lemma}\label{maxunglem} Let $\Omega \subset \R^d$ be a compact set. Let further $a>0$
and $\bgamma = (\gamma_1,...,\gamma_d) \in \N_0^d$. Then there exist two constants
$c_1,c_2>0$ (independent of $f$) such that
  \begin{equation}
    \begin{split}
      P_{(1,...,1),a}(D^{\bgamma}f)(x) &\leq c_1 P_{(1,...,1),a} f(x)\\
      &\leq c_2\big(M_d\big(M_{d-1}\big(...\big(M_1|f|^{1/a}\big)...\big)
      \big)\big)^{a}(x)\quad
    \end{split}
    \label{g1}
  \end{equation}
  holds for all $f\in L_1(\R^d)$ with $\supp(\cf f) \subset \Omega$ and all $x\in\R^d$.
	The constants $c_1$, $c_2$ depend \mbox{on $\Omega$.}
\end{lemma}
We finally give a vector-valued version of the Peetre maximal inequality which is a direct consequence of Lemma \ref{maxunglem} together
with Theorem \ref{bagby}.
\begin{theorem}\label{peetremax} Let $0<p<\infty$, $0<q\leq \infty$ and $a>\max\{1/p,1/q\}$. Let further
            $\bar{b}^{\ell} = (b^{\ell}_1,...,b^{\ell}_d)>0$ for $\ell \in I$ and $\Omega = \{\Omega_{\ell}\}_{\ell\in I}$,
            such that
            $$
                       \Omega_{\ell} \subset [-b_1^{\ell},b_1^{\ell}]\times\cdots \times [-b_d^{\ell},b_d^{\ell}]
            $$
            is compact for $\ell\in I$. Then there is a constant $C>0$ (independent of $f$ and $\Omega$) such that
            $$
               \Big\|\Big(\sum\limits_{\ell \in I} |P_{\bar{b}^{\ell},a}f_{\ell}|^q\Big)^{1/q}\Big\|_p \leq
C
	       \Big\|\Big(\sum\limits_{\ell \in I} |f_{\ell}|^q\Big)^{1/q}\Big\|_p
            $$
            holds for all systems $f = \{f_{\ell}\}_{\ell\in I}$ with $\supp(\cf f_{\ell}) \subset \Omega_{\ell}$, $\ell
\in I$\,.
 \end{theorem}

\section{Hyperbolic Littlewood-Paley analysis}
\label{section3}

Let $\theta_0 \in \mathcal{S}(\mathbb{R})$ be supported on $[-2,2]$ with $\theta_0 =1$ on $[-1,1]$. For any $j \in
\mathbb{N}$, let us further define
$$
\theta_j = \theta_0(2^{-j} \cdot) - \theta_0(2^{-(j-1)} \cdot)
$$
such that $(\theta_{j})_{j}$ is a univariate resolution of unity, i.e.,
$\sum_{j \ge 0} \theta_j(\cdot) =1$. Observe that, for any $j \ge 1$,
$\mathrm{supp}(\theta_j) \subset \{ 2^{j-1} \le \vert \xi \vert \le 2^{j+1} \}$.

\begin{remark}
In the following, the function $\theta_0$ can be chosen with an arbitrary compact support. It does not change the main
results even if technical details of proofs and lemmas have to be adapted. This allows to choose $\theta_0$ as the Fourier transform of a
Meyer scaling function.
\end{remark}

Let us now come to the main concept of this paper, the hyperbolic Littlewood-Paley analysis.

\begin{definition}\label{def:hypLP} {\em (i)} For any $\bar{j} = (j_1, ..., j_d) \in \N_0^d$ and any
$\xi=(\xi_1,...,\xi_d) \in \mathbb{R}^d$ set
$$
\theta_{\bj}(\xi_1, ..., \xi_d) := \theta_{j_1}(\xi_1) \theta_{j_2}(\xi_2) ... \theta_{j_d}(\xi_d)\;.
$$
The function $\theta_{\bj}$ belongs to $\mathcal{S}(\mathbb{R}^d)$ for all $\bj \in \N_0^d$ and is compactly supported
on a dyadic rectangle. Further $\sum_{\bj \in \N_0^d}  \theta_{\bj} \equiv 1$ and
$(\theta_{\bj})_{\bj}$ is called a hyperbolic resolution of unity.\\
{\em (ii)} For $f \in \mathcal{S}'(\mathbb{R}^d)$ and $\bj \in \N_0^d$ set
$$
\Delta_{\bj} f :=  {\mathcal{F}}^{-1} (\theta_{\bj} {\mathcal{F}}f )\;.
$$
The sequence $(\Delta_{\bj} f)_{\bj\in \N_0^d}$ is called a hyperbolic Littlewood-Paley analysis of $f$.
\end{definition}

We are now in the position to introduce new functional spaces called {\it anisotropic hyperbolic Besov spaces} and {\it
anisotropic hyperbolic Triebel-Lizorkin spaces} defined {\it via} the hyperbolic Littlewood-Paley analysis.

\begin{definition}\label{defB} For $0<p \le \infty$, $0<q\le \infty$, $s\in \mathbb{R}$, and $\balpha=(\alpha_1,\ldots,\alpha_d)>0$ such that
$\sum_{i=1}^d
\alpha_i = d$ we define the anisotropic hyperbolic Besov space $\widetilde{B}^{s,\balpha}_{p,q}(\mathbb{R}^d)$ via
$$
\B^{s,\balpha}_{p,q}(\mathbb{R}^d) = \Big\{ f \in \mathcal{S}'(\mathbb{R}^d)~:~\Big( \sum_{\bj \in \N_0^d}
2^{\|(j_1/\alpha_1, ..., j_d/\alpha_d)\|_{\infty}sq} \Vert \Delta_{\bj} f \Vert_p^q \Big)^{1/q} < \infty
\Big\}\,,
$$
with the usual modification in case $q=\infty$. 

This definition does not depend on the chosen resolution of unity
$(\theta_{\bj})_{\bj}$ and the quantity
$$
\|f\|_{\B^{s,\balpha}_{p,q}}:= \Big( \sum_{\bj \in \N_0^d}  2^{\|(j_1/\alpha_1, ..., j_d/\alpha_d)\|_{\infty}sq}
\Vert \Delta_{\bj} f \Vert_p^q \Big)^{1/q}
$$
is a norm (resp. quasi-norm) on $\B^{s,\balpha}_{p,q}(\mathbb{R}^d)$ for $1 \leq p, \, q \leq \infty$
(resp. $0< \min\{p,q\} <1$) and with usual modification if $q=\infty$.
\end{definition}

\begin{definition}\label{defHTLS} For $0<p<\infty$, $0<q \leq \infty$, $s\in \R$, $\balpha=(\alpha_1,\ldots,\alpha_d)>0$ such that $\sum_i \alpha_i
= d$ we define the anisotropic hyperbolic Triebel-Lizorkin space via
$$
\F^{s,\balpha}_{p,q}(\mathbb{R}^d) = \Big\{f \in \mathcal{S}'(\mathbb{R}^d)~:~ \Big\| \Big(
\sum_{\bj \in \N_0^d}  2^{\|(j_1/\alpha_1,..., j_d/\alpha_d)\|_{\infty} sq} \vert
\Delta_{\bj} f(\cdot) \vert^q \Big)^{1/q} \Big\|_p <\infty \Big\}\,,
$$
with the usual modification in case $q=\infty$. 

This definition does not depend on the chosen resolution of unity
$(\theta_{\bj})_{\bj}$ and the quantity
$$
\|f\|_{\F^{s,\balpha}_{p,q}} := \Big\| \Big(
\sum_{\bj \in \N_0^d}  2^{\|(j_1/\alpha_1,..., j_d/\alpha_d)\|_{\infty} sq} \vert
\Delta_{\bj} f(\cdot) \vert^q \Big)^{1/q} \Big\|_p
$$
is a norm on $\F^{s,\balpha}_{p,q}(\mathbb{R}^d)$ for $1 \leq p <\infty$, $1 \le q \leq \infty$
(resp.\ quasi-norm for $0< \min \{p,q\} <1$).
\end{definition}

\begin{remark} The above definitions of anisotropic hyperbolic Besov and Sobolev spaces
include four indices: $s$ stands for the regularity, $p$ is the integration parameter and $q$ the
so-called fine-index. The parameter $\balpha=(\alpha_1,\ldots,\alpha_d)$ encodes the present anisotropy: the more $\alpha_{\text{min}}=\min\{\alpha_1,..,\alpha_d\}$
is close to $0$ and $\alpha_{\text{max}}=\max\{\alpha_1,...,\alpha_d\}$ is close to $d$, the more we need directional smoothness in one axis
compared to others. On the other hand, if $\balpha=(1,...,1)$ the anisotropy becomes an ``isotropy''.
\end{remark}

\begin{remark} By analogy with the classical spaces, if $q=2$ and $1<p<\infty$, $\F^{s,\balpha}_{p,q}(\R^d)$ is
called {\it anisotropic hyperbolic Sobolev space} and is denoted by $\widetilde{W}^{s,\balpha}_p({\mathbb{R}}^d)$.
In case $\balpha = (1,...,1)$ we write $\widetilde{W}^{s}_p({\mathbb{R}}^d)$.
\end{remark}

Let us finally introduce classical spaces with dominating mixed smoothness in the spirit of \cite{SchTr87,Vyb06}.

\begin{definition}\label{def_mixed} Let $r \in \R$, $0<p,q\leq \infty$ ($p<\infty$ in the $F$-case).\\
{\em (i)} The Besov space with dominating mixed smoothness $S^r_{p,q}B(\R^d)$ is the collection of all distributions
$f\in \mathcal{S}'(\R^d)$ such that the following (quasi-)norm
$$
    \|f\|_{S^r_{p,q}B(\R^d)}:=\Big(\sum\limits_{\bj\in \N_0^d} 2^{r\|j\|_1q}\|\Delta_{\bj}(f)\|_p^q\Big)^{1/q} \quad\text{is finite.}
$$
{\em (ii)} The Triebel-Lizorkin space with dominating mixed smoothness $S^r_{p,q}F(\R^d)$ is the collection of all
distributions $f\in \mathcal{S}'(\R^d)$ such that the following (quasi-)norm
$$
    \|f\|_{S^r_{p,q}F(\R^d)}:=\Big\|\Big(\sum\limits_{\bj\in \N_0^d}
2^{r\|j\|_1q}|\Delta_{\bj}(f)(x)|^q\Big)^{1/q}\Big\|_p \quad\text{is finite.}
$$
{\em (iii)} If $1<p<\infty$ and $r\in \R$ then the Sobolev space with dominating mixed smoothness $S^r_{p}W(\R^d)$ is
the collection of all $f\in \mathcal{S}'(\R^d)$ such that
$$
    \|f\|_{S^r_{p}W} :=\Big\|\cf^{-1}\Big[\prod\limits_{i=1}^d (1+|\xi_i|^2)^{r/2}\cf f\Big]\Big\|_p \quad\text{is finite.}
$$
\end{definition}

Let us also state a useful Fourier multiplier theorem, see~\cite[Thm.\ 1.12]{Vyb06} or~\cite[p.\ 77]{SchTr87}.

\begin{lemma}[\cite{Vyb06,SchTr87}] \label{Fouriermultiplier} Let $0<p<\infty$, $0 < q \le \infty$, and $r > \frac{1}{\min(p,q)} +
\frac{1}{2}$. Further, let $\{ \Omega_{\bar{k}} \}_{\bar{k} \in {\mathbb{N}}_0 ^d}$ be a sequence of compact subsets of
${\mathbb{R}}^d$ such that
$$
\Omega_{\bk} \subset \big\{ x \in {\mathbb{R}}^d \,:\, \vert x_i \vert \le 2^{k_i}, \, i = 1,\ldots,d \big\}.
$$
Then, there is a positive constant $C>0$ such that
$$
\Big\| \Big( \sum_{\bar{k} \in {\mathbb{N}}_0 ^d} \Big\vert {\mathcal{F}}^{-1} [\rho_{\bar{k}} {\mathcal{F}}
f_{\bk} ](\cdot) \Big\vert^q \Big)^{\frac{1}{q}} \Big\|_p \le C \Big\Vert \Big( \sum_{\bar{k} \in
{\mathbb{N}}_0 ^d} \vert f_{\bk} (\cdot) \vert^q \Big)^{\frac{1}{q}} \Big\Vert_p \times \sup_{\bl \in {\mathbb{N}}_0^d}
\Vert \rho_{\bl} (2^{\ell_1} \cdot, \ldots, 2^{\ell_d} \cdot )\Vert_{S^{r}_2 W}
$$
holds for all systems $\{f_{\bar{k}}\}_{\bar{k} \in {\mathbb{N}}_0^d} \in L_p(\ell_q)$ with $\supp(\mathcal{F}f_{\bar{k}}) \subset \Omega_{\bar{k}}$ and all systems $\{\rho_{\bar{k}} \}_{\bar{k} \in {\mathbb{N}}_0^d} \subset
S^{r}_2W(\mathbb{R}^d)$.
\end{lemma}


\section{Hyperbolic wavelet analysis}

In this section we prove hyperbolic wavelet characterizations of the spaces $\B^{s,\balpha}_{p,q}(\mathbb{R}^d)$ and $\F^{s,\balpha}_{p,q}(\mathbb{R}^d)$ defined in Definitions \ref{defB} and
\ref{defHTLS}, respectively. It should be noted that the proof technique used for Theorem \ref{main_wav} below also represents a progress towards new optimal
wavelet characterizations of Besov-Lizorkin-Triebel spaces with dominating mixed smoothness, which extends the results
in \cite[Sect.\ 2.4]{Vyb06} significantly, see Remark \ref{dom_mixed} below.

Let us start with univariate wavelets given
by a scaling function $\psi_0$ and a corresponding wavelet
$\psi$. These functions are supposed to satisfy the following (minimal) conditions:
\begin{itemize}
 \item[(K)] It holds $\psi_0,\psi \in C^K(\R)$. For any $M\in \N$ there is a constant $C_M>0$ such that  for all
$0\leq \alpha \leq K$ it holds
 $$
    |D^{\alpha}\psi_0(x)| + |D^{\alpha}\psi(x)| \leq C_M(1+|x|)^{-M}\quad,\quad x\in \R\,.
 $$
 \item[(L)] The wavelet $\psi$ has vanishing moments up to order $L-1$: For $L>\beta\in \N_0$ it holds
 $$
    \int_{\R} \psi(x)x^{\beta}\,dx = 0\,.
 $$
 In case $L=0$ the condition is void.
\end{itemize}
We shall denote
$$
    \psi_{j,k}:=\frac{1}{\sqrt{2}}\psi(2^{j-1}\cdot-k)\quad,\, j\in \N,\, k\in \mathbb{Z}\,,
$$
and $\psi_{0,k}:=\psi_0(\cdot-k)$. We set $\psi_{j,k} \equiv 0$ if $j<0$. To obtain the hyperbolic wavelet basis
in $L_2(\R^d)$ we tensorize over all scales and obtain
$$
    \psi_{\bj,\bk}(x_1,...,x_d):=\psi_{j_1,k_1}(x_1)\cdot...\cdot \psi_{j_d,k_d}(x_d)\quad,\quad
x=(x_1,...,x_d)\in\R^d,\, \bj\in\Z^d,\, \bk\in \Z^d\,.
$$
The following lemma recalls a useful convolution relation. Let us clarify the notation first. For a given
univariate function $\Lambda$ we will use the notation $\Lambda_j(\cdot):=2^{j-1}\Lambda(2^{j-1}\cdot)$, $j\in \N$\,.
We will further put $x_{j,m}:=2^{-j}m$ and $I_{j,m} := [2^{-j}m, 2^{-j}(m+1))$ with associated characteristic function $\chi_{j,m}:=\bbone_{I_{j,m}}$.

\begin{lemma}\label{conv} Let $\Lambda_0, \Lambda \in \mathcal{S}(\R)$ with $\Lambda$ having infinitely many vanishing
moments, i.e.,
$$
    \int_{\R} \Lambda(x)x^{\beta}\,dx = 0
$$
for all $\beta \in \N$\,. Let further $\psi_0$ and $\psi$ satisfy $(K)$ and $(L)$ as above and $R>0$ be a given real
number. Then it exists a constant $C_R>0$ such that for any $j\in \N_0$ and $\ell,m\in \Z$ the convolution
relation
$$
    |(\Lambda_j \ast \psi_{j+\ell,m})(x)| \leq C_R
2^{-N_{\sign(\ell)}|\ell|}(1+2^{\min\{j,j+\ell\}}|x-x_{j+\ell,m}|)^{-R}
$$
holds true with $\sign(\ell)\in\{+,-,0\}$ and $N_0:=0$, $N_+:= L+1$ and $N_-:=K$.
\end{lemma}
\begin{proof} The above lemma is a special case of a more general convolution relation, see for instance \cite[p.~466]{Gr08} for the most general version. Originally, this relation is due to Frazier, Jawerth \cite[Lem.\
3.3]{FrJa86}, \cite[Lem.\ B.1, B.2]{FrJa90}.
 \end{proof}

Lemma \ref{conv} immediately implies the following multivariate version by exploiting the tensor product
structure. Similar as for the hyperbolic wavelet system, we use the notation
$$
    \Lambda_{\bj}(x):=\Lambda_{j_1}(x_1)\cdot ... \cdot \Lambda_{j_d}(x_d)\quad,\quad x\in \R^d,\, \bj \in \Z^d\,.
$$
In the sequel we will further need the notation 
\begin{align}\label{dyQubes}
Q_{\bj,\bm} := I_{j_1,m_1} \times \ldots \times
I_{j_d,m_d} \quad\text{and}\quad \chi_{\bj,\bm}(x_1,...,x_d):=\chi_{j_1,m_1}(x_1) \cdot ... \cdot \chi_{j_d,m_d}(x_d)\,,
\end{align}
with the notation $I_{j_i,m_i}$ and $\chi_{j_i,m_i}$, $i\in[d]=\{1,\ldots,d\}$, introduced right before
Lemma~\ref{conv}.

\begin{lemma}\label{conv2} Let $\Lambda, \Lambda_0, \psi_0, \psi$ as in Lemma \ref{conv}. For any $R>0$ there exists a
contant $C_R>0$ such that for any $\bj\in \N_0^d$ and $\bl,\bm\in\Z^d$ the convolution
relation
\begin{equation}\label{conv3}
   |(\Lambda_{\bj}\ast \psi_{\bj+\bl,\bm})(x)| \leq C_R \prod\limits_{i=1}^d
2^{-N_{\sign(\ell_i)}|\ell_i|}(1+2^{\min\{j_i,j_i+\ell_i\}}|x_i-2^{-(j_i+\ell_i)}m_i|)^{-R}
\end{equation}
holds true with $\sign(\ell_i)\in\{+,-,0\}$ and $N_0:=0$, $N_+:=L+1$ and $N_-:=K$\,.
\end{lemma}

The next proposition is also crucial and represents the ``hyperbolic version''  of \cite[Lem.\ 3,7]{Ke10}. An isotropic version is
originally due to Kyriazis \cite[Lem.\ 7.1]{Ky03}. For the convenience of the reader we give a proof.

\begin{proposition}\label{HL} Let $0<r\leq 1$ and $R>1/r$. For any sequence $(\lambda_{\bj})_{\bj\in \N_0^d}$ of complex
numbers and any $\bl\in \Z^d, \bj \in \N_0^d$ we have, using the notation $\bl_+=((\ell_1)_{+},\ldots,(\ell_d)_{+})$,
\begin{equation}\nonumber
  \begin{split}
    &\sum\limits_{\bm \in \Z^d} |\lambda_{\bj+\bl,\bm}|\prod\limits_{i=1}^d
    \Big(1+2^{\min\{j_i,j_i+\ell_i\}}|x_i-2^{-(j_i+\ell_i)}m_i|\Big)^{-R} \\
    &~~~~~\lesssim 2^{\|\bl_+\|_1/r}\Big[M\Big|\sum_{\bm\in
    \Z^d}\lambda_{\bj+\bl,\bm}\chi_{\bj+\bl,\bm}\Big|^r\Big]^{1/r}(x)\quad,\, x \in \R^d\,,
  \end{split}
\end{equation}
where $M$ stands for the Hardy-Littlewood maximal operator.
\end{proposition}

\begin{proof} We follow the proof from~\cite[Lem.\ 7]{Ke10}. Put $\delta = R-1/r>0$
and
define a decomposition $\{\Omega_{\bk}(x)\}_{\bk\in \N_0^d}$ of $\Z^d$ depending on $x=(x_1,\ldots,x_d)\in\R^d$ as follows:
$$
  \Omega_{\bk}(x):=\Omega_{k_1}(x_1)\times...\times \Omega_{k_d}(x_d)\quad,\quad \bk=(k_1,\ldots,k_d) \in \N_0^d\,,
$$
with
\begin{equation}\nonumber
  \begin{split}
\Omega_{k_i}(x_i)&:=\{m\in \Z~:~2^{k_i-1}<2^{\min\{j_i,j_i+\ell_i\}}|x_i-2^{-(j_i+\ell_i)}m| \leq
2^{k_i}\}\quad,\quad k_i\in \N\,,\\
\Omega_0(x_i)&:=\{m\in \Z~:~2^{\min\{j_i,j_i+\ell_i\}}|x_i-2^{-(j_i+\ell_i)}m| \leq
1\}\,.
\end{split}
\end{equation}
We then estimate for fixed $x=(x_1,\ldots,x_d)\in\R^d$
\begin{equation}\label{est1}
  \begin{split}
       &\sum\limits_{\bm \in \Z^d} |\lambda_{\bj+\bl,\bm}|\prod\limits_{i=1}^d
    (1+2^{\min\{j_i,j_i+\ell_i\}}|x_i-2^{-(j_i+\ell_i)}m_i|)^{-R}\\
    &~~~~~=\sum\limits_{\bk \in \N_0^d}^{\infty}\sum\limits_{\bm \in \Omega_{\bk}(x)}
|\lambda_{\bj+\bl,\bm}|\prod\limits_{i=1}^d
    (1+2^{\min\{j_i,j_i+\ell_i\}}|x_i-2^{-(j_i+\ell_i)}m_i|)^{-R}\\
    &~~~~~\lesssim\sum\limits_{\bk\in \N_0^d}^{\infty}\sum\limits_{\bm \in \Omega_{\bk}(x)}
|\lambda_{\bj+\bl,\bm}|2^{-\delta\|\bk\|_1-\|\bk\|_1/r} \lesssim \sup\limits_{\bk\in \N_0^d} \Big(\sum\limits_{\bm \in
\Omega_{\bk}(x)}|\lambda_{\bj+\bl,\bm}|\Big)2^{-\|\bk\|_1/r}\\
&~~~~~\lesssim \Big(\sup\limits_{\bk\in \N_0^d}2^{-\|\bk\|_1}\sum\limits_{\bm \in
\Omega_{\bk}(x)}|\lambda_{\bj+\bl,\bm}|^r\Big)^{1/r}\,.
  \end{split}
\end{equation}
We further note that
\begin{equation}\label{est2}
  2^{-\|\bk\|_1}\sum\limits_{\bm \in
\Omega_{\bk}(x)}|\lambda_{\bj+\bl,\bm}|^r = 2^{-\|k\|_1}\int\limits_{\bigcup\limits_{\bm \in
\Omega_{\bk}(x)}Q_{\bj+\bl,\bm}} 2^{\|\bj+\bl\|_1} \sum\limits_{\bw \in \Omega_{\bk}(x)}
|\lambda_{\bj+\bl,\bw}|^r\chi_{\bj+\bl,\bw}(y)\, dy
\end{equation}
and observe that for $Q(x):=\bigcup\limits_{\bm \in
    \Omega_{\bk}(x)}Q_{\bj+\bl,\bm}$ we have $x\in Q(x)$ and
$$
    |Q(x)| \asymp 2^{\|\bk\|_1-\|\min\{\bj,\bj+\bl\}\|_1}.
$$
Recalling the definition of the Hardy-Littlewood maximal function in \eqref{maxfunc}, we obtain
\begin{equation}\nonumber
\begin{split}
&2^{-\|\bk\|_1+\|\min\{\bj,\bj+\bl\}\|_1} \int\limits_{Q(x)} \sum\limits_{\bm \in \Omega_{\bk}(x)}
|\lambda_{\bj+\bl,\bm}|^r\chi_{\bj+\bl,\bm}(y)\,dy \\
&\qquad\qquad\lesssim  
\frac{1}{|Q(x)|}\int\limits_{Q(x)} \sum\limits_{\bm \in \Omega_{\bk}(x)} |\lambda_{\bj+\bl,\bm}|^r\chi_{\bj+\bl,\bm}(y)\,dy
\leq M\Big|\sum\limits_{\bm\in \Z^d}
\lambda_{\bj+\bl,\bm}\chi_{\bj+\bl,\bm}\Big|^r(x)\,.
\end{split}
\end{equation}
Putting this into \eqref{est2}, we arrive at
$$
     2^{-\|\bk\|_1}\sum\limits_{\bm \in
\Omega_{\bk}(x)}|\lambda_{\bj+\bl,\bm}|^r \lesssim 2^{\|\bl_+\|_1}M\Big|\sum\limits_{\bm\in \Z^d}
\lambda_{\bj+\bl,\bm}\chi_{\bj+\bl,\bm}\Big|^r(x)\,.
$$
Finally, we plug this estimate into \eqref{est1} and obtain the desired assertion.
\end{proof}

Before stating our main result we need a further definition.

\begin{definition} Let $0<q\leq\infty$, $s\in\R$, and $\balpha=(\alpha_1,\ldots,\alpha_d)>0$ such that
$\sum\limits_{i=1}^d \alpha_i= d$. \\
{\em (i)} If $0<p<\infty$ we define the sequence space $\tilde{f}^{s,\balpha}_{p,q}$ as the
collection of all sequences $(\lambda_{\bj,\bk})_{\bj\in \N_0^d, \bk\in \Z^d} \subset \C$ such that
the (quasi-)norm (usual modification in case $q=\infty$)
$$
    \|(\lambda_{\bj,\bk})_{\bj\in \N_0^d, \bk\in \Z^d}\|_{\tilde{f}^{s,\balpha}_{p,q}}:=
    \Big\|\Big(\sum\limits_{\bj\in \N_0^d}2^{\|(j_1/\alpha_1,...,j_d/\alpha_d)\|_{\infty}sq}
    \Big|\sum\limits_{\bk\in \Z^d}\lambda_{\bj,\bk}\chi_{\bj,\bk}(\cdot)\Big|^q\Big)^{1/q}\Big\|_p
    \quad \text{is finite.}
$$
{\em (ii)} If $0<p\leq \infty$ we define the sequence space $\tilde{b}^{s,\balpha}_{p,q}$ as the
collection of all sequences $(\lambda_{\bj,\bk})_{\bj\in \N_0^d, \bk\in \Z^d} \subset \C$ such that
the (quasi-)norm (usual modification in case $q=\infty$)
$$
    \|(\lambda_{\bj,\bk})_{\bj\in \N_0^d, \bk\in \Z^d}\|_{\tilde{b}^{s,\balpha}_{p,q}}:=
    \Big(\sum\limits_{\bj\in \N_0^d}2^{\|(j_1/\alpha_1,...,j_d/\alpha_d)\|_{\infty}sq}
    \Big\|\sum\limits_{\bk\in \Z^d}\lambda_{\bj,\bk}\chi_{\bj,\bk}(\cdot)\Big\|_p^q\Big)^{1/q}
    \quad \text{is finite.}
$$
\end{definition}

Now we are ready to state the wavelet characterization of the space $\F^{s,\balpha}_{p,q}(\R^d)$\,. Recall that for
$0<p,q\leq \infty$ we put
$$
      \sigma_{p,q}:=\max\{1/p-1,1/q-1,0\}\quad\mbox{and}\quad \sigma_p:=\max\{1/p-1,0\}\,.
$$

\begin{remark}\label{rem:main_wav}
The theorem below states the result for the $F$-scale of spaces $\F^{a,\balpha}_{p,q}(\R^d)$. As for the corresponding result
for the Besov type spaces $\B^{s,\balpha}_{p,q}(\R^d)$, we simply replace  condition~\eqref{f0} on $K,L$ by
$$
K,L>\sigma_p+|s|/\alpha_{\text{min}}
$$
and use the corresponding sequence spaces $\tilde{b}^{s,\balpha}_{p,q}$.
\end{remark}

\begin{theorem}\label{main_wav} Let $0<p<\infty$, $0<q\leq\infty$, 
$s\in\R$, $\balpha=(\alpha_1,\ldots,\alpha_d)>0$ with $\sum_{i=1}^d \alpha_i = d$. 
Let further $\psi_0,\psi$ be wavelets satisfying (K) and (L) above with
\begin{equation}\label{f0}
      K, L>\sigma_{p,q}+|s|/\alpha_{\text{min}}.
\end{equation}
Then any $f\in \mathcal{S}'(\R^d)$ belongs to $\F^{s,\balpha}_{p,q}(\R^d)$ if and only if it can be represented as
\begin{equation}\label{f2}
f = \sum\limits_{\bj\in \N_0^d} \sum\limits_{\bk\in \Z^d} \lambda_{\bj,\bk} \psi_{\bj,\bk} 
\end{equation}
with $(\lambda_{\bj,\bk})_{\bj,\bk} \in \tilde{f}^{s,\balpha}_{p,q}$ and the sum converging in $\mathcal{S}'(\R^d)$ with respect to some ordering. 
For each $f\in\F^{s,\balpha}_{p,q}(\R^d)$ the convergence of the representation \eqref{f2} is then even unconditional. Moreover,
if $q<\infty$, the sum also converges in $\F^{s,\balpha}_{p,q}(\R^d)$ and
$(\psi_{\bj,\bk})_{\bj,\bk}$ constitutes an unconditional basis in $\F^{s,\balpha}_{p,q}(\R^d)$. The sequence of coefficients
$\lambda(f):=(\lambda_{\bj,\bk})_{\bj,\bk}$ is uniquely determined via
\begin{equation}\label{coeff}
\lambda_{\bj,\bk}= 2^{\|\bj\|_1}\langle f , \psi_{\bj,\bk}\rangle
\end{equation}
and we have the wavelet isomorphism (equivalent (quasi-)norm)
$$
\|f\|_{\F^{s,\balpha}_{p,q}(\R^d)} \asymp \|\lambda(f)\|_{\tilde{f}^{s,\balpha}_{p,q}}\quad,\quad f\in
\F^{s,\balpha}_{p,q}(\R^d)\,.
$$
\end{theorem}

\begin{remark}\label{dualp} Following \cite[Prop.\ 3.20]{LiYaYuSaUl13}, the dual pairing of $f\in\mathcal{S}^{\prime}(\R^d)$ and $\psi_{\bj,\bk}\in C^{K}(\R^d)$ in \eqref{coeff} has to be
understood in the way
\begin{equation}\label{f3}
    \langle f, \psi \rangle:= \sum\limits_{\bj \in \N_0^d} \langle \Theta_{\bj}\ast f, \Lambda_{\bj}\ast \psi
\rangle_{L_2(\R^d)} \,.
\end{equation}
Here we choose $\Theta_{\bj}:=\cf^{-1}\theta_{\bj}$ and $\Lambda_{\bj}:=\cf^{-1}\lambda_{\bj}$ such that
$\sum\limits_{j=0}^{\infty} \theta_j \lambda_j \equiv 1\,.$
Using elementary estimates and the Nikolskij inequality in case $p<1$, one can show 
$$\F^{s,\balpha}_{p,q}(\R^d) \hookrightarrow
\B^{s,\balpha}_{p,\infty}(\R^d) \hookrightarrow S^{-|s|/\alpha_{\text{min}}-\sigma_p}_{\max\{p,1\},\infty}B(\R^d)\,.$$
Setting $s_{\balpha,p}:=|s|/\alpha_{\text{min}}+\sigma_p$ and $\tilde{p}:=\max\{p,1\}$ we obtain
\begin{equation*}
\begin{split}
   |\langle f, \psi\rangle| &\leq \sum\limits_{\bj \in \N_0^d} \|\Theta_{\bj}\ast f\|_{\tilde{p}} \|\Lambda_{\bj}\ast
    \psi\|_{\tilde{p}'}\\
    &\leq \sup\limits_{\bj \in \N_0^d} 2^{-s_{\balpha,p}\|\bj\|_1}\|\Theta_{\bj}\ast
    f\|_{\tilde{p}} \cdot \sum\limits_{\bj\in \N_0^d} 2^{s_{\balpha,p}\|\bj\|_1}\|\Lambda_{\bj}\ast
   \psi\|_{\tilde{p}'}\,\\
   &\lesssim \|f\|_{S^{-s_{\balpha,p}}_{\tilde{p},\infty}B(\R^d)}\cdot
\|\psi\|_{S^{s_{\balpha,p}}_{\tilde{p}',1}B(\R^d)}\,,
\end{split}
\end{equation*}
where the right-hand side is finite due to \eqref{f0} and \eqref{conv3}. In other words, $f\in\F^{s,\balpha}_{p,q}(\R^d)$ generates a (conjugate)
linear functional on the Banach space $S^{s_{\balpha,p}}_{\tilde{p}',1}B(\R^d)$.
\end{remark}

\begin{remark}\label{dom_mixed} As we will see below, our arguments apply as well to classical
spaces of dominating mixed smoothness $S^r_{p,q}B(\R^d)$ and $S^r_{p,q}F(\R^d)$, defined in Definition \ref{def_mixed}
above. Examining the proof, we obtain for the relation
$$
      \|f\|_{S^r_{p,q}F(\R^d)} \lesssim \|\lambda(f)\|_{s^{r}_{p,q}f} \,, 
$$
where $s^{r}_{p,q}f$ is the sequence space associated to $S^r_{p,q}F$ (for a definition see~\cite[Def.~2.1]{Vyb06}),
the condition
\begin{equation}\label{f12b}
      L>\sigma_{p,q}-r \quad \mbox{and} \quad K>r\,.
\end{equation}
The converse relation holds under the condition
$$
    K>\sigma_{p,q}-r \quad\mbox{and}\quad L>r\,.
$$
For the spaces $S^r_{p,q}B(\R^d)$ we replace $\sigma_{p,q}$ by $\sigma_p$
and $s^{r}_{p,q}f$ by $s^{r}_{p,q}b$, which is the sequence space associated to $S^r_{p,q}B$ (for a definition see~\cite[Def.~2.1]{Vyb06}).
\end{remark}

\begin{proof}[of Theorem~\ref{main_wav}]
{\em Step 1.} We consider the sum
\begin{equation}\label{sum}
      f := \sum\limits_{\bj\in \N_0^d} \sum\limits_{\bk\in \Z^d} \lambda_{\bj,\bk} \psi_{\bj,\bk}
\end{equation}
with $\lambda:=(\lambda_{\bj,\bk})_{\bj,\bk} \in \tilde{f}^{s,\balpha}_{p,q}$ and show the relation
$$
      \|f\|_{\F^{s,\balpha}_{p,q}(\R^d)} \lesssim \|\lambda\|_{\tilde{f}^{s,\balpha}_{p,q}}\,.
$$
For the issues on the convergence and uniqueness of \eqref{sum}  and \eqref{coeff} we refer to Step 3 and 4 below, where
we show that under the assumption $(\lambda_{\bj,\bk})_{\bj,\bk} \in \tilde{f}^{s,\balpha}_{p,q}$ the element $f$ is well defined, with unconditional convergence of \eqref{sum} at least in 
$\mathcal{S}'(\R^d)$, which is sufficient for the subsequent considerations. 

Let us consider $\Delta_{\bj}(f)$ for some chosen hyperbolic Littlewood-Paley analysis. This gives for
fixed
$\bj\in \N_0^d$
\begin{equation*}
    2^{\|\bj/\balpha\|_{\infty}s}|\Delta_{\bj}(f)(x)| \leq
\sum\limits_{\bl\in \Z^d}2^{\|\bj/\balpha\|_{\infty}s}\Big|\Theta_{\bj}\ast\Big(\sum\limits_{\bk\in \Z^d}
\lambda_{\bj+\bl,\bk}\psi_{\bj+\bl,\bk}\Big)(x)\Big|\,,\\
\end{equation*}
where $\Theta_{\bj} := \cf^{-1}\theta_{\bj}$ and $(\theta_{\bj})_{\bj}$ is the system from Definition \ref{def:hypLP}.
With $u:=\min\{p,q,1\}$
\begin{equation}\label{f5}
 \begin{split}
    \|f\|_{\F^{s,\balpha}_{p,q}(\R^d)} &=
\big\|2^{\|\bj/\balpha\|_{\infty}s}\Delta_{\bj}(f)(\cdot)\big\|_{L_p(\ell_q)}\\
    &\lesssim \Big(\sum\limits_{\bl\in \Z^d}
\Big\|2^{\|\bj/\balpha\|_{\infty}s}\Theta_{\bj}\ast\Big(\sum\limits_{\bk\in \Z^d}
\lambda_{\bj+\bl,\bk}\psi_{\bj+\bl,\bk}\Big)(\cdot)\Big\|^u_{L_p(\ell_{q}[\bj])}\Big)^{1/u} \,.
 \end{split} 
\end{equation}
With the help of Lemma \ref{conv2} we are aiming for pointwise estimates first.
\begin{equation}\nonumber
 \begin{split}
  &\Big|\Theta_{\bj}\ast\Big(\sum\limits_{\bk\in \Z^d}
\lambda_{\bj+\bl,\bk}\psi_{\bj+\bl,\bk}\Big)(x)\Big| \leq \sum\limits_{\bk\in \Z^d}
|\lambda_{\bj+\bl,\bk}|\cdot|(\Theta_{\bj} \ast \psi_{\bj+\bl,\bk})(x)|\\
&~~~~~~~\lesssim \Big(\prod\limits_{i=1}^d 2^{-N_{\sign(\ell_i)}|\ell_i|}\Big)\sum\limits_{\bk\in \Z^d}
|\lambda_{\bj+\bl,\bk}| \prod\limits_{i=1}^d
(1+2^{\min\{j_i,j_i+\ell_i\}}|x_i-2^{-(j_i+\ell_i)}k_i|)^{-R}\,,
\end{split}
\end{equation}
where we choose $R>1/r$ with $r<\min\{1,p,q\}=u$. Note that, due to condition $(K)$ for the wavelet and Lemma~\ref{conv2},
we can choose $R>0$ arbitrarily large. 
This allows for estimating with the help of Proposition~\ref{HL}
$$
    \Big|\Theta_{\bj}\ast\Big(\sum\limits_{\bk\in \Z^d}
\lambda_{\bj+\bl,\bk}\psi_{\bj+\bl,\bk}\Big)(x)\Big| \lesssim 2^{-\langle \bar{N}_{\sign(\bl)},
|\bl|\rangle}2^{\|\bl_+\|_1/r}\Big[M\Big|\sum\limits_{\bk\in \Z^d}
\lambda_{\bj+\bl,\bk}\chi_{\bj+\bl,\bk}(\cdot)\Big|^r\Big]^{1/r}(x)\,.
$$
Hereby we use the short-hand notation $\bl_+:=((\ell_1)_{+},\ldots,(\ell_d)_{+})$ and
$$
\langle \bar{N}_{\sign(\bl)}, |\bl|\rangle := \sum\limits_{i=1}^d N_{\sign(\ell_i)}|\ell_i| \,.
$$
Plugging this estimate into \eqref{f5} gives
\begin{equation}\nonumber
\begin{split}
    \|f\|_{\F^{s,\balpha}_{p,q}(\R^d)} \lesssim& \Big(\sum\limits_{\bl\in \Z^d}2^{-u\langle \bar{N}_{\sign(\bl)},
|\bl|\rangle}2^{u\|\bl_+\|_1/r}\\
    &~~~\Big\|2^{-\|(\bj+\bl)/\balpha\|_{\infty}s}2^{\|\bj/\balpha\|_{\infty}s}
    \Big[M\Big|2^{\|(\bj+\bl)/\balpha\|_{\infty}s}\sum\limits_{\bk\in \Z^d}
\lambda_{\bj+\bl,\bk}\chi_{\bj+\bl,\bk}\Big|^r\Big]^{1/r}(\cdot)\Big\|^u_{L_p(\ell_{q}[\bj])}\Big)^{\frac{1}{u}}_{\,.}
\end{split}
\end{equation}
Clearly, if $s\geq 0$ then $\|(\bj+\bl)/\balpha\|_\infty s \geq \|\bj/\balpha\|_{\infty}s-\|\bl/\balpha\|_{\infty}s$ and
hence
\begin{equation}\label{f7}
   2^{-\|(\bj+\bl)/\balpha\|_{\infty}s}2^{\|\bj/\balpha\|_{\infty}s} \leq 2^{\|\bl/\balpha\|_{\infty}|s|}\,.
\end{equation}
In addition, if $s<0$ we also obtain~\eqref{f7} via the usual $\ell_{\infty}$-triangle inequality.

Putting this into the previous
estimate, using the vector-valued Hardy-Littlewood maximal inequality (Theorem \ref{feffstein}) for the space
$L_{p/r}(\ell_{q/r}[\bj])$, we obtain
\begin{equation*}
  \begin{split}
   \|f\|_{\F^{s,\balpha}_{p,q}(\R^d)} \lesssim& \Big(\sum\limits_{\bl\in \Z^d}2^{-u\langle \bar{N}_{\sign(\bl)},
    |\bl|\rangle}2^{u\|\bl_+\|_1/r}2^{u\|\bl/\balpha\|_{\infty}|s|}\\
    &~~~\Big\|2^{\|(\bj+\bl)/\balpha\|_{\infty}s}\Big|\sum\limits_{\bk\in \Z^d}
\lambda_{\bj+\bl,\bk}\chi_{\bj+\bl,\bk}\Big|\Big\|^u_{L_p(\ell_{q}[\bj])}\Big)^{1/u}\\
\lesssim& \|\lambda\|_{\tilde{f}^{s,\balpha}_{p,q}}\Big(\sum\limits_{\bl\in \Z^d}2^{-u\langle \bar{N}_{\sign(\bl)},
    |\bl|\rangle }2^{u\|\bl_+\|_1/r}2^{u\|\bl/\balpha\|_1|s|}\Big)^{1/u}\,.
  \end{split}
\end{equation*}
The sum over $\bl$ converges if $L+1=N_+>1/r+|s|/\alpha_{\text{min}}$ and $K=N_-> |s|/\alpha_{\text{min}}\,.$

{\em Step 2. } Let us prove the converse relation
$\|\lambda(f)\|_{\tilde{f}^{s,\balpha}_{p,q}} \lesssim \|f\|_{\F^{s,\balpha}_{p,q}(\R^d)}$ with $\lambda(f)=(2^{\|j\|_1}\langle f,\psi_{\bj,\bk} \rangle)_{\bj,\bk}$ 
 and start with $f\in \F^{s,\balpha}_{p,q}(\R^d)$. As already pointed
out in Remark \ref{dualp}, the dual pairing $\langle f,\psi_{\bj,\bk}\rangle$ makes sense due to condition~\eqref{f0}. Our estimation begins as follows
\begin{equation}\label{step2}
  \begin{split}
     |2^{\|j\|_1}\langle f,\psi_{\bj,\bk}\rangle| \leq& \sum\limits_{\bl \in \Z^d} |2^{\|\bj\|_1}\langle
\Theta_{\bj+\bl}\ast f, \Lambda_{\bj+\bl}\ast \psi_{\bj,\bk}\rangle|\\
 =& \sum\limits_{\bl\in \Z^d} \sum\limits_{\bm\in \Z^d} \Big|\int\limits_{Q_{\bj+\bl,\bm}}(\Theta_{\bj+\bl}\ast
f)(y)2^{\|\bj\|_1}(\Lambda_{\bj+\bl}\ast \psi_{\bj,\bk})(y)\,dy\Big|\\
\leq& \sum\limits_{\bl\in \Z^d} \sum\limits_{\bm\in \Z^d} |\theta_{\bj+\bl,\bm}(f)|
  \int\limits_{Q_{\bj+\bl,\bm}}2^{\|\bj\|_1}|(\Lambda_{\bj+\bl}\ast \psi_{\bj,\bk})(y)|\,dy\,,
  \end{split}
\end{equation}
where we put
\begin{equation}\label{f8}
   \theta_{\bj+\bl,\bm}(f):=\sup\limits_{y\in Q_{\bj+\bl,\bm}} |(\Theta_{\bj+\bl}\ast f)(y)|\,.
\end{equation}
We next estimate the integral, as in Step~1 with the help of Lemma~\ref{conv2}. Here we have to be particularly careful with
the normalization factors. Note that, compared to \eqref{conv3}, the signs of the components of $\bl$ in the
convolution $\Lambda_{\bj+\bl}\ast \psi_{\bj,\bk}$ change the role. This is why we put this time $M_0:=0$, $M_-:= L$, $M_+:= K+1$ and $\bar{M}_{\sign(\bl)}:=(M_{\sign(b_1)},\ldots,M_{\sign(b_d)})$. We then obtain for $z=(z_1,\ldots,z_d)\in
Q_{\bj,\bk}$ the estimate
\begin{equation*}
 \begin{split}
    &\int\limits_{Q_{\bj+\bl,\bm}}2^{\|j\|_1}|(\Lambda_{\bj+\bl}\ast \psi_{\bj,\bk})(y)|\,dy\\
    &~~~~~\lesssim \sup\limits_{y\in Q_{\bj+\bl,\bm}}2^{-\langle \bar{M}_{\sign(\bl)},|\bl|
\rangle}\prod\limits_{i=1}^d (1+2^{\min\{j_i,j_i+\ell_i\}}|y_i-2^{-j_i}k_i|)^{-R}\\
&~~~~~\lesssim 2^{-\langle \bar{M}_{\sign(\bl)},|\bl|
\rangle}\prod\limits_{i=1}^d (1+2^{\min\{j_i,j_i+\ell_i\}}|2^{-(j_i+\ell_i)}m_i-2^{-j_i}k_i|)^{-R}\\
&~~~~~\lesssim 2^{-\langle \bar{M}_{\sign(\bl)},|\bl|
\rangle}\prod\limits_{i=1}^d (1+2^{\min\{j_i,j_i+\ell_i\}}|2^{-(j_i+\ell_i)}m_i-z_i|)^{-R}\,.
 \end{split}
\end{equation*}
This together with \eqref{step2} and Proposition~\ref{HL} yields for $r<\min\{p,q,1\}=u$ and $z\in \R^d$
\begin{equation*}
  \begin{split}
    &\sum\limits_{\bk\in \Z^d} |2^{\|j\|_1}\langle
f,\psi_{\bj,\bk}\rangle|\chi_{\bj,\bk}(z)\\
    &~~~~~\lesssim \sum\limits_{\bl\in \Z^d}2^{-\langle \bar{M}_{\sign(\bl)},|\bl|
\rangle} \sum\limits_{\bm\in \Z^d} |\theta_{\bj+\bl,\bm}(f)|\prod\limits_{i=1}^d
(1+2^{\min\{j_i,j_i+\ell_i\}}|2^{-(j_i+\ell_i)}m_i-z_i|)^{-R}\\
&~~~~~\lesssim \sum\limits_{\bl\in \Z^d}2^{-\langle \bar{M}_{\sign(\bl)},|\bl|\rangle}
2^{\|\ell_+\|_1/r}\Big[M\Big|\sum\limits_{\bm\in \Z^d}
\theta_{\bj+\bl,\bm}(f)\chi_{\bj+\bl,\bm}(\cdot)\Big|^r\Big]^{1/r}(z)\,.
\end{split}
\end{equation*}
This leads to
\begin{equation}\nonumber
 \begin{split}
   &2^{\|\bj/\balpha\|_{\infty}s}\sum\limits_{\bk\in \Z^d} |2^{\|j\|_1}\langle
f,\psi_{\bj,\bk}\rangle|\chi_{\bj,\bk}\\
&~~~\lesssim \sum\limits_{\bl\in \Z^d}
2^{\|\ell_+\|_1/r-\langle\bar{M}_{\sign(\bl)},|\bl|\rangle+\|\bj/\balpha\|_{\infty}s-\|(\bj+\bl)/\balpha
\|_{\infty}s}\Big[M\Big|2^{\|(\bj+\bl)/\balpha\|_{\infty}} \sum\limits_ { \bm\in \Z^d}
\theta_{\bj+\bl,\bm}(f)\chi_{\bj+\bl,\bm}(\cdot)\Big|^r\Big]^{\frac{1}{r}}_{\,.}
 \end{split}
\end{equation}
Taking the $L_p(\ell_{q}[\bj])$-(quasi-)norm on both sides and using \eqref{f7} once more, we obtain
\begin{equation}\label{f9}
  \begin{split}
     &\Big\|2^{\|\bj/\balpha\|_{\infty}s}\sum\limits_{\bk\in \Z^d} |2^{\|j\|_1}\langle
f,\psi_{\bj,\bk}\rangle|\chi_{\bj,\bk}\Big\|_{L_p(\ell_{q}[\bj])}\\
&~~\lesssim\Big(\sum\limits_{\bl\in \Z^d}2^{-u\langle \bar{M}_{\sign(\bl)},|\bl|\rangle}
2^{u\|\ell_+\|_1/r}2^{u\|\bl/\balpha\|_{\infty}|s|}\times\\
&~~~~~~~\times\Big\|\Big[M\Big|2^{\|(\bj+\bl)/\balpha\|_{\infty}s}
\sum\limits_{ \bm\in \Z^d}
\theta_{\bj+\bl,\bm}(f)\chi_{\bj+\bl,\bm}(\cdot)\Big|^r\Big]^{1/r}\Big\|^u_{L_p(\ell_{q}[\bj])}\Big)^{1/u}\,.
  \end{split}
\end{equation}
Due to $r<\min\{1,p,q\} = u$ we can apply the Hardy-Littlewood maximal inequality (Theorem~\ref{feffstein}) and obtain
$$
\Big\|\Big[M\Big|2^{\|(\bj+\bl)/\balpha\|_{\infty}s} \sum\limits_{
\bm\in \Z^d}
\theta_{\bj+\bl,\bm}(f)\chi_{\bj+\bl,\bm}\Big|^r\Big]^{1/r}\Big\|_{L_p(\ell_{q}[\bj])}
\lesssim \Big\|2^{\|\bj/\balpha\|_{\infty}s} \sum\limits_{
\bm\in \Z^d}
\theta_{\bj,\bm}(f)\chi_{\bj,\bm}\Big\|_{L_p(\ell_{q})}\,.
$$
From \eqref{f8} we obtain for any $z\in Q_{\bj,\bm}$ and any $a>0$
$$
   |\theta_{\bj,\bm}(f)| = \sup\limits_{y\in Q_{\bj,\bm}} |(\Theta_{\bj}\ast f)(y)| \lesssim
\sup\limits_{y\in \R^d}\frac{|(\Theta_{\bj}\ast f)(y)|}{\prod\limits_{i=1}^d(1+2^{j_i}|z_i-y_i|)^a} =
P_{2^{\bj},a}(\Theta_{\bj}\ast f)(z)\,,
$$
where we used the definition of the Peetre maximal function in \eqref{petfefste}. Choosing $a>\max\{\frac{1}{p},\frac{1}{q}\}$ 
with the corresponding maximal inequality in Theorem \ref{peetremax} then yields the relation
\begin{equation*}
 \begin{split}
   \Big\|2^{\|\bj/\balpha\|_{\infty}s} \sum\limits_{
   \bm\in \Z^d} \theta_{\bj,\bm}(f)\chi_{\bj,\bm}\Big\|_{L_p(\ell_{q})} &\lesssim
\|P_{2^{\bj},a}(2^{\|\bj/\balpha\|_{\infty}s}\Theta_{\bj}\ast f)\|_{L_p(\ell_q)}\\
&\lesssim \|2^{\|\bj/\balpha\|_{\infty}s}\Theta_{\bj}\ast f\|_{L_p(\ell_q)}
\asymp \|f\|_{\F^{s,\balpha}_{p,q}(\R^d)}\,.
\end{split}
\end{equation*}
Returning to \eqref{f9}, we have seen
\begin{equation*}
 \begin{split}
   &\Big\|2^{\|\bj/\balpha\|_{\infty}s}\sum\limits_{\bk\in \Z^d} |2^{\|j\|_1}\langle
   f,\psi_{\bj,\bk}\rangle|\chi_{\bj,\bk}\Big\|_{L_p(\ell_{q}[\bj])}\\
   &~~~~~\lesssim\|f\|_{\F^{s,\balpha}_{p,q}(\R^d)}\Big(\sum\limits_{\bl\in \Z^d}2^{-u\langle
\bar{M}_{\sign(\bl)},|\bl|\rangle}
   2^{u\|\ell_+\|_1/r}2^{u\|\bl/\balpha\|_{\infty}|s|}\Big)^{1/u}\,.
\end{split}
\end{equation*}
It remains to discuss the sum over $\bl$. It is easy to see that it converges if $K+1 = M_+>
1/r+|s|/\alpha_{\text{min}}$ and $L=M_- > |s|/\alpha_{\text{min}}$. Recall that $r$ is chosen such that
$r<\min\{1,p,q\}$.

{\em Step 3.} Let us now clarify the convergence issues in \eqref{f2} in case
$q<\infty$. The arguments in Step 1 above show in particular for a finite partial summation of \eqref{f2} that
$$
    \Big\|\sum\limits_{\bj}\sum\limits_{\bk} \lambda_{\bj,\bk}\psi_{\bj,\bk}\Big\|_{\F^{s,\balpha}_{p,q}(\R^d)}
\lesssim \Big\|\Big(\sum\limits_{\bj}2^{\|\bj/\balpha\|_{\infty}sq}\Big|\sum\limits_{\bk}
\lambda_{\bj,\bk}\chi_{\bj,\bk}\Big|^q\Big)^{1/q}\Big\|_p\,.
$$
If $q<\infty$ (note that $p<\infty$ anyway) we use Lebesgue's dominated convergence theorem to conclude the
unconditional convergence of \eqref{f2} in $\F^{s,\balpha}_{p,q}(\R^d)$. The required majorant is thereby given by
$$\|(\lambda_{\bj,\bk})_{\bj\in \N_0^d, \bk\in \Z^d}\|_{\tilde{f}^{s,\balpha}_{p,q}} < \infty.$$ In case $q=\infty$ we
use the observation in Remark \ref{dom_mixed}. From a simple application of H\"older's inequality (with respect to the
sum
over $\bj$) we first obtain for any $\varepsilon>0$ the relation
\begin{equation}\label{majorant}
   \|(\lambda_{\bj,\bk})_{\bj\in \N_0^d,\bk\in \Z^d}\|_{s^{r(s,\balpha)-\varepsilon}_{p,1}f} \lesssim
\|(\lambda_{\bj,\bk})_{\bj\in \N_0^d,\bk\in \Z^d}\|_{\tilde{f}^{s,\balpha}_{p,\infty}} \quad\text{with $r(s,\balpha) := -|s|/\alpha_{\min}$}\,.
\end{equation}
Choosing $\varepsilon>0$ small enough, we then obtain from \eqref{f0} that
condition~\eqref{f12b} in Remark \ref{dom_mixed} is satisfied. Hence, for a finite partial
summation of \eqref{f2} we have
$$
    \Big\|\sum\limits_{\bj}\sum\limits_{\bk}
\lambda_{\bj,\bk}\psi_{\bj,\bk}\Big\|_{S^{r(s,\alpha)-\varepsilon}_{p,1}F(\R^d)}
\lesssim \Big\|\sum\limits_{\bj}2^{(r(s,\balpha)-\varepsilon)\|\bj\|_1}\sum\limits_{\bk}
|\lambda_{\bj,\bk}|\chi_{\bj,\bk}\Big\|_p\,.
$$
Again, by Lebesgue's dominated convergence theorem (the majorant given by \eqref{majorant}) we see the unconditional
convergence of \eqref{f2} in the space $S^{r(s,\alpha)-\varepsilon}_{p,1}F(\R^d)$. Taking the embedding
$S^{r(s,\alpha)-\varepsilon}_{p,1}F(\R^d) \hookrightarrow \mathcal{S}'(\R^d)$ into account, we actually proved more
than stated in the theorem.

{\em Step 4.} It remains to prove \eqref{f2} for $f\in \F^{s,\balpha}_{p,q}(\R^d)$ and coefficients
$\lambda_{\bj,\bk}(f)$ chosen as in \eqref{coeff}. From Steps 1, 2, 3 above we have learned that
$\{\lambda_{\bj,\bk}(f)\}_{\bj,\bk} \in \tilde{f}^{s,\balpha}_{p,q}$, which implies that the sum
$$
    \sum\limits_{\bj\in \N_0^d}\sum\limits_{\bk\in \Z^d} \lambda_{\bj,\bk}(f)\psi_{\bj,\bk}
$$
converges (at least in) $\mathcal{S}'(\R^d)$ to an element $g \in \mathcal{S}'(\R^d)$. We now prove that
$f(\varphi) = g(\varphi)$ for all $\varphi \in \mathcal{S}(\R^d)$. Fix $\varphi\in \mathcal{S}(\R^d)$, then clearly
$\bar{\varphi} \in L_2(\R^d)$ and we have
\begin{equation}\label{f30}
    \bar{\varphi} = \sum\limits_{\bj \in \N_0^d} \sum\limits_{\bk\in \Z^d} \langle \bar{\varphi}, \psi_{\bj,\bk}\rangle
\psi_{\bj,\bk}
\end{equation}
with convergence in $L_2(\R^d)$. Since $\bar{\varphi} \in S^{s_{\balpha,p}}_{\tilde{p}',1}B(\R^d)$ we have by Step 1,
2, 3 above that the right-hand side of \eqref{f30} converges in $S^{s_{\balpha,p}}_{\tilde{p}',1}B(\R^d)$ to some $\eta
\in S^{s_{\balpha,p}}_{\tilde{p}',1}B(\R^d)$. Hence, we have $\bar{\varphi} = \eta$ in $\mathcal{S}'(\R^d)$ which
finally gives $\bar{\varphi} = \eta$ almost everywhere and, in other words, \eqref{f30} holds true in
$S^{s_{\balpha,p}}_{\tilde{p}',1}B(\R^d)$. Then $f(\varphi)$ can be rewritten as follows, using the continuity of
$\langle f,\cdot\rangle$ (see Remark \ref{dualp}),
\begin{equation*}
\begin{split}
  f(\varphi) = \langle f,\bar{\varphi} \rangle &= \Big\langle f,\sum\limits_{\bj\in \N_0^d}\sum\limits_{\bk \in
\Z^d} \langle
\bar{\varphi} ,\psi_{\bj,\bk}\rangle\psi_{\bj,\bk}\Big\rangle \\
&= \sum\limits_{\bj\in \N_0^d}\sum\limits_{\bk \in \Z^d} \overline{\langle \bar{\varphi}
,\psi_{\bj,\bk}\rangle}\langle f,\psi_{\bj,\bk}\rangle
= \sum\limits_{\bj\in \N_0^d}\sum\limits_{\bk \in \Z^d} \langle \varphi
,\overline{\psi_{\bj,\bk}}\rangle\langle f,\psi_{\bj,\bk}\rangle\,.
\end{split}
\end{equation*}
On the other hand,
$$
    g(\varphi) = \sum\limits_{\bj\in \N_0^d}\sum\limits_{\bk\in \Z^d}\lambda_{\bj,\bk}(f)\psi_{\bj,\bk}(\varphi) =
\sum\limits_{\bj\in \N_0^d}\sum\limits_{\bk \in \Z^d} \langle \varphi
,\overline{\psi_{\bj,\bk}}\rangle\langle f,\psi_{\bj,\bk}\rangle = f(\varphi)\,,
$$
which finishes the proof. 
\end{proof}


\section{Hyperbolic Haar characterization}

\noindent
We next utilize a hyperbolic Haar basis for the characterization of the spaces
$\widetilde{B}^{s,\balpha}_{p,q}(\R^d)$ and $\widetilde{F}^{s,\balpha}_{p,q}(\R^d)$ from
Definitions~\ref{defB} and~\ref{defHTLS},
the main result being Theorem~\ref{thm:haarmain}.
It will show that Haar characterizations are possible in a certain restricted range of parameters, although the Haar wavelet does not fulfill smoothness requirements $(K)$ as assumed for the derivation of Theorem~\ref{main_wav} in the previous section.
Hence, for the proof of Theorem~\ref{thm:haarmain} a different methodology is needed than in Section~4. 
We follow the technique used in \cite{GaSeUl2019b}, exploiting the special
structure of the Haar wavelet.

We begin by fixing a convenient inhomogeneous Haar system on the real line, namely 
\[\label{HaarS}
\mathcal{H}_{1} := \big\{h_{j,k}~:~k\in\Z,\, j\in\N_{0} \big\}\,,
\]
where for $j\in\N$, $k\in\Z$, the functions $h_{j,k}$ are scaled Haar functions of the form
$$
h_{j,k}(x):=\frac{1}{\sqrt{2}} h(2^{j-1}x - k)  \quad\text{,}\quad\text{where}\quad h(x):=\bbone_{I^+_{0,0}}(x)-\bbone_{I^-_{0,0}}(x) \,.
$$
The intervals
$I^{+}_{j,k}=[2^{-j}k, 2^{-j}(k+1/2))$ and
$I^{-}_{j,k}=[2^{-j}(k+1/2), 2^{-j}(k+1))$ thereby represent the dyadic children
of the standard dyadic intervals $I_{j,k}=[2^{-j}k, 2^{-j}(k+1))$.
At the lowest scale $j=0$ the ordinary Haar functions $\bbone_{I^+_{0,k}}-\bbone_{I^-_{0,k}}$ are replaced by the characteristic functions $h_{0,k}:=\bbone_{I_{0,k}}$. Further, we set $h_{j,k} \equiv 0$ if $j<0$.
Defined like this, the structure of the system $\mathcal{H}_{1}$
fits closely to the wavelet systems considered in Section 4. The
inhomogeneous scale is at $j=0$ (and not the usual standard $j=-1$ for Haar systems).

For dimension $d\in\N$ we derive a corresponding hyperbolic $d$-variate Haar system by the following tensorization procedure,
\begin{align}\label{HaarSd}
\mathcal{H}_{d} := \big\{ h_{\bj,\bk}:= h_{j_1,k_1} \otimes \cdots \otimes h_{j_d,k_d}  ~:~ \bk=(k_1,\ldots,k_d)\in\Z^d, \bj=(j_1,\ldots,j_d)\in\N^d_{0} \big\}\,.
\end{align}
Note that for $\bj\in\N^d$ and $\bk\in\Z^{d}$ the cube
\begin{align*}
Q_{\bj,\bk}:=[2^{-j_1}k_1,2^{-j_1}(k_1+1))\times \cdots \times [2^{-j_d}k_d,2^{-j_d}(k_d+1)),
\end{align*}
whose characteristic function will subsequently be denoted by $\chi_{\bj,\bk}$ as already earlier in \eqref{dyQubes},
corresponds to the strict support of the Haar function $h_{\bj,\bk}$. 
At each fixed ``scale'' $\bj$ these cubes represent 
a partition of the $d$-dimensional domain $\R^d$.

\begin{proposition}\label{lem5.1}
Let $1<p,q<\infty$, $s\in\R$, and $\balpha=(\alpha_1,\ldots,\alpha_d)>0$ such that $\sum_{i=1}^d\alpha_i=d$.
Under the condition
\begin{align*}
|s|/\alpha_{\text{min}}< 
\min\Big\{1-\frac{1}{p},1-\frac{1}{q}\Big\}
\end{align*}
we have for $f\in S^\prime(\R^d)$ (with the dual pairing $\langle  f,h_{\bj,\bk} \rangle$ defined as in \eqref{f3} in Remark~\ref{dualp})
\begin{align}\label{lem5.1(i)}
\Big\| \Big( \sum_{\bj\in\N^{d}_{0}} 2^{\|\bj/\balpha\|_\infty sq} \Big| \sum_{\bk\in\Z^d} 2^{\|\bj\|_1} \langle f,h_{\bj,\bk} \rangle \chi_{\bj,\bk}(x)  \Big|^q \Big)^{1/q} \Big\|_p \lesssim \|f\|_{\F^{s,\balpha}_{p,q}} \,,
\end{align}
whenever the left-hand side is defined. In case
\: $|s|/\alpha_{\text{min}}< 1-\frac{1}{p}$ \: we have
\begin{align}\label{lem5.1(ii)}
\Big(  \sum_{\bj\in\N^{d}_{0}} 2^{\|\bj/\balpha\|_\infty sq} \Big\| \sum_{\bk\in\Z^d} 2^{\|\bj\|_1} \langle f,h_{\bj,\bk} \rangle \chi_{\bj,\bk}(\cdot)  \Big\|^q_p  \Big)^{1/q} \lesssim \|f\|_{\B^{s,\balpha}_{p,q}} \,.
\end{align}
\end{proposition}

\begin{proof}
For the proof, we first build a suitable decomposition 
of unity adapted to the hyperbolic tiling 
of the frequency domain. 
	For a respective construction, we start with univariate functions
	$\phi_0,\phi\in S(\R)$ and $\lambda_0,\lambda\in S(\R)$ such that
	\begin{align*}
	\lambda_0\phi_0 + \sum_{j\in\N} \lambda(2^{-j}\cdot)\phi(2^{-j}\cdot) = \sum_{j\in\N_0} \lambda_j\phi_j \equiv 1,
	\end{align*}
	where $\phi_j:=\phi(2^{-j}\cdot)$ and $\lambda_j:=\lambda(2^{-j}\cdot)$ for $j\in\N$. The functions $\phi_0$ and $\phi_1$ shall thereby, as usual, be compactly supported with
	\begin{align*}
	\supp(\phi_0) \subset \{ |x|\le 2\varepsilon \} \quad,\quad 
		\supp(\phi) \subset\{ \varepsilon/2\le|x|\le 2\varepsilon \} 
	\end{align*}
	for some $\varepsilon>0$.
	As a consequence, their inverse Fourier transforms however, namely $\Phi_0:={\mathcal{F}}^{-1} \phi_0$ and $\Phi:={\mathcal{F}}^{-1} \phi$,
	cannot have compact supports.
	
	The functions $\lambda_0$, $\lambda$, on the other hand are chosen such that
	the supports of
	$\Lambda_0:={\mathcal{F}}^{-1} \lambda_0$ and $\Lambda:={\mathcal{F}}^{-1} \lambda$ are compact. Further, they are assumed to fulfill the Tauberian conditions
	\begin{align*}
	|\lambda_0(x)|>0 \quad\text{on } \{ |x|\le 2\varepsilon \} \quad,\quad
	|\lambda(x)|>0 \quad\text{on } \{ \varepsilon/2\le|x|\le 2\varepsilon \}
	\end{align*}
	with the same $\varepsilon>0$ as above and furthermore $D^{\bgamma}\lambda(0)=0$ for multi-indices $\bgamma\in\N_0^{d}$ with $\|\bgamma\|_1\le1$. 
	Such a construction is indeed possible, see \cite[Lem.~3.6]{Ullrich2016TheRO}
	for example.

	 For the subsequent proof, it is convenient to also define the functions $\Phi_j:={\mathcal{F}}^{-1} \phi_j$ and $\Lambda_j:={\mathcal{F}}^{-1}\lambda_j$
	 for $j\in\N$.
	 They fulfill the scaling relations $\Phi_j=2^{j}\Phi(2^{j}\cdot)$ and $\Lambda_j=2^{j}\Lambda(2^{j}\cdot)$.
		
	Next, we put $\Phi_j:=0$ and $\Lambda_j:=0$ for $j\in\Z$ with $j<0$ and build the tensor products
	\begin{align*}
	\Phi_{\bl}:=\bigotimes_{i\in\{1,\ldots,d\}} \Phi_{\ell_i} \quad\text{and}\quad
	\Lambda_{\bl}:=\bigotimes_{i\in\{1,\ldots,d\}} \Lambda_{\ell_i}
	\quad\text{for $\bl\in\Z^{d}$}.
	\end{align*}
	
	Then we have the decomposition, which in fact is a discrete version of Calder\'on's reproducing formula,
	\begin{align*}
	f=\sum_{\bl\in\Z^{d}} \Phi_{\bl}\ast\Lambda_{\bl}\ast f \quad\text{for every }f\in S^\prime(\R^{d}) \,,
	\end{align*}
	enabling a component-wise evaluation of the scalar product $\langle f,h_{\bj,\bk} \rangle$. 
	Each Haar coefficient can in this way be understood in the following sense (see also Remark~\ref{dualp}),
    \begin{align*}
    \langle f,h_{\bj,\bk} \rangle = \sum_{\bl\in\Z^{d}} \langle  \Phi_{\bl}\ast\Lambda_{\bl}\ast f,
	h_{\bj,\bk} \rangle = \sum_{\bl\in\Z^{d}} \langle  \Phi_{\bl}\ast f,
	\Lambda_{\bl}(-\cdot)\ast h_{\bj,\bk} \rangle \,,
    \end{align*}
    whenever the right-hand sum converges.

    If we further assume that $\Lambda_{\bl}$ is even, we arrive at the estimate
	\begin{align*}
	|2^{\|\bj\|_1}\langle f,h_{\bj,\bk} \rangle|
	 \lesssim \sum_{\bl\in\Z^{d}} \Big| \int_{\R^{d}} 2^{\|\bj\|_1} (\Phi_{\bj+\bl}\ast f)(y) (\Lambda_{\bj+\bl}\ast h_{\bj,\bk})(y) \,dy \Big|.
	\end{align*}

	Let us investigate the integral on the right-hand side, and for this let us define
	\begin{align*}
	S_{\bj,\bk,\bl} :=\supp \big(\Lambda_{\bj+\bl}\ast h_{\bj,\bk}\big).
	\end{align*}
	If $\min_{i\in[d]}\{j_i+\ell_i\}<0$ 
	we have $S_{\bj,\bk,\bl}=\emptyset$ and the integral vanishes.
	Otherwise, when $\min_{i\in[d]}\{j_i+\ell_i\}\ge0$, we fix $a>0$ and $x\in\R^{d}$
	and obtain the estimate
	\begin{align*}
	\Big| &\int_{\R^{d}} 2^{\|\bj\|_1} (\Phi_{\bj+\bl}\ast f)(y) (\Lambda_{\bj+\bl}\ast h_{\bj,\bk})(y) \,dy \Big| \\
	&\le P_{2^{\bj+\bl},a}(\Phi_{\bj+\bl}\ast f)(x) \cdot \sup_{y\in S_{\bj,\bk,\bl}} \Big[ \prod_{i=1}^{d} (1+2^{j_i+\ell_i}|x_i-y_i|)^a \Big]
	\cdot \Big|\int_{\R^{d}} 2^{\|\bj\|_1}  (\Lambda_{\bj+\bl}\ast h_{\bj,\bk})(z) \,dz\Big| \,,
	\end{align*}
	where $P_{2^{\bj+\bl},a}(\Phi_{\bj+\bl}\ast f)$ denotes the Peetre maximal function (see~\eqref{petfefste})
	\begin{align*}
	P_{2^{\bj+\bl},a}(\Phi_{\bj+\bl}\ast f)(x) = \sup_{y\in\R^{d}} \frac{|(\Phi_{\bj+\bl}\ast f)(y)|}{(1+2^{j_1+\ell_1}|x_1-y_1|)^a\cdots(1+2^{j_d+\ell_d}|x_d-y_d|)^a}.
	\end{align*}
	
	The integral term splits into
    \begin{align*}
	 \int_{\R^{d}} 2^{\|\bj\|_1}  (\Lambda_{\bj+\bl}\ast h_{\bj,\bk})(z) \,dz =  2^{\|\bj\|_1} \int_{\R} (\Lambda_{j_1+\ell_1}\ast h_{j_1,k_1})(t)\,dt  \cdots \int_{\R} (\Lambda_{j_d+\ell_d}\ast h_{j_d,k_d})(t)\,dt
	\end{align*}
    according to the relation
	\begin{align*}
	\Lambda_{\bj+\bl}\ast h_{\bj,\bk} = (\Lambda_{j_1+\ell_1}\ast h_{j_1,k_1}) \otimes \cdots \otimes (\Lambda_{j_d+\ell_d}\ast h_{j_d,k_d}).
	\end{align*}
	
	For fixed $i\in\{1,\ldots,d\}$, assuming $\ell_i<0$, we can then further estimate
	\begin{align}\label{auxformula1}
	\int_{\R} (\Lambda_{j_i+\ell_i}\ast h_{j_i,k_i})(y) \,dy
	\lesssim 2^{-j_i+\ell_i}
	\end{align}
	since $|\supp(\Lambda_{j_i+\ell_i}\ast h_{j_i,k_i})| \asymp 2^{-(j_i+\ell_i)}$ and $\| \Lambda_{j_i+\ell_i}\ast h_{j_i,k_i} \|_\infty \lesssim 2^{2\ell_i}$.
	For the latter of these two inequalities the first order vanishing moment of the Haar wavelet comes into play. Note here that indeed $j_i>0$ due to $\min_{i\in[d]}\{j_i+\ell_i\}\ge0$, allowing for the estimate
	\begin{align*}
	\| \Lambda_{j_i+\ell_i}\ast h_{j_i,k_i}\|_\infty &= \| (\Lambda_{j_i+\ell_i}-\Lambda_{j_i+\ell_i}(2^{-j}k_i))\ast h_{j_i,k_i}\|_\infty  \\
    &\le \|\Lambda_{j_i+\ell_i}-\Lambda_{j_i+\ell_i}(2^{-j}k_i)\|_\infty \|h_{j_i,k_i}\|_1
	\lesssim 2^{2\ell_i}.
	\end{align*}

    In case $\ell_i\ge 0$ we obtain a different estimate than \eqref{auxformula1}, namely
	\begin{align*}
	\int_{\R} (\Lambda_{j_i+\ell_i}\ast h_{j_i,k_i})(y) \,dy
	\lesssim 2^{-(j_i+\ell_i)}.
	\end{align*}
	Here we use the fact that the integrand is bounded by a constant together with the observation that its support is contained in
    at most three intervals of length $\asymp 2^{-(j_i+\ell_i)}$.
    Indeed, as a consequence of the $L_1$- resp. $L_\infty$-normalization of $\Lambda_{j_i+\ell_i}$ and $h_{j_i,k_i}$, we have
    $\|\Lambda_{j_i+\ell_i}\ast h_{j_i,k_i}\|_\infty\le \|\Lambda_{j_i+\ell_i}\|_1 \|h_{j_i,k_i}\|_\infty \lesssim 1$.
    Furthermore, due to the vanishing moment properties of $\Lambda_{j_i+\ell_i}$,
    the support of the convolution merely stems from the either two or three discontinuities of the function $h_{j_i,k_i}$.

	Now, let us turn our attention to the factor
	\begin{align*}
	\sup_{y\in S_{\bj,\bk,\bl}} \Big[ \prod_{i=1}^{d} (1+2^{j_i+\ell_i}|x_i-y_i|)^a \Big].
	\end{align*}
	Here, we have with $x_i\in Q_{j_i,k_i}$ and $y_i\in \supp (\Lambda_{j_i+\ell_i}\ast h_{j_i,k_i}) 
	\subset  \supp (\Lambda_{j_i+\ell_i}) + \supp(h_{j_i,k_i})$ (and therefore $|x_i-y_i|\lesssim 2^{-j_i}$ if $\ell_i\ge 0$ and $|x_i-y_i|\lesssim 2^{-(j_i+\ell_i)}$ if $\ell_i<0$)
	\begin{align*}
	(1+2^{j_i+\ell_i}|x_i-y_i|)^a \lesssim 1
	\end{align*}
	if $\ell_i<0$. Otherwise, if $\ell_i\ge0$, we estimate
	\begin{align*}
	(1+2^{j_i+\ell_i}|x_i-y_i|)^a \lesssim 2^{\ell_ia}.
	\end{align*}
	
	Putting all together, this yields for $x\in Q_{\bj,\bk}$
	\begin{align*}
	\Big| \int_{\R^{d}} 2^{\|\bj\|_1} (\Phi_{\bj+\bl}\ast f)(y) (\Lambda_{\bj+\bl}\ast h_{\bj,\bk})(y) \,dy \Big| \lesssim
	A(\bl,a) P_{2^{\bj+\bl},a}(\Phi_{\bj+\bl}\ast f)(x) \,,
	\end{align*}
	where
	\begin{align*}
	A(\bl,a):=\prod_{i\in\{1,\ldots,d\}} A(\ell_i,a) \quad\text{with}\quad
	A(\ell_i,a):= \begin{cases} 2^{\ell_i} \quad &, \ell_i<0 \\ 2^{(a-1)\ell_i} &, \ell_i\ge 0 \end{cases}.
	\end{align*}
	
	Hence, we obtain uniformly in $x\in\R^d$ and for fixed $\bj\in\N^{d}_{0}$
	\begin{align*}
	\sum_{\bk\in\Z^d} |2^{\|\bj\|_1} \langle f,h_{\bj,\bk} \rangle| \chi_{\bj,\bk} (x) \lesssim
	\sum_{\bl\in\Z^{d}} A(\bl,a) P_{2^{\bj+\bl},a}(\Phi_{\bj+\bl}\ast f) (x).
	\end{align*}
	
	Finally, we can turn to the proof of \eqref{lem5.1(i)}. 
	We estimate
	\begin{align*}
	\Big\| &\Big( \sum_{\bj\in\N^{d}_{0}} 2^{\|\bj/\balpha\|_\infty sq} \Big| \sum_{\bk\in\Z^d} 2^{\|\bj\|_1} \langle f,h_{\bj,\bk} \rangle \chi_{\bj,\bk}(x)  \Big|^q \Big)^{1/q} \Big\|_p \\
	&\lesssim  \Big\| \Big( \sum_{\bj\in\N^{d}_{0}} 2^{\|\bj/\balpha\|_\infty sq} \Big| \sum_{\bl\in\Z^{d}} A(\bl,a) P_{2^{\bj+\bl},a}(\Phi_{\bj+\bl}\ast f) (x)  \Big|^q \Big)^{1/q} \Big\|_p \\
	&=  \Big\| \Big( \sum_{\bj\in\N^{d}_{0}}    \Big| \sum_{\bl\in\Z^{d}} A(\bl,a) 2^{(\|\bj/\balpha\|_\infty - \|(\bj+\bl)/\balpha\|_\infty)s}  2^{\|(\bj+\bl)/\balpha\|_\infty s}  P_{2^{\bj+\bl},a}(\Phi_{\bj+\bl}\ast f) (x)  \Big|^q \Big)^{1/q} \Big\|_p .
	\end{align*}
	According to \eqref{f7} it holds $2^{(\|\bj/\balpha\|_\infty -\|(\bj+\bl)/\balpha\|_\infty)s}\le 2^{\|\bl/\balpha\|_\infty |s|} $ for $s\in\R$, and hence
    \begin{align*}
	\Big\| &\Big( \sum_{\bj\in\N^{d}_{0}} 2^{\|\bj/\balpha\|_\infty sq} \Big| \sum_{\bk\in\Z^d} 2^{\|\bj\|_1} \langle f,h_{\bj,\bk} \rangle \chi_{\bj,\bk}(x)  \Big|^q \Big)^{1/q} \Big\|_p \\
	&\lesssim  \Big\| \Big( \sum_{\bj\in\N^{d}_{0}}    \Big| \sum_{\bl\in\Z^{d}} A(\bl,a)  2^{\|\bl/\balpha\|_\infty |s|}  2^{\|(\bj+\bl)/\balpha\|_\infty s}  P_{2^{\bj+\bl},a}(\Phi_{\bj+\bl}\ast f) (x)  \Big|^q \Big)^{1/q} \Big\|_p \\
	&\le  \sum_{\bl\in\Z^{d}} A(\bl,a) 2^{\|\bl/\balpha\|_\infty |s|} \cdot \Big\| \Big( \sum_{\bj\in\N^{d}_{0}}  \Big| 2^{\|\bj/\balpha\|_\infty s}  P_{2^{\bj},a}(\Phi_{\bj}\ast f)(x) \Big|^q  \Big)^{1/q} \Big\|_p \,,
	\end{align*}
    where Young's convolution inequality was used in the last step.

    For \eqref{lem5.1(ii)} we argue analogously, namely
    \begin{align*}
    \Big( &\sum_{\bj\in\N^{d}_{0}} 2^{\|\bj/\balpha\|_\infty sq} \Big\| \sum_{\bk\in\Z^d} 2^{\|\bj\|_1} \langle f,h_{\bj,\bk} \rangle \chi_{\bj,\bk}(\cdot)  \Big\|_p^q \Big)^{1/q}  \\
    &\lesssim  \Big( \sum_{\bj\in\N^{d}_{0}} 2^{\|\bj/\balpha\|_\infty sq} \Big\| \sum_{\bl\in\Z^{d}} A(\bl,a) P_{2^{\bj+\bl},a}(\Phi_{\bj+\bl}\ast f)(\cdot)  \Big\|_p^q \Big)^{1/q}  \\
    &\le  \Big( \sum_{\bj\in\N^{d}_{0}}    \Big\| \sum_{\bl\in\Z^{d}} A(\bl,a) 2^{\|\bl/\balpha\|_\infty|s|}  2^{\|(\bj+\bl)/\balpha\|_\infty s}  P_{2^{\bj+\bl},a}(\Phi_{\bj+\bl}\ast f)(\cdot)  \Big\|_p^q \Big)^{1/q} \\
    &\le  \sum_{\bl\in\Z^{d}} A(\bl,a) 2^{\|\bl/\balpha\|_\infty|s|} \cdot \Big( \sum_{\bj\in\N^{d}_{0}}    \Big\|    2^{\|\bj/\balpha\|_\infty s}  P_{2^{\bj},a}(\Phi_{\bj}\ast f)(\cdot)  \Big\|_p^q \Big)^{1/q}.
    \end{align*}

    Choosing $a>\max\{1/p,1/q\}$ in the F-case and $a>1/p$ in the B-case, to ensure the boundedness of the Peetre maximal operator (see Theorem~\ref{peetremax}, also compare e.g.~\cite[Thm.~2.6]{Ul10}), as well as $|s|<\min\limits_{i\in\{1,\ldots,d\}}\{\alpha_i\}(1-a)$
    we get
    \begin{align*}
    \sum_{\bl\in\Z^{d}} A(\bl,a) 2^{\|\bl/\balpha\|_\infty |s|} <\infty
    \end{align*}
    and thus \eqref{lem5.1(i)} and \eqref{lem5.1(ii)}, respectively.
\end{proof}

Using a duality argument, we can deduce an immediate companion result.

\begin{proposition}\label{lemcompanion}
	Let $1<p,q<\infty$, $s\in\R$, and $\balpha=(\alpha_1,\ldots,\alpha_d)>0$ such that $\sum_{i=1}^{d}\alpha_i=d$.
	Under the condition
	\begin{align*}
	|s|/\alpha_{\text{min}}<
	 \min\Big\{\frac{1}{p},\frac{1}{q}\Big\}
	\end{align*}
	we have for all $f\in S^\prime(\R^d)$ (with the dual pairing $\langle  f,h_{\bj,\bk} \rangle$ defined as in \eqref{f3} in Remark~\ref{dualp})
	\begin{align}\label{duality}
	\|f\|_{\F^{s,\balpha}_{p,q}} \lesssim
	\Big\| \Big( \sum_{\bj\in\N^{d}_{0}} 2^{\|\bj/\balpha\|_\infty sq} \Big| \sum_{\bk\in\Z^d} 2^{\|\bj\|_1} \langle f,h_{\bj,\bk} \rangle \chi_{\bj,\bk}(x) \Big|^q \Big)^{1/q} \Big\|_p \,,
	\end{align}
	whenever the right-hand side is defined. In case \: $|s|/\alpha_{\text{min}}<1/p$ \: we have
		\begin{align}\label{duality2}
	\|f\|_{\B^{s,\balpha}_{p,q}} \lesssim
	\Big( \sum_{\bj\in\N^{d}_{0}} 2^{\|\bj/\balpha\|_\infty sq} \Big\| \sum_{\bk\in\Z^d} 2^{\|\bj\|_1} \langle f,h_{\bj,\bk} \rangle \chi_{\bj,\bk}(\cdot) \Big\|_p^q \Big)^{1/q} \,.
	\end{align}
\end{proposition}

\begin{proof}
We showed in Proposition~\ref{lem5.1} (i) that the linear operator
\begin{align*}
A:~&\F^{s,\balpha}_{p,q}\to \tilde{f}^{s,\balpha}_{p,q} \,,\quad f \mapsto ( 2^{\|\bj\|_1} \langle f,h_{\bj,\bk}\rangle)_{\bj,\bk} \,,
\end{align*}
is well-defined and bounded in the parameter range
\[
|s|<\min\limits_{i\in\{1,\ldots,d\}}\{\alpha_i\} \min\Big\{1-\frac{1}{p},1-\frac{1}{q}\Big\}.
\]

Consequently, in this range, the dual operator
\[
A': \big(\tilde{f}^{s,\balpha}_{p,q}\big)^\prime \to \big(\F^{s,\balpha}_{p,q} \big)^\prime
\]
is also well-defined and bounded. Identifying
$\big(\tilde{f}^{s,\balpha}_{p,q}\big)^\prime$ with $\tilde{f}^{-s,\balpha}_{p^{\prime},q^{\prime}}$ with respect to the
non-standard duality product
\begin{align}\label{eqdef:nonstdprod}
\langle (\lambda_{\bj,\bk}), (\mu_{\bj,\bk}) \rangle := \sum_{\bj\in\N_0^{d}} 2^{-\|\bj\|_1} \sum_{\bk\in\Z^{d}} \lambda_{\bj,\bk}\mu_{\bj,\bk} \,,
\end{align}
which is possible according to Theorem~\ref{theodual} (i) below, it can be represented in the form
\begin{equation}\nonumber
A':(\lambda_{\bj,\bk})_{\bj,\bk} \mapsto \sum\limits_{\bj,\bk}  \lambda_{\bj,\bk} h_{\bj,\bk}\,,
\end{equation}
where the convergence is weak*ly in $\big(\F^{s,\balpha}_{p,q}\big)^\prime$.
This is a consequence of the relation
\begin{equation}\nonumber
\begin{split}
\langle Af, g'\rangle_{Y\times Y'} &= \sum\limits_{\bj\in\N_0^{d}} 2^{-\|\bj\|_1}  \sum\limits_{\bk\in\Z^{d}} 2^{\|\bj\|_1}\langle f,h_{\bj,\bk}\rangle\cdot \lambda_{\bj,\bk}\\
&= \int_{\R^{d}} f(x) \cdot \Big(\sum\limits_{\bj\in\N_0^{d}}\sum\limits_{\bk\in\Z^{d}}  \lambda_{\bj,\bk} h_{\bj,\bk}(x)\Big) \,dx\\
&= \langle f, A'g'\rangle_{X\times X'}\,,
\end{split}
\end{equation}
where we used the short-hand notation $Y=\tilde{f}^{s,\balpha}_{p,q}$ 
and $X=\F^{s,\balpha}_{p,q}$.

Invoking Theorem~\ref{theodual}(i) another time, $X'$ can be identified with $\F^{-s,\balpha}_ {p',q'}$
(and $X''$ with $\F^{s,\balpha}_ {p,q}$). Then
for $(\lambda_{\bj,\bk})_{\bj,\bk}\in\tilde{f}^{-s,\balpha}_{p',q'}$ and some enumeration
\[
(\lambda_{\bj,\bk})_{\bj,\bk} \quad\rightsquigarrow\quad (\lambda_n)_{n\in\N} 
\]
we have $\|(\lambda_n)_{n\ge N}\|_{\tilde{f}^{-s,\balpha}_{p',q'}}\to 0$
for $N\to\infty$. We estimate
\begin{align*}
\| A^\prime\big( (\lambda_n)_{n\ge N} \big) \|_{\F^{-s,\balpha}_{p',q'}}
&= \sup_{\|f\|_{\F^{s,\balpha}_{p,q}}=1} |\langle A'\big( (\lambda_n)_{n\ge N} \big), f \rangle|
= \sup_{\|f\|_{\F^{s,\balpha}_{p,q}}=1} |\langle  (\lambda_n)_{n\ge N}, Af \rangle| \\
&\le \sup_{\|f\|_{\F^{s,\balpha}_{p,q}}=1}  \|(\lambda_n)_{n\ge N} \|_{\tilde{f}^{-s,\balpha}_{p',q'}}  \| Af \|_{\tilde{f}^{s,\balpha}_{p,q}} \lesssim  \|(\lambda_n)_{n\ge N} \|_{\tilde{f}^{-s,\balpha}_{p',q'}} \to 0 \quad(N\to\infty).
\end{align*}
Hence, $ A^\prime\big( (\lambda_n)_{n\le N}\big) \to A^\prime\big( (\lambda_n)_{n\in\N}\big)$ strongly and unconditionally in $\F^{-s,\balpha}_{p',q'}$.

In other words, we have shown that if $|s|<\min\limits_{i\in\{1,\ldots,d\}}\{\alpha_i\} \min\{1/p,1/q\}$ and $(\lambda_{\bj,\bk})_{\bj,\bk}\in\tilde{f}^{s,\balpha}_{p,q}$ then
$\sum\limits_{\bj,\bk} \lambda_{\bj,\bk}h_{\bj,\bk}$ converges strongly and unconditionally in $\F^{s,\balpha}_{p,q}$ and
\begin{align}\label{boundedness2}
\Big\|\sum\limits_{\bj,\bk}  \lambda_{\bj,\bk}h_{\bj,\bk}\Big\|_{\F^{s,\balpha}_{p,q}} \lesssim \|(\lambda_{\bj,\bk})_{\bj,\bk}\|_{\tilde{f}^{s,\balpha}_{p,q}}\,.
\end{align}

Hence, choosing $\lambda_{\bj,\bk}:= 2^{\|\bj\|_1}\langle f,h_{\bj,\bk}\rangle$ and assuming a finite sequence norm $\|(\lambda_{\bj,\bk})_{\bj,\bk}\|_{\tilde{f}^{s,\balpha}_{p,q}}$, then
$\sum_{\bj,\bk} \lambda_{\bj,\bk}h_{\bj,\bk}=\sum_{\bj,\bk} 2^{\|\bj\|_1}\langle f,h_{\bj,\bk}\rangle h_{\bj,\bk}$ converges strongly to some limit $\tilde{f}\in\F^{s,\balpha}_{p,q}(\R^{d})$.
Since this sum also converges (weak*ly) in $S'(\R)$ to $f$ we have
$f = \tilde{f} = \sum_{\bj,\bk} 2^{\|\bj\|_1}\langle f,h_{\bj,\bk}\rangle h_{\bj,\bk}$ in $\F^{s,\balpha}_{p,q}$.
Now \eqref{duality} follows from \eqref{boundedness2}.

The proof of \eqref{duality2} works the same, using Proposition~\ref{lem5.1} (ii) and Theorem~\ref{theodual} (ii).
\end{proof}

Combining both, Proposition~\ref{lem5.1} and Proposition~\ref{lemcompanion}, we arrive at the following
proposition, whereby we now concentrate on the F-case. 

\begin{proposition}\label{prop_combo}
Let $1<p,q<\infty$, $s\in\R$, and $\balpha=(\alpha_1,\ldots,\alpha_d)>0$ such that $\sum_{i=1}^{d}\alpha_i=d$.
Under the condition
\begin{align*}
|s|/\alpha_{\text{min}}<
 \min\Big\{\frac{1}{p},\frac{1}{q},1-\frac{1}{p},1-\frac{1}{q}\Big\}.
\end{align*}
we have for all $f\in S^\prime(\R^d)$ (with the dual pairing $\langle  f,h_{\bj,\bk} \rangle$ defined as in \eqref{f3} in Remark~\ref{dualp})
\begin{align}\label{equivHaar}
\Big\| \Big( \sum_{\bj\in\N^{d}_{0}} 2^{\|\bj/\balpha\|_\infty sq} \Big| \sum_{\bk\in\Z^d} 2^{\|\bj\|_1} \langle f,h_{\bj,\bk} \rangle \chi_{\bj,\bk}(x) \Big|^q \Big)^{1/q} \Big\|_p \asymp \|f\|_{\F^{s,\balpha}_{p,q}} \,,
\end{align}
whenever the left-hand side is defined.
\end{proposition}

As a direct consequence of this result, we can finally formulate the main theorem of this section
which corresponds to Theorem~\ref{main_wav}.

\begin{theorem}\label{thm:haarmain}
	Let $1<p,q<\infty$, $s\in\R$, and $\balpha=(\alpha_1,\ldots,\alpha_d)>0$ such that $\sum_{i=1}^{d}\alpha_i=d$.
	Further, assume
	\begin{align}\label{scond}
	|s|/\alpha_{\text{min}}<
	 \min\Big\{\frac{1}{p},\frac{1}{q},1-\frac{1}{p},1-\frac{1}{q}\Big\}.
	\end{align}
	Then the Haar system $\mathcal{H}_{d}=(h_{\bj,\bk})_{\bj,\bk}$ defined in \eqref{HaarSd} constitutes an unconditional basis of $\F^{s,\balpha}_{p,q}(\R^d)$ with associated sequence space
	$\tilde{f}^{s,\balpha}_{p,q}$.
	The unique sequence of basis coefficients for $f\in\F^{s,\balpha}_{p,q}(\R^d)$
	is determined by $\lambda:=\lambda(f)=(\lambda_{\bj,\bk})_{\bj,\bk}$ with
	\begin{equation}\label{coeffHaar}
	\lambda_{\bj,\bk}:=\lambda_{\bj,\bk}(f) = 2^{\|\bj\|_1}\langle f , h_{\bj,\bk}\rangle \,.
	\end{equation}
	Further, we have the wavelet isomorphism (equivalent norm)
	$$
	\|f\|_{\F^{s,\balpha}_{p,q}(\R^d)} \asymp \|\lambda(f)\|_{\tilde{f}^{s,\balpha}_{p,q}}\quad,\quad f\in
	\F^{s,\balpha}_{p,q}(\R^d)\,.
	$$
	
	In addition, we can use $\mathcal{H}_{d}$ to distinguish those elements of
	$\mathcal{S}'(\R^d)$ that belong to $\F^{s,\balpha}_{p,q}(\R^d)$.
	Those are characterized by either of the following two criteria:
	
	\begin{enumerate}[(i)]
		\item  
	    $f$ can be represented as a sum
		\begin{equation}\label{sumHaar}
		f = \sum\limits_{\bj\in \N_0^d} \sum\limits_{\bk\in \Z^d} \lambda_{\bj,\bk} h_{\bj,\bk} \quad\text{converging (weak*ly) in $\mathcal{S}'(\R^d)$}
		\end{equation}
		with coefficients $(\lambda_{\bj,\bk})_{\bj,\bk} \in \tilde{f}^{s,\balpha}_{p,q}$ (with respect to some chosen ordering).
		\item 
		With $\lambda(f)$ being defined as in \eqref{coeffHaar}, it holds
		\begin{equation*}
		\lambda(f)=(\lambda_{\bj,\bk})_{\bj,\bk} \in\tilde{f}^{s,\balpha}_{p,q}.
		\end{equation*}
	\end{enumerate}
	In both cases, the sequence $(\lambda_{\bj,\bk})_{\bj,\bk}$ is necessarily the sequence of basis coefficients and the representation \eqref{sumHaar} converges
	unconditionally to $f$ in $\F^{s,\balpha}_{p,q}(\R^d)$.
\end{theorem}
\begin{proof}
As a direct consequence of the equivalence~\eqref{equivHaar}
proved in Proposition~\ref{prop_combo} the analysis operator
\[
\lambda : f \mapsto \big(  2^{\|\bj\|_1}\langle f, h_{\bj,\bk}\rangle\big)_{\bj,\bk}
\]
is well-defined and bounded from $\F^{s,\balpha}_{p,q}$ to $\tilde{f}^{s,\balpha}_{p,q}$. Moreover, it is injective and
we have the equivalence of norms $\|f\|_{\F^{s,\balpha}_{p,q}}\asymp\|\lambda(f)\|_{\tilde{f}^{s,\balpha}_{p,q}}$.
Further, for
$f\in\mathcal{S}'(\R^d)$ we have $f\in\F^{s,\balpha}_{p,q}$
if and only if $\lambda(f)\in\tilde{f}^{s,\balpha}_{p,q}$ (whenever $\lambda(f)$ is defined).

Now, let us have a look at the synthesis operator
\begin{align}\label{synth-Op}
S:(\lambda_{\bj,\bk})_{\bj,\bk}\mapsto \sum_{\bj,\bk} \lambda_{\bj,\bk}h_{\bj,\bk}.
\end{align}
Clearly, for every finite sequence the assignment $S$ defines an element in $\F^{s,\balpha}_{p,q}$. By completion, using \eqref{equivHaar} and the fact that the finite sequences lie dense
in $\tilde{f}^{s,\balpha}_{p,q}$, this synthesis further extends to all sequences of $\tilde{f}^{s,\balpha}_{p,q}$, with unconditional and strong convergence of \eqref{synth-Op} in $\F^{s,\balpha}_{p,q}$.
Hence, $S$ is a well-defined bounded linear operator from $\tilde{f}^{s,\balpha}_{p,q}$ to $\F^{s,\balpha}_{p,q}$, which, as another consequence of \eqref{equivHaar}, is also injective.

Next, turning to the composition $\lambda\circ S$ operating from $\tilde{f}^{s,\balpha}_{p,q}$ to $\tilde{f}^{s,\balpha}_{p,q}$, we deduce for each fixed $\bj_\ast\in\N_0^d$ and $\bk_\ast\in\Z^d$
\[
\Big\langle \sum_{\bj,\bk} 2^{\|\bj\|_1} \lambda_{\bj,\bk} h_{\bj,\bk} , h_{\bj_\ast,\bk_\ast} \Big\rangle = \sum_{\bj,\bk} 2^{\|\bj\|_1} \lambda_{\bj,\bk} \big\langle   h_{\bj,\bk} , h_{\bj_\ast,\bk_\ast} \big\rangle = \lambda_{\bj_\ast,\bk_\ast} \,,
\]
using the 
orthogonality of the system $(h_{\bj,\bk})_{\bj,\bk}$ in $L_2(\R^d)$. We obtain $Id_{\tilde{f}^{s,\balpha}_{p,q}}=\lambda\circ S$
and in turn $\lambda\circ S\circ \lambda=\lambda$. 
Due to the injectivity of $\lambda$, the latter equality further implies $Id_{\F^{s,\balpha}_{p,q}}= S\circ \lambda$.
In particular, $\lambda$ and $S$ are thus bijections and every $f\in\F^{s,\balpha}_{p,q}$ allows for a representation~\eqref{sumHaar}.

To see that the representing coefficients $(\lambda_{\bj,\bk})_{\bj,\bk}$ are unique, under the assumption of strong convergence of the sum, 
let $\lambda^\ast=(\lambda^\ast_{\bj,\bk})_{\bj,\bk}$ be some sequence which satisfies \eqref{sumHaar} in a strong sense for some special ordering of the sum.
Then again \eqref{equivHaar} together with a completion argument yields $\lambda^\ast\in\tilde{f}^{s,\balpha}_{p,q}$, and thus
$\lambda^\ast=\lambda(f)$ by the injectivity of $S$. Hence, the expansion coefficients in \eqref{sumHaar} are unique and it follows that $\mathcal{H}_{d}$ is a basis. Its unconditionality is due to the fact that the convergence 
of \eqref{sumHaar} for sequences $\lambda^\ast\in\tilde{f}^{s,\balpha}_{p,q}$ is always unconditional.

For the proof of criterion (i) we just remark that for sequences $(\lambda_{\bj,\bk})_{\bj,\bk}$ in $\tilde{f}^{s,\balpha}_{p,q}$ with
weak*-convergence of \eqref{sumHaar} in $\mathcal{S}^\prime(\R^d)$
the convergence is automatically in the stronger sense of $\F^{s,\balpha}_{p,q}$.
\end{proof}


\begin{remark}
For brevity, the above theorem was only stated for the F-case.
There also exists a B-version, which reads precisely the same apart from condition~\eqref{scond} which is replaced by
\begin{align*}
|s|/\alpha_{\text{min}}<
\min\Big\{\frac{1}{p},1-\frac{1}{p}\Big\}.
\end{align*}
\end{remark}

In the proof of Proposition~\ref{lemcompanion} we utilized isomorphisms $\big(\F^{s,\balpha}_{p,q}\big)^\prime \cong \F^{-s,\balpha}_{p^\prime,q^\prime}$
and $\big(\tilde{f}^{s,\balpha}_{p,q}\big)^\prime \cong \tilde{f}^{-s,\balpha}_{p^\prime,q^\prime}$ as well as
$\big(\B^{s,\balpha}_{p,q}\big)^\prime \cong \B^{-s,\balpha}_{p^\prime,q^\prime}$
and $\big(\tilde{b}^{s,\balpha}_{p,q}\big)^\prime \cong \tilde{b}^{-s,\balpha}_{p^\prime,q^\prime}$.
Hence, for the completeness of our exposition, it remains to establish those.

\begin{theorem}\label{theodual}
Let $1<p,q<\infty$, $s\in\R$, and $\balpha=(\alpha_1,\ldots,\alpha_d)>0$ with $\sum_{i=1}^{d}\alpha_i=d$. Then
\begin{align*}
\Big( \F^{s,\balpha}_{p,q}(\R^d)  \Big)^\prime \cong \F^{-s,\balpha}_{p^\prime,q^\prime}(\R^d)
\quad\text{and}\quad
\Big( \tilde{f}^{s,\balpha}_{p,q} \Big)^\prime \cong \tilde{f}^{-s,\balpha}_{p^\prime,q^\prime} \,,
\end{align*}	
whereby the second of these isomorphies has to be understood with respect to the non-standard pairing~\eqref{eqdef:nonstdprod}. In the Besov case, we have the analogous relations
\begin{align*}
\Big( \B^{s,\balpha}_{p,q}(\R^d)  \Big)^\prime \cong \B^{-s,\balpha}_{p^\prime,q^\prime}(\R^d)
\quad\text{and}\quad
\Big( \tilde{b}^{s,\balpha}_{p,q} \Big)^\prime \cong \tilde{b}^{-s,\balpha}_{p^\prime,q^\prime} \,.
\end{align*}
\end{theorem}

The proof of this theorem 
is based on two auxiliary propositions. The first one of these is stated without proof, since it is a straightforward generalization
of the classic identifications $\big(L_{p}(\R^d)\big)^\prime\cong L_{p^\prime}(\R^d)$ and $\big(\ell_{p}(I)\big)^\prime\cong \ell_{p^\prime}(I)$ when $1<p<\infty$ (see e.g.~\cite{edwards1965functional}).
\begin{proposition}\label{proplqLp}
Let $I$ be an arbitrary countable index set. Then
\begin{align*}
\Big(\ell_q\big(I,L_p(\R^d)\big)\Big)^\prime \cong \ell_{q^\prime}\big(I,L_{p^\prime}(\R^d)\big) \quad\text{and}\quad \Big(L_p\big(\R^d,\ell_q(I)\big)\Big)^\prime \cong L_{p^\prime}\big(\R^d,\ell_{q^\prime}(I)\big)
\end{align*}
in the sense that there exist isomorphisms $f\mapsto (f_i)_{i\in I}$ such that
\[
\langle f, g \rangle_{Y^\prime\times Y} = \sum_{i\in I} \int_{\R^d} f_i(x)g_i(x)  \,dx
\]
for the respective cases $Y=\ell_q\big(I,L_p(\R^d)\big)$ and $Y=L_p\big(\R^d,\ell_q(I)\big)$.
\end{proposition}

The second proposition provides an alternative way to characterize functions in $\F^{s,\balpha}_{p,q}(\R^d)$ and $\B^{s,\balpha}_{p,q}(\R^d)$.
Its counterpart in the classical setting of Triebel-Lizorkin spaces is Proposition~1 in~\cite[Sec.~2.3.4]{triebel2010theory}.

\begin{proposition}\label{propCharAlternative}
Assume $1<p,q<\infty$, $s\in\R$, $\balpha=(\alpha_1,\ldots,\alpha_d)>0$, and $\sum_{i=1}^{d} \alpha_i=d$.
\begin{enumerate}[(i)]
\item
Then $f\in S^\prime(\R^d)$ belongs to $\F^{s,\balpha}_{p,q}(\R^d)$
if and only if there exists a family $\{f_{\bj}\}_{\bj\in\N_0^d}\subset L_p(\R^d)$  such that
\[
f=\sum_{\bj} \Delta_{\bj}f_{\bj} \text{ in }S^\prime(\R^d) \quad\text{ and }\quad \big\|  2^{\|\bj/\balpha\|_\infty s}f_{\bj} \big\|_{L_p(\R^d,\ell_q)}  <\infty \,.
\]
\item
Then $f\in S^\prime(\R^d)$ belongs to $\B^{s,\balpha}_{p,q}(\R^d)$
if and only if there exists a family $\{f_{\bj}\}_{\bj\in\N_0^d}\subset L_p(\R^d)$  such that
\begin{align}\label{llllll}
f=\sum_{\bj} \Delta_{\bj}f_{\bj} \text{ in }S^\prime(\R^d) \quad\text{ and }\quad \big\|  2^{\|\bj/\balpha\|_\infty s}f_{\bj} \big\|_{\ell_q(L_{p}(\R^d))}  <\infty \,.
\end{align}
\end{enumerate}
\end{proposition}
\begin{proof}
(i) Adaption of proof of Proposition~1 in~\cite[Sec.~2.3.4]{triebel2010theory}.

(ii) Let $(\phi_{\bj})_{\bj}$ be a hyperbolic resolution of unity as introduced in Definition~\ref{def:hypLP} with associated hyperbolic Littlewood-Paley analysis $ (\Delta_{\bj})_{\bj}$. Further, take $f\in \B^{s,\balpha}_{p,q}$ and put $f_{\bj}:=\sum\limits_{\br\in\{-1,0,1\}^{d}}
\mathcal{F}^{-1}\phi_{\bj+\br}\mathcal{F}f$ for $\bj\in\N_0^d$, whereby we let $\phi_{-1}:=0$.
Then $\{f_{\bj}\}_{\bj}\subset L_p$ and
\begin{align*}
\sum_{\bj\in\N_0^d} \mathcal{F}^{-1}\phi_{\bj}\mathcal{F}f_{\bj}
= \sum_{\bj\in\N_0^d} \mathcal{F}^{-1}\phi_{\bj} \Big(\sum\limits_{\br\in\{-1,0,1\}^{d}}\phi_{\bj+\br}\Big)\mathcal{F}f
= \sum_{\bj\in\N_0^d} \mathcal{F}^{-1}\phi_{\bj}\mathcal{F}f
= \sum_{\bj\in\N_0^d} \Delta_{\bj}f = f
\end{align*}
as well as
\begin{align*}
\big\| 2^{\|\bj/\balpha\|_\infty s} f_{\bj} \big\|_{\ell_q(L_p)}
&= \Big\| 2^{\|\bj/\balpha\|_\infty s} \sum\limits_{\br\in\{-1,0,1\}^{d}} 
\mathcal{F}^{-1}\phi_{\bj+\br}\mathcal{F}f \Big\|_{\ell_q(L_p)} \\
&\lesssim  \sum\limits_{\br\in\{-1,0,1\}^{d}} \Big\| 2^{\|\bj/\balpha\|_\infty s} 
\mathcal{F}^{-1}\phi_{\bj+\br}\mathcal{F}f \Big\|_{\ell_q(L_p)} 
\asymp \|f\|_{\B^{s,\balpha}_{p,q}} < \infty \,.
\end{align*}

This settles one direction of the assertion.
For the other direction, let $f\in S^\prime(\R^d)$ satisfy \eqref{llllll}
with associated $\{f_{\bj}\}_{\bj} \subset L_p $.
In view of $f=\sum_{\bj} \Delta_{\bj}f_{\bj}$, we can estimate
\begin{align*}
\big\|&2^{\|\bj/\balpha\|_\infty s} \Delta_{\bj}f \big\|_{\ell_q(L_p)} 
= \Big\|2^{\|\bj/\balpha\|_\infty s} \Delta_{\bj} \sum_{\br\in\{-1,0,1\}^{d}} \Delta_{\bj+\br}f_{\bj+\br} \Big\|_{\ell_q(L_p)}  \\
&\le \sum_{\br\in\{-1,0,1\}^{d}} \big\| 2^{\|\bj/\balpha\|_\infty s}
\Delta_{\bj}\Delta_{\bj+\br} f_{\bj+\br} \big\|_{\ell_q(L_p)}   
\lesssim \sum_{\br\in\{-1,0,1\}^{d}}  \big\|2^{\|\bj/\balpha\|_\infty s}  \Delta_{\bj+\br}f_{\bj+\br} \big\|_{\ell_q(L_p)}  \\
&\asymp  \big\|2^{\|\bj/\balpha\|_\infty s} \Delta_{\bj}f_{\bj} \big\|_{\ell_q(L_p)} 
\lesssim  \big\|2^{\|\bj/\balpha\|_\infty s} f_{\bj} \big\|_{\ell_q(L_p)} <\infty  \,,
\end{align*}
where the last two lines are due to the H\"ormander-Mikhlin multiplier theorem, which is applied twice.
This estimate shows $f\in \B^{s,\balpha}_{p,q}$, finishing the proof.
\end{proof}

Now we are ready to give a thorough proof of the duality relations stated in Theorem~\ref{theodual}.

\begin{proof}[of Theorem~\ref{theodual}]
We restrict to the F-case and begin with the more involved relation $\big(\F^{s,\balpha}_{p,q}\big)^\prime \cong \F^{-s,\balpha}_{p^\prime,q^\prime}$.
The subsequent proof is thereby an adaption of the proof of the classical theorem in~\cite[Sec.~2.11.2]{triebel2010theory} to the setting of hyperbolic spaces.

It is essential to note that, since $S(\R^d)$ lies dense in $\F^{s,\balpha}_{p,q}$, there is a natural embedding
\begin{align}\label{eq:natemb}
\kappa~:~ \Big(\F^{s,\balpha}_{p,q}\Big)^{\prime}  \hookrightarrow S^\prime(\R^d) \,.
\end{align}
Hence, both $\big(\F^{s,\balpha}_{p,q}\big)^\prime$ and $\F^{-s,\balpha}_{p^\prime,q^\prime}$
can be interpreted as subspaces of $S^\prime(\R^d)$, simplifying the following considerations.

Let us first assume that $f\in S^\prime(\R^d)$ is an element of $\F^{-s,\balpha}_{p^\prime,q^\prime}$ and take $\Phi_{\bl}$ and $\Lambda_{\bl}$ as in the proof of Proposition~\ref{lem5.1}, for instance.
Then $f$ defines an element of $\big(\F^{s,\balpha}_{p,q}\big)^\prime$ via
\begin{align}\label{eqdef:dualprod}
\langle f, g\rangle_{\ast}:= \sum_{\bl\in\Z^d} \langle \Phi_{\bl}\ast f, \overline{\Lambda_{\bl}\ast g}\rangle \,,\quad\text{where $g\in\F^{s,\balpha}_{p,q}$}\,,
\end{align}
as can be seen by the following estimate,
\begin{align*}
|\langle f, g\rangle_{\ast}|&=\Big|\sum_{\bl\in\Z^d} \langle \Phi_{\bl}\ast f, \overline{\Lambda_{\bl}\ast g}\rangle\Big|
= \Big|\sum_{\bl\in\Z^d} \int_{\R^d} (\Phi_{\bl}\ast f)(y) \cdot (\Lambda_{\bl}\ast g)(y) \,dy \Big| \\
&= \Big|\int_{\R^d} \sum_{\bj\in\N_0^d}  (\Phi_{\bj}\ast f)(y) \cdot (\Lambda_{\bj}\ast g)(y) \,dy \Big| \\
&\le \int_{\R^d}  \Big(\sum_{\bj\in\N_0^d} 2^{-\|\bj/\balpha\|_\infty sq^\prime} |\Phi_{\bj}\ast f(y)|^{q^\prime}\Big)^{1/q^\prime} \Big(\sum_{\bj\in\N_0^d} 2^{\|\bj/\balpha\|_\infty sq} |\Lambda_{\bj}\ast g(y)|^{q}\Big)^{1/q} \,dy \\
&\le   \Big( \int_{\R^d} \Big(\sum_{\bj\in\N_0^d} 2^{-\|\bj/\balpha\|_\infty sq^\prime} |\Phi_{\bj}\ast f(y)|^{q^\prime}\Big)^{p^\prime/q^\prime} \,dy \Big)^{1/p^\prime} \\
&\qquad\qquad\qquad \times\Big( \int_{\R^d} \Big(\sum_{\bj\in\N_0^d} 2^{\|\bj/\balpha\|_\infty sq} |\Lambda_{\bj}\ast g(y)|^{q}\Big)^{p/q} \,dy \Big)^{1/p} \\
&\lesssim  \| f \|_{\F^{-s,\balpha}_{p^\prime,q^\prime}}  \cdot 
\| g \|_{\F^{s,\balpha}_{p,q}}  \,.
\end{align*}
Hereby, we applied H\"older's inequality and used $\Phi_{\bl}=\Lambda_{\bl}=0$ if $\bl\notin\N_0^d$. 

The duality product $\langle \cdot,\cdot \rangle_{\ast}$ thus yields an embedding $\iota:\F^{-s,\balpha}_{p^\prime,q^\prime} \to \big(\F^{s,\balpha}_{p,q}\big)^\prime$. Further, we have natural embeddings
$\nu: \F^{-s,\balpha}_{p^\prime,q^\prime} \hookrightarrow S^\prime(\R^d)$ and $\kappa:\big(\F^{s,\balpha}_{p,q}\big)^\prime\hookrightarrow S^\prime(\R^d)$ (see~\eqref{eq:natemb}).
To establish a bridge between $\kappa$, $\iota$, and $\nu$,
we now consider the special case of a Schwartz function $g=\phi\in S(\R^d)$ in~\eqref{eqdef:dualprod}.
We obtain
\[
\langle f, \phi\rangle_{\ast} = \sum_{\bl\in\Z^d} \langle \Phi_{\bl}\ast f, \Lambda_{\bl}\ast \phi\rangle_{S^\prime\times S}
= \sum_{\bl\in\Z^d} \langle  \Phi_{\bl}\ast\Lambda_{\bl}\ast f,  \phi\rangle_{S^\prime\times S}= \langle f, \phi\rangle_{S^\prime\times S}\,.
\]
Hence, the above (somewhat artificially) defined operation of $f$ on $\F^{s,\balpha}_{p,q}$ via $\langle \cdot,\cdot \rangle_{\ast}$
is compatible with the operation of $f$ as an element of $S^\prime(\R^d)$ on $S(\R^d)$.
This proves $\nu=\kappa\circ\iota$ and thus $\F^{-s,\balpha}_{p^\prime,q^\prime}\subset \big(\F^{s,\balpha}_{p,q}\big)^\prime$, considered as subsets of $S^\prime(\R^d)$.

It remains to prove the converse inclusion $\big(\F^{s,\balpha}_{p,q}\big)^\prime \subset \F^{-s,\balpha}_{p^\prime,q^\prime}$.
For this, let $f\in S^\prime(\R^d)$ be an element of $\big(\F^{s,\balpha}_{p,q}\big)^\prime$.
We will show that this implies 
$f\in \F^{-s,\balpha}_{p^\prime,q^\prime}$ and to this end start with a construction
of an isometric embedding
\begin{align}\label{isoemb}
\mu:~ \Big(\F^{s,\balpha}_{p,q}\Big)^\prime \to L_{p^\prime}\big(\R^d,\ell_{q^\prime}\big)  \,,\quad f\mapsto (f_{\bj})_{\bj}.
\end{align}
Thereby, we build upon the observation that the assignment $g\mapsto ( 2^{|\bj/\balpha|_\infty s} \Delta_{\bj}g)_{\bj}$ maps $\F^{s,\balpha}_{p,q}$ isometrically to a closed subspace of $L_{p}\big(\R^d,\ell_{q}\big)$.
Via this assignment and the Hahn-Banach extension theorem, it is therefore possible to identify each functional $f\in\big(\F^{s,\balpha}_{p,q}\big)^\prime$ with a functional on
$L_{p}\big(\R^d,\ell_{q}\big)$ having the same norm.
Invoking Proposition~\ref{proplqLp} (i), this then yields an associated family $(f_{\bj})_{\bj}\in L_{p^\prime}\big(\R^d,\ell_{q^\prime}\big)$ with
$\|(f_{\bj})_{\bj} \|_{L_{p^\prime}(\R^d,\ell_{q^\prime})}=\|f\|_{(\F^{s,\balpha}_{p,q})^\prime}$ and $\langle f, g\rangle = \sum_{\bj} \langle f_{\bj}, 2^{|\bj/\balpha|_\infty s} \Delta_{\bj} g\rangle$, establishing \eqref{isoemb}.

In particular, for every $\phi\in S(\R^d)$
\begin{align*}
\langle f,\phi \rangle = \sum_{\bj} \langle f_{\bj}, 2^{\|\bj/\balpha\|_\infty s} \Delta_{\bj} \phi\rangle
= \sum_{\bj} \langle \Delta_{\bj} \tilde{f}_{\bj},  \phi\rangle \,,
\end{align*}
with $\tilde{f}_{\bj}:= 2^{\|\bj/\balpha\|_{\infty}s} f_{\bj}$. Hence, we have
\[
f = \sum_{\bj} 2^{\|\bj/\balpha\|_{\infty}s}  \Delta_{\bj}f_{\bj} = \sum_{\bj} \Delta_{\bj} \tilde{f}_{\bj} \quad\text{weak*ly in $S^\prime(\R^d)$.}
\]
Further, it holds $\| (2^{-\|\bj/\balpha\|_{\infty}s} \tilde{f}_{\bj})_{\bj\in\N_0^d} \|_{\ell_{q^\prime}( L_{p^\prime} )} = \|(f_{\bj})_{\bj\in\N_0^d} \|_{\ell_{q^\prime}( L_{p^\prime} )} 
= \|f\|_{( \F^{s,\balpha}_{p,q} )^\prime}$.
In view of Proposition~\ref{propCharAlternative} (i), this shows $f\in \F^{-s,\balpha}_{p^\prime,q^\prime}$ and finishes the proof of $\big(\F^{s,\balpha}_{p,q}\big)^\prime \cong \F^{-s,\balpha}_{p^\prime,q^\prime}$.




We next establish $\big( \tilde{f}^{s,\balpha}_{p,q} \big)^\prime \cong \tilde{f}^{-s,\balpha}_{p^\prime,q^\prime}$, which can be elegantly
done using the previous result together with the wavelet isomorphism
$\lambda:\F^{s,\balpha}_{p,q}\to\tilde{f}^{s,\balpha}_{p,q}$ established in Theorem~\ref{main_wav}. For this, we first verify that $\lambda$ preserves the duality structure of $\F^{-s,\balpha}_{p^\prime,q^\prime}\times\F^{s,\balpha}_{p,q}$. Indeed, for 
$f\in\F^{-s,\balpha}_{p^\prime,q^\prime}$ and $g\in\F^{s,\balpha}_{p,q}$ we have
\begin{align*}
\langle f,g\rangle_{\F^{-s,\balpha}_{p^\prime,q^\prime}\times\F^{s,\balpha}_{p,q}}
&= \Big\langle \sum_{\bj,\bk}  2^{\|\bj\|_1} \langle f,\psi_{\bj,\bk} \rangle \psi_{\bj,\bk} ,g \Big\rangle_{\F^{-s,\balpha}_{p^\prime,q^\prime}\times\F^{s,\balpha}_{p,q}}
 = \sum_{\bj,\bk} 2^{\|\bj\|_1} \langle f,\psi_{\bj,\bk} \rangle 
\langle  \psi_{\bj,\bk} ,g\rangle_{\F^{-s,\balpha}_{p^\prime,q^\prime}\times\F^{s,\balpha}_{p,q}} \\
&= \sum_{\bj,\bk} 2^{-\|\bj\|_1}  \big( 2^{\|\bj\|_1} \langle f,\psi_{\bj,\bk} \rangle \big) \big( 2^{\|\bj\|_1}
\langle g, \psi_{\bj,\bk} \rangle \big)
= \langle \lambda(f),\lambda(g)\rangle_{\tilde{f}^{-s,\balpha}_{p^\prime,q^\prime}\times\tilde{f}^{s,\balpha}_{p,q}}\,.
\end{align*}
Note that hereby we relied on the strong convergence of the wavelet expansion in the space $\F^{-s,\balpha}_{p^\prime,q^\prime}$.
Next we recall the isomorphism $\iota:\F^{-s,\balpha}_{p^\prime,q^\prime} \to \big(\F^{s,\balpha}_{p,q}\big)^\prime$ established above and let
$\lambda^{\prime}:\big(\tilde{f}^{s,\balpha}_{p,q}\big)^{\prime}\to \big(\F^{s,\balpha}_{p,q}\big)^{\prime}$ denote the dual map of $\lambda$, which is also an isomorphism. Then we can read  $\big( \tilde{f}^{s,\balpha}_{p,q} \big)^\prime \cong \tilde{f}^{-s,\balpha}_{p^\prime,q^\prime}$ directly from the following chain of isomorphisms 
\begin{align*}
\big(\tilde{f}^{s,\balpha}_{p,q}\big)^\prime \times \tilde{f}^{s,\balpha}_{p,q} \xrightarrow{\lambda^{\prime} \times \lambda^{-1}}
\big(\F^{s,\balpha}_{p,q}\big)^\prime \times \F^{s,\balpha}_{p,q} \xrightarrow{\iota^{-1} \times Id}
\F^{-s,\balpha}_{p^\prime,q^\prime} \times \F^{s,\balpha}_{p,q} 
\xrightarrow{\lambda \times \lambda} \tilde{f}^{-s,\balpha}_{p',q'} \times \tilde{f}^{s,\balpha}_{p,q}\,. 
\end{align*}
%
%
%
%
\end{proof}



\section{Hyperbolic and classical (anisotropic) Sobolev spaces}

In the remaining two sections 
we will analyze the relationship between the newly introduced hyperbolic scale of spaces $\widetilde{A}^{s,\balpha}_{p,q}(\R^d)$ from Section~\ref{section3}, where $A\in\{B,F\}$, and the classical scale of anisotropic spaces $A^{s,\balpha}_{p,q}(\R^d)$, which was recalled in Section~\ref{section2}. Our first result shows that, surprisingly, for Sobolev spaces (i.e.\ the case $A=F$, $1<p<\infty$, $q=2$)
both scales coincide.

\begin{theorem}\label{comp} Let $1<p<\infty$, $s\in \R$, and $\balpha>0$ be an anisotropy vector as above. Then
$$
  \widetilde{W}^{s,\balpha}_p(\R^d) = W^{s,\balpha}_p(\R^d) \qquad\text{(in the sense of equivalent norms).}
$$
\end{theorem}

\begin{proof}
The proof is divided into two steps. For convenience, we will thereby abbreviate by 
$m_{\balpha,s} :=  \Big(\sum_{i=1}^d(1 +\xi_i^2 )^{1/(2\alpha_i)}\Big)^{s}$ the function which appears in the definition~\eqref{cls_iso} of $\widetilde{W}^{s,\balpha}_p$.

{\em Step 1.} In the first step we prove $\Vert f \Vert_{W^{s,\balpha}_p} \lesssim \Vert f
\Vert_{\widetilde{W}^{s,\balpha}_p}$.  For $\bj=(j_1,...,j_d)\in\N_0^d$, let 
$(\varphi_{\bj})_{\bj}$ denote a fixed hyperbolic resolution of unity as introduced in Definition~\ref{def:hypLP}, with corresponding univariate family $(\varphi_{j})_{j}$ where $\supp(\varphi_0)\subset[-2,2]$. 
In addition, let us also construct a second hyperbolic 
resolution of unity $(\psi_{\bj})_{\bj}$ 
such that $\psi_{\bj} \varphi_{\bj} = \varphi_{\bj}$ for every $\bj\in\N_0^d$.
Hereby, it is not possible for $(\psi_{\bj})_{\bj}$ to obey the same strict building law as formulated in Definition~\ref{def:hypLP}. 
We define functions
\[
\psi^{\ast}_{0}:=\varphi_0 +\varphi_1 \,,\; \psi^{\ast}_{1}:=\varphi_0 +\varphi_1 + \varphi_2
\,,\; \psi^{\ast}_{2}:=\varphi_0 +\varphi_1 + \varphi_2 + \varphi_3 \,,\;  \psi^{\ast}_{j}:=\sum\limits_{r=-1}^{1}\varphi_{j+r} \quad (j\ge3) 
\]
and then put $\psi_{j}:=\psi^{\ast}_{j}/3$ for $j\in\N_0$. Then clearly $\sum_j \psi_j=1$ and $\psi_{j} \varphi_{j} = \varphi_{j}$. Finally, we set
\[
\psi_{\bj}:=\psi_{j_1}\otimes\cdots\otimes\psi_{j_d} 
\quad\text{to obtain}\quad (\psi_{\bj})_{\bj} \,.
\]
By construction, $\psi_0=\varphi_0(2^{-1}\cdot)/3$,   $\psi_{j+1}=\psi_{j}(2^{-1}\cdot)$ for $j\in\N_0\backslash\{2\}$,
and $\psi_3=\psi_0(2^{-3}\cdot)-\psi_0$.
As a consequence, we can record $\supp \psi_j(2^{j}\cdot)\subset [-4,4]$ for all $j\in\N_0$.
 
Utilizing $(\psi_{\bj})_{\bj}$, we then first rewrite the $W^{s,\balpha}_p$-norm as follows,
\begin{eqnarray*}
\Vert f \Vert_{W^{s,\balpha}_p} & \asymp & \Big\Vert {\mathcal{F}}^{-1} \big\lbrack m_{\balpha,s} {\mathcal{F}}f \big\rbrack \Big\Vert_p \\ 
& \asymp &
\Big\Vert \Big( \sum_{\bj \in {\mathbb{N}}_0^d} \big\vert {\mathcal{F}}^{-1} \varphi_{\bj} m_{\balpha,s} {\mathcal{F}}f \big\vert^2
\Big)^{\frac{1}{2}} \Big\Vert_p \\
& = & \Big\Vert \Big( \sum_{\bj \in {\mathbb{N}}_0^d} 
\big\vert 2^{s \|\bj/\balpha\|_{\infty}} {\mathcal{F}}^{-1} \big(2^{-s\|\bj/\balpha\|_{\infty}} \psi_{\bj} m_{\balpha,s} \big)
\varphi_{\bj} {\mathcal{F}}f \big\vert^2 \Big)^{\frac{1}{2}} \Big\Vert_p \,.
\end{eqnarray*}
After this, we denote $M_{\balpha,s,\bj} :=  2^{- s \|\bj/\balpha\|_{\infty}} \psi_{\bj} m_{\balpha,s}$ and apply the Fourier multiplier lemma~\ref{Fouriermultiplier} with $\rho_{\bj}:=M_{\balpha,s,\bj}$ and $f_{\bj}:=2^{s\|\bj/\balpha\|_{\infty}} \mathcal{F}^{-1} \varphi_{\bj} \mathcal{F}f$. Due to $p>1$ and $q=2$ we can thereby choose $r:=2$. This leads to
$$
\Big\Vert \Big( \sum_{\bj \in {\mathbb{N}}_0^d} 2^{2 s \|\bj/\balpha\|_{\infty}} \vert
{\mathcal{F}}^{-1} M_{\balpha,s,\bj} \varphi_{\bj}
{\mathcal{F}}f \vert^2 \Big)^{\frac{1}{2}} \Big\Vert_p  \lesssim  \sup_{{\bj} \in {\mathbb{N}}_0^d} \Vert
M_{\balpha,s,\bj}(2^{j_1}\cdot, \ldots, 2^{j_d}\cdot) \Vert_{S^2_2W}  \Vert f \Vert_{\widetilde{W}^s_p}.
$$
To finish Step 1, it now merely remains to check whether $\sup_{\bj \in {\mathbb{N}}_0^d} \Vert M_{\balpha,s,\bj}(2^{j_1}\cdot, \ldots,
2^{j_d}\cdot)\Vert_{S^2_2W} $ is finite. But, defining $\tilde{\psi}_0:=\psi_0$  and $\tilde{\psi}_1:=\psi_3(2^{3}\cdot)$ and letting $\tilde{\psi}_{\br}:=\tilde{\psi}_{r_1}\otimes\cdots\otimes\tilde{\psi}_{r_d}$ for $\br\in\{0,1\}^d$, we have
$$
M_{\balpha,s,\bj}(2^{j_1}\xi_1, \ldots, 2^{j_d}\xi_d) = 2^{- s\Vert \bj/\balpha \Vert_{\infty}} \tilde{\psi}_{s(\bj)}(\xi_1, \ldots, \xi_d) m_{\balpha,s}(2^{j_1}
\xi_1, \ldots, 2^{j_d}\xi_d) \,, 
$$
where $s(\bj):=(\sgn(j_1-2),\ldots,\sgn(j_d-2))_{+}$.
And for each $\bj\in\N_0^d$, the function $\tilde{\psi}_{s(\bj)}$ belongs to ${\mathcal{S}}(\R^d)$ and is supported on $[-4,4]^d$. Hence, it is sufficient to verify that the second derivatives (up to order 2) in every component of 
\[  
F_{\balpha,s,\bj}:\; (\xi_1,\ldots,\xi_d)\mapsto 2^{- s \|\bj/\balpha\|_{\infty}}
m_{\balpha,s}(2^{j_1} \xi_1, \ldots, 2^{j_d}\xi_d) 
= \Big( \sum\limits_{i=1}^{d} 2^{-\|\bj/\balpha\|_\infty} (1+ 2^{2j_i} \xi^{2}_i )^{1/(2\alpha_i)} \Big)^{s}
\] 
are uniformly bounded over $\bj\in\N_0^d$ and $\xi\in[-4,4]^d$. For this, we first observe that
\begin{align*}
F_{\balpha,s,\bj}(\xi) &=\Big( \sum\limits_{i=1}^{d} 2^{-\|\bj/\balpha\|_\infty} (1+ 2^{2j_i} \xi^{2}_i )^{1/(2\alpha_i)} \Big)^{s} \\
&=  
\Big( \sum_{i=1}^d \Big(2^{-2 \alpha_i \|\bj/\balpha\|_{\infty}} + 2^{-2 \alpha_i (\|\bj/\balpha\|_{\infty} - j_i /\alpha_i)} \xi_i^2 \Big)^{1/(2\alpha_i)} \Big)^s \,,
\end{align*}
with quantities $2^{-2 \alpha_i \|\bj/\balpha\|_{\infty}}$ and $2^{-2 \alpha_i (\|\bj/\balpha\|_{\infty} - j_i /\alpha_i)}$  all positive and never larger than one. This immediately implies 
\begin{align}\label{unifest}
\sup\limits_{\substack{\bj\in\N_0^d \\ \xi\in[-4,4]^d}} |F_{\balpha,s,\bj}(\xi)| \lesssim 1.
\end{align}

Next, we determine the partial derivatives of $m_{\balpha,s}$. They are given by
\begin{align*}
\partial_\ell m_{\balpha,s}(\xi) &= \frac{s}{\alpha_\ell} m_{\balpha,s-1}(\xi) \xi_\ell \langle \xi_\ell \rangle^{\frac{1}{\alpha_\ell}-2} \,, \\
\partial^2_\ell m_{\balpha,s}(\xi) &= \frac{s(s-1)}{\alpha_\ell^2} m_{\balpha,s-2}(\xi) \xi^2_\ell \langle \xi_\ell \rangle^{\frac{2}{\alpha_\ell}-4} + 
\frac{s(s-1)}{\alpha_\ell} m_{\balpha,s-1}(\xi) \langle \xi_\ell \rangle^{\frac{1}{\alpha_\ell}-4} \Big( (\frac{1}{\alpha_\ell}-1)\xi^2_\ell + 1 \Big) \,, 
\end{align*}
where we use the abbreviation $\langle\xi_\ell\rangle$ for $(1+\xi^2_\ell)^{1/2}$. We deduce the estimates
\begin{align*}
|\partial_\ell m_{\balpha,s}(\xi)| &\lesssim  |m_{\balpha,s-1}(\xi)|  \langle \xi_\ell \rangle^{\frac{1}{\alpha_\ell}-1} \,, \\
|\partial^2_\ell m_{\balpha,s}(\xi)| &\lesssim |m_{\balpha,s-2}(\xi)|  \langle \xi_\ell \rangle^{\frac{2}{\alpha_\ell}-2} + 
|m_{\balpha,s-1}(\xi)| \langle \xi_\ell \rangle^{\frac{1}{\alpha_\ell}-2}  \,,
\end{align*}
and thus obtain, using $\langle 2^{j_\ell}\xi_\ell \rangle \le 2^{j_\ell} \langle \xi_\ell \rangle$,
\begin{align*}
|\partial_\ell F_{\balpha,s,\bj}(\xi)| &\lesssim 2^{- s \|\bj/\balpha\|_{\infty}} 2^{j_\ell} |m_{\balpha,s-1}(2^{\bj}\xi)|  \langle 2^{j_\ell}\xi_\ell \rangle^{\frac{1}{\alpha_\ell}-1} \lesssim 2^{- \|\bj/\balpha\|_{\infty}} 2^{j_\ell/\alpha_\ell}  |F_{\balpha,s-1,\bj}(\xi)| \langle \xi_\ell \rangle^{\frac{1}{\alpha_\ell}-1} 
\,,\\
|\partial^2_\ell F_{\balpha,s,\bj}(\xi)| &\lesssim 2^{- s \|\bj/\balpha\|_{\infty}} 2^{2j_\ell} \Big( |m_{\balpha,s-2}(2^{\bj}\xi)|  \langle 2^{j_\ell}\xi_\ell \rangle^{\frac{2}{\alpha_\ell}-2}
+ |m_{\balpha,s-1}(2^{\bj}\xi)| \langle 2^{j_\ell}\xi_\ell \rangle^{\frac{1}{\alpha_\ell}-2}  \Big) \\
&\lesssim   2^{- 2 \|\bj/\balpha\|_{\infty}} 2^{2j_\ell/\alpha_\ell} |F_{\balpha,s-2,\bj}(\xi)|  \langle \xi_\ell \rangle^{\frac{2}{\alpha_\ell}-2}
+ 2^{- \|\bj/\balpha\|_{\infty}} 2^{j_\ell/\alpha_\ell} |F_{\balpha,s-1,\bj}(\xi)| \langle \xi_\ell \rangle^{\frac{1}{\alpha_\ell}-2}  \,.
\end{align*}

Taking~\eqref{unifest} into account, we realize that the term $|F_{\balpha,s-2,\bj}(\xi)|  \langle \xi_\ell \rangle^{\frac{2}{\alpha_\ell}-2}$ and the term $|F_{\balpha,s-1,\bj}(\xi)| \langle \xi_\ell \rangle^{\frac{1}{\alpha_\ell}-2}$ are uniformly bounded in the range $\xi\in[-4,4]^d$ with respect to $\bj\in\N_0^d$. Since further $2^{- 2 \|\bj/\balpha\|_{\infty}} 2^{2j_\ell/\alpha_\ell}\le1$ and $2^{- \|\bj/\balpha\|_{\infty}} 2^{j_\ell/\alpha_\ell}\le 1$, Step~1 is finished.

{\em Step 2.} For the proof of the converse inequality 
$\Vert f \Vert_{\widetilde{W}^{s,\balpha}_p} \lesssim \Vert f \Vert_{W^{s,\balpha}_p}$ we argue analogously to Step 1 and use this time the multiplier
$$
\widetilde{M}_{\balpha,s,\bj}(\xi) := \frac{\psi_{\bj}(\xi) 2 ^{s \|\bj/\balpha\|_{\infty}}
}{m_{\balpha,s}(\xi)}\,.
$$
It is well-defined since $m_{\balpha,s}>0$, and we have, using the same notation as in Step~1,
$$
\widetilde{M}_{\balpha,s,\bj}(2^{j_1}\xi_1, \ldots, 2^{j_d} \xi_d) = \frac{\tilde{\psi}_{s(\bj)}(\xi_1,\ldots,\xi_d)}{2 ^{-s \|\bj/\balpha\|_{\infty}} m_{\balpha,s}(2^{j_1}\xi_1,\ldots,2^{j_d}\xi_d)}\,.
$$
Again, it is not difficult to
check that for every component the second derivatives are bounded on $[-4,4]^d$ independently of $\bj$.

By Lemma~\ref{Fouriermultiplier}, applied with $\rho_{\bj}:=\widetilde{M}_{\balpha,s,\bj}$, $f_{\bj}:= \mathcal{F}^{-1} m_{\balpha,s}\varphi_{\bj} \mathcal{F}f$, and $r:=2$, we get
\begin{eqnarray*}
\Vert f \Vert_{\widetilde{W}^{s,\balpha}_p} & \asymp &  \Big\Vert \Big(  \sum\limits_{\bj \in {\mathbb{N}}_0^d}  2^{2s \|\bj/\balpha\|_{\infty}} \big\vert
{\mathcal{F}}^{-1} \varphi_{\bj} \mathcal{F}f \big\vert ^2 \Big)^{\frac{1}{2}} \Big\Vert_p   \\
& =&  \Big\Vert \Big(
\sum_{\bj \in {\mathbb{N}}_0^d} 2^{2s \|\bj/\balpha\|_{\infty}} \big\vert {\mathcal{F}}^{-1}
\psi_{\bj} m_{\balpha,s}^{-1} m_{\balpha,s} \varphi_{\bj}\mathcal{F}f  \big\vert ^2\Big )^{\frac{1}{2}} \Big\Vert_p  \\
&  =&   \Big\Vert \Big( \sum_{\bj \in {\mathbb{N}}_0^d}  \big\vert
{\mathcal{F}}^{-1} \widetilde{M}_{\balpha,s,\bj} m_{\balpha,s} \varphi_{\bj}\mathcal{F}f  \big\vert ^2\Big )^{\frac{1}{2}}
\Big\Vert_p \\
&  \lesssim   & \Big( \sup_{\bj\in {\mathbb{N}}_0^d} \big\| \widetilde{M}_{\balpha,s,\bj}(2^{j_1}\cdot,\ldots, 2^{j_d}\cdot) \big\|_{S^2_2W} \Big) \cdot \Big\Vert \Big(
\sum_{\bj \in {\mathbb{N}}_0^d} \big\vert ({\mathcal{F}}^{-1} [\varphi_{\bj} m_{\balpha,s} {\mathcal{F}}f] ) \big\vert^ 2
\Big)^{\frac{1}{2}} \Big\Vert _p  \\
& \lesssim  &\Big\Vert {\mathcal{F}}^{-1} \big[m_{\balpha,s} {\mathcal{F}} f\big] \Big\Vert_p  
 \asymp  \Vert f \Vert_{W^{s,\balpha}_p}.
\end{eqnarray*}
\end{proof}

\begin{remark}\label{rem:comp}
We mention that, in contrast to this result, in case $A=B$ we only have
coincidence when $p=q=2$. A proof can be found in~\cite{ACJRV:2014}.
\end{remark}

As a direct consequence of Theorem~\ref{comp} and Theorem~\ref{main_wav},
we obtain new characterizations of classical Sobolev spaces via hyperbolic wavelets.

\begin{theorem}\label{thm:main_wave1} Let $1<p<\infty$, $s\in\R$, and $\balpha=(\alpha_1,\ldots,\alpha_d)>0$ such that $\sum_{i=1}^{d}\alpha_i=d$. Let further $\psi_0,\psi$ be wavelets satisfying (K) and (L) with
\begin{equation*}
      K, L>\sigma_{p,2}+|s|/\alpha_{\text{min}}.
\end{equation*}
Then any $f\in \mathcal{S}'(\R^d)$ belongs to $W^{s,\balpha}_{p}(\R^d)$ if and only if it can be represented as
\begin{equation}\label{f2sobolev}
      f = \sum\limits_{\bj\in \N_0^d} \sum\limits_{\bk\in \Z^d} \lambda_{\bj,\bk} \psi_{\bj,\bk}
\end{equation}
with $(\lambda_{\bj,\bk})_{\bj,\bk} \in \tilde{f}^{s,\balpha}_{p,2}$. The representation \eqref{f2sobolev} converges
unconditionally in $\mathcal{S}'(\R^d)$ and in $W^{s,\balpha}_{p}(\R^d)$. In addition,
$(\psi_{\bj,\bk})_{\bj,\bk}$ is an unconditional basis in $W^{s,\balpha}_{p}(\R^d)$. The sequence of coefficients
$\lambda:=\lambda(f)=(\lambda_{\bj,\bk})_{\bj,\bk}$ is uniquely determined via
\begin{equation*}
    \lambda_{\bj,\bk}:=\lambda_{\bj,\bk}(f) = 2^{\|j\|_1}\langle f , \psi_{\bj,\bk}\rangle
\end{equation*}
and we have the wavelet isomorphism (equivalent norm)
$$
    \|f\|_{W^{s,\balpha}_{p}(\R^d)} \asymp \|\lambda(f)\|_{\tilde{f}^{s,\balpha}_{p,2}}\quad,\quad f\in
W^{s,\balpha}_{p}(\R^d)\,.
$$
\end{theorem}

Analogously, combining Theorem~\ref{comp} with Theorem~\ref{thm:haarmain},
we also derive new characterizations of Sobolev spaces with the 
hyperbolic Haar system $\mathcal{H}_{d}$ from~\eqref{HaarSd}.

\begin{theorem}\label{thm:haarSobolev}
	Let $1<p<\infty$, $s\in\R$, and $\balpha=(\alpha_1,\ldots,\alpha_d)>0$ such that $\sum_{i=1}^{d}\alpha_i=d$.
	Further, assume
	\begin{align*}
	|s|/\alpha_{\text{min}}<
	 \min\Big\{\frac{1}{p},1-\frac{1}{p}\Big\}.
	\end{align*}
	Then the Haar system $\mathcal{H}_{d}=(h_{\bj,\bk})_{\bj,\bk}$ from \eqref{HaarSd} constitutes an unconditional basis of $W^{s,\balpha}_{p}(\R^d)$ with associated sequence space
	$\tilde{f}^{s,\balpha}_{p,2}$.
	The unique sequence of basis coefficients for $f\in W^{s,\balpha}_{p}(\R^d)$
	is determined by $\lambda:=\lambda(f)=(\lambda_{\bj,\bk})_{\bj,\bk}$ with
	\begin{equation}\label{coeffSobHaar}
	\lambda_{\bj,\bk}:=\lambda_{\bj,\bk}(f) = 2^{\|\bj\|_1}\langle f , h_{\bj,\bk}\rangle \,.
	\end{equation}
	Further, we have the wavelet isomorphism (equivalent norm)
	$$
	\|f\|_{W^{s,\balpha}_{p}(\R^d)} \asymp \|\lambda(f)\|_{\tilde{f}^{s,\balpha}_{p,2}}\quad,\quad f\in
	W^{s,\balpha}_{p}(\R^d)\,.
	$$
	
	In addition, those elements of
	$\mathcal{S}'(\R^d)$ belonging to $W^{s,\balpha}_{p}(\R^d)$
	are characterized by either of the following two criteria:
	
	\begin{enumerate}[(i)]
		\item  
		$f$ can be represented as a sum
		\begin{equation}\label{sumSobHaar}
		f = \sum\limits_{\bj\in \N_0^d} \sum\limits_{\bk\in \Z^d} \lambda_{\bj,\bk} h_{\bj,\bk} \quad\text{converging (weak*ly) in $\mathcal{S}'(\R^d)$}
		\end{equation}
		with coefficients $(\lambda_{\bj,\bk})_{\bj,\bk} \in \tilde{f}^{s,\balpha}_{p,2}$ (with respect to some chosen ordering).
		\item 
		With $\lambda(f)$ being defined as in \eqref{coeffSobHaar}, it holds
		\begin{equation*}
		\lambda(f)=(\lambda_{\bj,\bk})_{\bj,\bk} \in\tilde{f}^{s,\balpha}_{p,2}.
		\end{equation*}
	\end{enumerate}
	In both cases, the sequence $(\lambda_{\bj,\bk})_{\bj,\bk}$ is necessarily the sequence of basis coefficients and the representation \eqref{sumSobHaar} converges
	unconditionally to $f$ in $W^{s,\balpha}_{p}(\R^d)$.
\end{theorem}

\section{Hyperbolic and classical (anisotropic) BLT spaces}

The next and final theorem of this paper complements the statement of Theorem~\ref{comp},
showing that in general the spaces $\widetilde{A}^{s,\balpha}_{p,q}(\R^d)$ and $A^{s,\balpha}_{p,q}(\R^d)$,
with $A\in\{B,F\}$, do not coincide.

\begin{theorem}\label{neg_comp}
Let $0< p, \, q \le \infty$, $s\in\mathbb{R}$, and $\balpha=(\alpha_1,\ldots,\alpha_d)>0$ with $\sum_{i=1}^{d}\alpha_i=d$. 
\begin{enumerate}[(i)]
\item If $B^{s,\balpha}_{p,q}(\mathbb{R}^d) = {\widetilde{B}}^{s,\balpha}_{p,q}(\mathbb{R}^d)$ then $p=q=2$. 
\item  In the range $0<p<\infty$: If $F^{s,\balpha}_{p,q}(\mathbb{R}^d) = {\widetilde{F}}^{s,\balpha}_{p,q}(\mathbb{R}^d)$ then $q=2$ and
$1<p<\infty$.
\end{enumerate}
\end{theorem}

\begin{remark}
	The Besov result (i) follows directly from the very general investigations on
	embeddings between decomposition spaces conducted in~\cite{Voigt2016}. Even more, the findings
	there allow to strengthen the statement to more general independent parameters, namely 
	\begin{align*}
	B^{s_1,\balpha_1}_{p_1,q_1}(\mathbb{R}^d) = {\widetilde{B}}^{s_2,\balpha_2}_{p_2,q_2}(\mathbb{R}^d) \quad\Leftrightarrow\quad  p_1=p_2=q_1=q_2=2 \text{ and }\balpha_1=\balpha_2\text{ and }s_1=s_2\,.
	\end{align*}
	The results of~\cite{Voigt2016}, however, are not applicable in the proof of (ii) since Triebel-Lizorkin spaces do not fit into the decomposition space framework. In the sequel, we will therefore give a proof
	for the $F$-case~(ii) which by slight modifications would also provide a direct way to establish the $B$-case~(i).
\end{remark}

\noindent
Before we start, let us remind ourselves that the converse statement of (ii), the coincidence of $F^{s,\balpha}_{p,q}$ and $\widetilde{F}^{s,\balpha}_{p,q}$ when $1<p<\infty$ and $q=2$, is given by Theorem~\ref{comp}. The coincidence of 
$B^{s,\balpha}_{p,q}$ and $\widetilde{B}^{s,\balpha}_{p,q}$ when $p=q=2$, the converse of (i), is further observed in Remark~\ref{rem:comp}.

\begin{proof}[of Theorem~\ref{neg_comp} (ii)] 
{\em Step 1: Preparation.} Fix an anisotropy vector $\balpha=(\alpha_1,\ldots,\alpha_d)$ and consider a
univariate resolution of unity $(\theta_j)_{j\in\N_0}$ of the following form:

The generator $\theta_0\in\mathcal{S}(\R)$ shall satisfy
\begin{align*}
\supp \theta_0 \subset [-2^{\alpha_{\text{min}}/3}, 2^{\alpha_{\text{min}}/3}]
\qquad\text{and}\qquad \theta_0=1 \quad\text{on}\quad [-1,1]\,,
\end{align*}
and the functions $\theta_j$ for $j\in\N$ shall be obtained via $\theta_j(\cdot):= \theta_0(2^{-j}\cdot ) - \theta_0(2^{-(j-1)}\cdot )$.

Using $(\theta_j)_{j\in\N_0}$, we can then construct two multivariate resolutions of unity on $\R^d$. 
First, via simple tensorization, we get the hyperbolic resolution $(\theta_{\bj})_{\bj\in\N_0^d}$ with
\begin{align*}
\theta_{\bj} &:= \theta_{j_1} \otimes \ldots \otimes \theta_{j_d} \,,\quad \bj=(j_1,\ldots,j_d)\in\N_0^d\,.
\end{align*}
It clearly fulfills all the specifications formulated at the beginning of Section~3.

Second, putting $\varphi_0^{\balpha} := \theta_0 \otimes \ldots \otimes \theta_0$ and
\begin{align*}
\varphi_j^{\balpha} := \varphi_0^{\balpha}(2^{-j\balpha}\cdot ) - \varphi_0^{\balpha}(2^{-(j-1)\balpha}\cdot ) \quad\text{for }j\in\N \,,
\end{align*}
we obtain $(\varphi_j^{\balpha})_{j\in\N_0}$, which is a classical anisotropic resolution of unity in compliance with the definition from Subsection~2.1.

For parameters $\alpha>0$ and $\ell\in\N$, let us next introduce the intervals
\[
I^\alpha_\ell := 2^{(\ell-1)\alpha} \cdot [2^{\alpha_{\text{min}}/3},2^{\alpha}]
\quad\text{and}\quad  J^\alpha_\ell:=[-2^{\ell\alpha}, 2^{\ell\alpha}] \,.
\]
Then $\theta_0(2^{-\ell\alpha}\cdot)=1$ on $J^\alpha_\ell$ and
$\theta_0(2^{-\ell\alpha}\cdot)- \theta_0(2^{-(\ell-1)\alpha}\cdot)=1$ on $I^\alpha_\ell$.
In particular, $\theta_{j}=1$ on $I^{1}_j$ for every $j\in\N$ and
thus $\theta_{\bj}=1$ on $I^{1}_{j_1}\times\ldots\times I^{1}_{j_d}$.
Further, we have $\varphi_0^{\balpha}(2^{-j\balpha}\cdot )=1$ on $J^{\alpha_1}_j\times\ldots\times J^{\alpha_d}_j$ and as a consequence $\varphi_j^{\balpha} = 1$ on  $(J^{\alpha_1}_j\times\ldots\times J^{\alpha_d}_j) \backslash (J^{\alpha_1}_{j-1}\times\ldots\times J^{\alpha_d}_{j-1})$.
This, in turn, implies $\varphi_j^{\balpha} = 1$ on the subset 
$I^{\alpha_1}_j\times\cdots\times I^{\alpha_{d-1}}_j \times J^{\alpha_d}_j$.

Observe now that for every $\ell\in\N$ and every $i\in\{1,\ldots,d\}$ either
$I^{\alpha_i}_\ell\cap I^{1}_{\lfloor\ell\alpha_i\rfloor}$ or $I^{\alpha_i}_\ell\cap I^{1}_{\lfloor\ell\alpha_i\rfloor+1}$ is a nonempty interval of nonzero length.
This is due to the fact that always 
\begin{align}\label{interval_est}
2^{\gamma}\cdot\mathcal{L}(I^{\alpha_i}_\ell)  \le \mathcal{R}(I^{1}_{\lfloor\ell\alpha_i\rfloor})\le  \mathcal{R}(I^{\alpha_i}_\ell)
\quad\text{or}\quad  2^{\gamma}\cdot\mathcal{L}(I^{1}_{\lfloor\ell\alpha_i\rfloor+1}) \le \mathcal{R}(I^{\alpha_i}_\ell)\le \mathcal{R}(I^{1}_{\lfloor\ell\alpha_i\rfloor+1})  \,,
\end{align}
where $\mathcal{L}(I)$ and $\mathcal{R}(I)$ denote the left resp.\ right endpoint
of a given interval $I=[a,b]$ and $\gamma=\amin^2/8$. The verification of this fact is postponed
to Step~3 at the end of this proof.

As a consequence, for each $i\in\{1,\ldots,d\}$ and each $\ell\in\N$, we may pick one of those intersections with nonvanishing interior and denote it by $\tilde{I}^{(i)}_\ell$. Depending on our choice, we then either have
\begin{align}\label{auxlabel}
\tilde{I}^{(i)}_\ell=I^{\alpha_i}_\ell\cap I^{1}_{\lfloor\ell\alpha_i\rfloor}  \quad\text{or}\quad \tilde{I}^{(i)}_\ell=I^{\alpha_i}_\ell\cap I^{1}_{\lfloor\ell\alpha_i\rfloor+1}.
\end{align}
Due to the nonvanishing interior of $\tilde{I}^{(i)}_\ell$ we can further fix nontrivial functions 
\begin{align*}
h^{[i]}_\ell\in\mathcal{S}(\R) \quad\text{with}\quad \supp h^{[i]}_\ell\subset \tilde{I}^{(i))}_\ell \,,\quad i\in\{1,\ldots,d-1\}\,.  
\end{align*}
With this preparation we are finally ready for the main argumentation.

{\em Step 2: Main Proof.}
For $\ell\in\N$ let us consider $g_{\ell}:\,\R\to {\mathbb{C}}$ with the property 
\begin{align}\label{propgl}
\supp(\mathcal{F}g_{\ell})\subset J^{\alpha_d}_{\ell}=[-2^{\ell\alpha_d}, 2^{\ell\alpha_d}]
\end{align} and associate a multivariate function $f_{\ell}:\,\R^{d}\to {\mathbb{C}}$  defined by its Fourier transform
$$
\mathcal{F}f_{\ell}(\xi_1,\ldots, \xi_d) := h^{[1]}_{\ell}(\xi_1) h^{[2]}_{\ell}(\xi_2)\cdots h^{[d-1]}_{\ell}(\xi_{d-1}) \mathcal{F}g_{\ell}(\xi_d)\,,
$$
where $h^{[i]}_\ell$ are the functions introduced at the end of Step~1.

Since $\mathcal{F}f_{\ell}$ is supported inside $I^{\alpha_1}_\ell\times\ldots\times I^{\alpha_{d-1}}_\ell\times J^{\alpha_d}_{\ell}$, on which $\varphi^{\balpha}_{\ell}=1$ according to our considerations in Step~1, we can easily
compute the classical anisotropic Triebel-Lizorkin \mbox{(quasi-)norm} of $f_{\ell}$.
Denoting by $(\Delta^{\varphi}_j)_{j\in\N_0}$ the
Littlewood-Paley analysis associated to $(\varphi_j^{\balpha})_{j\in\N_0}$, 
we have
\begin{align*}
\Vert f_{\ell} \Vert_{F^{s,\balpha}_{p,q}(\R^d)} = \Big\Vert \Big( \sum_{j \ge 0} 2^{jsq} \vert {\Delta^{\varphi}_j} f_{\ell} \vert^q
\Big)^{\frac{1}{q}} \Big\Vert_p = 2^{\ell s} \Vert\Delta^{\varphi}_{\ell} f_{\ell}  \Vert_p = 2^{\ell s} \Vert f_{\ell} \Vert_p\,. 
\end{align*}
Moreover, as $f_{\ell}$ is a tensor product, we can compute
$$
\Vert f_{\ell} \Vert_p = \Vert {\mathcal{F}}^{-1} (h^{[1]}_{\ell}) \Vert_p\cdot...\cdot\Vert {\mathcal{F}}^{-1}
(h^{[d-1]}_{\ell}) \Vert_p \Vert g_{\ell} \Vert_p
= C_{\ell} \Vert g_{\ell} \Vert_p 
$$
with $C_{\ell}:= C^{(1)}_{\ell}\cdot...\cdot C^{(d-1)}_{\ell}$
and
$C^{(i)}_{\ell}:=\Vert {\mathcal{F}}^{-1} (h^{[i]}_{\ell}) \Vert_p$ 
for $i\in\{1,\ldots,d-1\}$.
Altogether, we end up with
$$
\Vert f_{\ell} \Vert_{F^{s,\balpha}_{p,q}} \asymp 2^{\ell s} C_{\ell} \Vert g_{\ell} \Vert_p\,.
$$

We proceed with the computation of the hyperbolic Triebel-Lizorkin (quasi-)norm 
of~$f_{\ell}$.
It follows right from the definition of the intervals $\tilde{I}^{(i)}_{\ell}$ from~\eqref{auxlabel} that there exist numbers $k_i(\ell)\in\N$, 
either taking the value $\lfloor\ell\alpha_i\rfloor$ or the value $\lfloor\ell\alpha_i\rfloor+1$, such
that $\tilde{I}^{(i)}_{\ell}\subset I^{1}_{k_i(\ell)}$.
Hence, due to $\supp \mathcal{F}f_\ell \subset \tilde{I}^{(1)}_{\ell}\times\cdots\times \tilde{I}^{(d-1)}_{\ell} \times J^{\alpha_d}_\ell$ the function $\mathcal{F}f_\ell$ is supported inside $I^{1}_{k_1(\ell)}\times\cdots\times I^{1}_{k_{d-1}(\ell)} \times J^{\alpha_d}_\ell$.

Let now $(\Delta^{\theta}_{\bj})_{\bj\in\N^d_0}$ denote the Littlewood-Paley analysis 
corresponding to $(\theta_{\bj})_{\bj\in\N^d_0}$ and let us abbreviate
$k_i(\ell)$ by $k_i$
and the vector $(k_1,\ldots,k_{d-1},j_d)$ by $\overline{\ell}_{j_d}$. 
Then, since
\[
\sum_{j_d =0}^{\lfloor\ell\alpha_d\rfloor+1} \theta_{\overline{\ell}_{j_d}} = 1
\quad\text{on}\quad I^{1}_{k_1}\times\cdots\times I^{1}_{k_{d-1}} \times J^{\alpha_d}_\ell \,,
\]
we calculate for the hyperbolic Triebel-Lizorkin (quasi-)norm of $f_{\ell}$
\begin{align*}
\Vert f_{\ell} \Vert_{\widetilde{F}^{s,\balpha}_{p,q} }& = \Big\Vert \Big( \sum_{\overline{j} \in {\mathbb{N}}_0^d}  2^{\max
	\{j_1/\alpha_1,\ldots, j_d /\alpha_d\}sq}\vert\Delta^{\theta}_{\overline{j}} f_{\ell}(\cdot) \vert^q  \Big)^{1/q} \Big\Vert_p \\
& = \Big\Vert \Big(  \sum_{j_d =0}^{\lfloor\ell\alpha_d\rfloor+1} 2^{\max \{k_1/\alpha_1,\ldots, k_{d-1}/\alpha_{d-1}, j_d/\alpha_d\}sq} \vert
\Delta^{\theta}_{\overline{\ell}_{j_d}} f_{\ell}(\cdot) \vert^q\Big)^{1/q} \Big\Vert_p\\
& \asymp  2^{\ell s} \Big\Vert \Big(\sum_{j_d=0}^{\lfloor\ell \alpha_d\rfloor+1} \vert \Delta^{\theta}_{\overline{\ell}_{j_d}} f_{\ell}(\cdot) \vert^q \Big)^{1/q} \Big\Vert_p
=  2^{\ell s} C_{\ell} \Vert g_{\ell} \Vert_{F^{0}_{p,q}}
\end{align*}
with the same constant $C_\ell$ as obtained before in the computation of $\Vert f_{\ell} \Vert_{F^{s,\balpha}_{p,q}}$.

Now we come to the core argument.
Assuming that the spaces $F^{s,\balpha}_{p,q}(\R^d)$ and $\widetilde{F}^{s,\balpha}_{p,q}(\R^d)$ coincide, the
associated (quasi-)norms would be equivalent. By our calculations, this would imply that $\Vert g_{\ell} \Vert_{F^0_{p,q}(\R)}$ is equivalent to
$\Vert g_{\ell} \Vert_{L_p(\R)}$ for any band-limited function $g_{\ell}$ with frequency support as in~\eqref{propgl}. Moreover, since the proof holds true for all $\ell\in\N$ this 
equivalence remains valid for any band-limited function $g$ on $\R$.

But, as a consequence of Lemma~\ref{lem:auxeqiv2}(iii), since the sequence $(f^{(3)}_N)_N$ constructed in its proof consists of band-limited functions, this is
only possible in the range $1<p<\infty$. Furthermore, if $1<p<\infty$ the band-limited functions are dense in $L_p(\R)$ as well as $F^{0}_{p,q}(\R)$.
Hence, by Lemma~\ref{lem:auxeqiv2}(i) also $q=2$ is a necessary condition. It now only remains
to verify~\eqref{interval_est}.

{\em Step 3: Proof of~\eqref{interval_est}.}
We distinguish two cases depending on the size of the quantity $\delta:=\ell\alpha_i-\lfloor\ell\alpha_i\rfloor\in[0,1)$. 
Let us subsequently abbreviate $\rho:=\frac{\amin}{4+\amin}$ and 
$\sigma:=\frac{8\amin}{3\amin+12}=\frac{2}{3}(1-\rho)\amin$. 
Recalling that $\amin=\min_{i\in\{1,\ldots,d\}}\{\alpha_i\}\in(0,1]$, we note that 
$\rho\in(0,\frac{1}{5}]$ and $\sigma\in(0,\frac{2}{3}]$. 

In case $\delta\in[0,\sigma)$ we have 
$
\delta\le(1-\rho)(\alpha_i-\amin/3)
$
and thus 
\begin{align*}
\log_2\mathcal{R}\big(I^{1}_{\lfloor\ell\alpha_i\rfloor}\big) &= \lfloor \ell\alpha_i\rfloor = \ell\alpha_i - \delta \ge \ell\alpha_i - (1-\rho)(\alpha_i-\amin/3) \\
&= (\ell-1)\alpha_i + \amin/3 + \rho(\alpha_i -\amin/3) 
\ge \log_2\mathcal{L}\big(I^{\alpha_i}_{\ell}\big) + \frac{2}{3}\rho\amin \,,
\end{align*} 
where we used
$\mathcal{L}\big(I^{\alpha_i}_{\ell}\big)=2^{(\ell-1)\alpha_i + \amin/3}$
and $\mathcal{R}\big(I^{1}_{\lfloor\ell\alpha_i\rfloor}\big)=2^{\lfloor \ell\alpha_i\rfloor}$.
In view of $\gamma=\amin^2/8< \frac{2}{3}\rho\amin$ and since we always have
$\lfloor\ell\alpha_i\rfloor\le \ell\alpha_i \le \lfloor\ell\alpha_i\rfloor+1$, i.e.\
\begin{align}\label{RPoEst}
\log_2\mathcal{R}\big(I^{1}_{\lfloor\ell\alpha_i\rfloor}\big)\le  \log_2\mathcal{R}\big(I^{\alpha_i}_\ell\big)\le \log_2\mathcal{R}\big(I^{1}_{\lfloor\ell\alpha_i\rfloor+1}\big)   \,,
\end{align}
the left inequality in \eqref{interval_est} is hence valid in the respective range of $\delta$.

In case $\delta\in[\sigma,1)$ we first estimate 
\[
\delta>\frac{7\amin}{3\amin+1}=(1-\rho)+\rho\amin/3 \ge (1-\rho)\amin + \rho\amin/3 \,,
\]
from which we deduce, using $\ell\alpha_i=\lfloor\ell\alpha_i\rfloor +\delta$
and $\mathcal{L}\big(I^{1}_{\lfloor\ell\alpha_i\rfloor+1}\big)=2^{\lfloor\ell\alpha_i\rfloor + \amin/3}$,
\begin{align*}
\log_2 \mathcal{R}\big(I^{\alpha_i}_\ell\big)=\ell\alpha_i  > \lfloor\ell\alpha_i\rfloor + (1-\rho)\amin + \rho\amin/3 = \log_2 \mathcal{L}\big(I^{1}_{\lfloor\ell\alpha_i\rfloor+1}\big)  + \frac{2}{3}(1-\rho)\amin \,.
\end{align*} 
This time, again taking into account~\eqref{RPoEst} and $\gamma=\amin^2/8< \frac{2}{3}(1-\rho)\amin$,
the right inequality in~\eqref{interval_est} holds true.
Altogether, the proof of~\eqref{interval_est} is thus finished. 
\end{proof}

The behavior of the $L_p$-(quasi-)norms
in relation to the $A^0_{p,q}$-(quasi-)norms is crucial for the proof of
Theorem~\ref{neg_comp}. Concretely, we have shown for $A\in\{B,F\}$
\begin{align}\label{aaabbc}
A^{s,\balpha}_{p,q}(\mathbb{R}^d) = {\widetilde{A}}^{s,\balpha}_{p,q}(\mathbb{R}^d) \quad\Longleftrightarrow\quad 
\|f\|_{A^{0}_{p,q}(\mathbb{R})} \asymp \|f\|_{L_{p}(\mathbb{R})} \quad\text{for band-limited functions $f$}.
\end{align}
On the right-hand side, the (quasi-)norms are thereby all classical and only the  univariate case matters.
Using known embedding theorems, the exact parameters for equality could therefore be determined (see \cite{triebel:1978}  Section 2.3.2 or
\cite{SicTri} Theorem 3.1.1., for example). 

Prefering a direct and shorter route, the following lemma provides a simple and quantitative argument for what we need. It investigates the behavior of the respective (quasi-)norms for certain sequences of test functions. 
As a consequence of statement (i), we extract the necessity $p=q=2$ for equality
in \eqref{aaabbc}.
From (ii) we further obtain $p=q$ in the B-case. Statement (iii) yields
$1<p<\infty$ in the F-case. Altogether, this shows that the Sobolev spaces in
Theorem~\ref{comp} are precisely those, where equality holds true.

\begin{lemma}\label{lem:auxeqiv2}
Assume $0<p<\infty$, $0<q\le\infty$, $A\in\{B,F\}$. 
There are sequences $(f^{(i)}_N)_{N\in\N}$, $i\in\{1,2,3\}$, of functions on $\R$ such that
	\begin{enumerate}[(i)]
		\item $\|f^{(1)}_N\|_p \asymp N^{1/2}$ and $\|f^{(1)}_N\|_{A^0_{p,q}(\R)}\asymp N^{1/q}$,
		\item $\|f^{(2)}_N\|_p \asymp N^{1/p}$ and $\|f^{(2)}_N\|_{B^0_{p,q}(\R)}\asymp N^{1/q}$,
		\item $\|f^{(3)}_N\|_p \asymp 2^{N(1-1/p)}$ and $\|f^{(3)}_N\|_{F^0_{p,q}(\R)}\gtrsim 
		\begin{cases} 1 &,\,p<1, \\ N^{1/p} &,\,p\ge 1  .
		\end{cases}$ 
	\end{enumerate}
\end{lemma}

\begin{remark}
	In case $q=\infty$ we need to interpret $N^{1/q}\asymp 1$. Further, the case $A=B$ with $p=\infty$ is not considered in Lemma~\ref{lem:auxeqiv2}. By an analogous argument, one can show however that (ii) holds true also for $p=\infty$. So, for $B^{0}_{\infty,q}(\R)$ we have the necessary condition $q=\infty$ to be equivalent to $L_{\infty}(\R)$.
	It is further not difficult to show that the sequence $(f^{(1)}_N)_{N\in\N_0}$ from (i) fulfills $\|f^{(1)}_N\|_\infty \nearrow \infty$ whereas $\|f^{(1)}_N\|_{B^0_{\infty,\infty}(\R)}\asymp 1$. Hence, $B^{0}_{\infty,\infty}(\R)\neq L_{\infty}(\R)$.
\end{remark}

\begin{proof}
ad (i):\: We provide the proof for $q<\infty$. Let $\bar{\varepsilon}=(\varepsilon_0,\varepsilon_1,\ldots)\in\{-1,1\}^{\N_0}$ and define 
\begin{align*}
f_{N,\bar{\varepsilon}}:=\sum_{j=0}^{N} \varepsilon_j \sum_{k=0}^{2^{j}} \psi_{j,k} \,,
\end{align*}	
where $(\psi_{j,k})_{j,k}$ shall be a compactly supported, orthogonal, and $L_\infty$-normalized wavelet system 
with sufficient vanishing moments and smoothness to characterize the space
$A^0_{p,q}(\R)$. Further, for each $j\in\N_0$ and $k\in\{0,\ldots,2^{j}\}$, we assume the support condition $\supp(\psi_{j,k})\subset[0,1]$.

Now we note that in a univariate setting, as considered here, we have the coincidence $\tilde{A}^0_{p,q}(\R)=A^0_{p,q}(\R)$. Hence, using the wavelet isomorphism established by Theorem~\ref{main_wav} for the F-scale and taking into account Remark~\ref{rem:main_wav} for the B-scale, we immediately obtain
\begin{align*}
\|f_{N,\bar{\varepsilon}}\|_{B^0_{p,q}(\R)}
&\asymp \Big( \sum_{j=0}^{N} \Big\| \sum_{k=0}^{2^{j}} \chi_{j,k} \Big\|^{q}_{p} \Big)^{1/q}\asymp N^{1/q} \,, \\
\|f_{N,\bar{\varepsilon}}\|_{F^0_{p,q}(\R)}
&\asymp  \Big\| \Big( \sum_{j=0}^{N} \Big| \sum_{k=0}^{2^{j}} \chi_{j,k} \Big|^{q} \Big)^{1/q} \Big\|_{p} \asymp N^{1/q} \,,
\end{align*}
whereby the (quasi-)norms on the left-hand side do not depend on the choice of $\bar{\varepsilon}$.

From here we proceed with a probabilistic argument and interpret $\bar{\varepsilon}$ as a Rademacher random variable. Then,
for the expectation of the $L_p$-(quasi-)norms over $\bar{\varepsilon}$, 
\begin{align*}
\E_{\bar{\varepsilon}}(\| f_{N,\bar{\varepsilon}}\|^p_p)
=\int_{0}^{1} \int_{\R} \Big| \sum_{j=0}^{N} r_{j}(t)  \sum_{k=0}^{2^{j}} \psi_{j,k}(x) \Big|^p \,dx\,dt \,,
\end{align*}
where $r_j(t):=\sgn(\sin 2^{j}\pi t)$ is the $j$-th Rademacher function.
Applying Khintchine's inequality, we obtain from this
\begin{align*}
\E_{\bar{\varepsilon}}(\| f_{N,\bar{\varepsilon}}\|^p_p)
&=\int_{0}^{1} \int_{0}^{1} \Big| \sum_{j=0}^{N}  r_{j}(t) \sum_{k=0}^{2^{j}} \psi_{j,k}(x) \Big|^p \,dt\,dx \\
&\asymp 
\int_{0}^{1} \Big( \sum_{j=0}^{N}  \Big|   \sum_{k=0}^{2^{j}} \psi_{j,k}(x)  \Big|^{2}  \Big)^{p/2} \,dx \,.
\end{align*}
Next, we define for $j\in\N_0$ the auxiliary functions
\begin{align*}
F_j:= \frac{1}{M} \Big|   \sum_{k=0}^{2^{j}} \psi_{j,k}(x)  \Big|^{2} \quad\text{with $M>0$ such that}\quad \|F_j\|_{\infty}\le 1.
\end{align*}

Observe that $0\le F_j\le 1$. Then, in case $0<p\le 2$,
\begin{align*}
N^{-p/2} \cdot \E_{\bar{\varepsilon}}(\| f_{N,\bar{\varepsilon}}\|^p_p) \asymp 
\int_{0}^{1} \Big( \frac{1}{N} \sum_{j=0}^{N} F_j(x)  \Big)^{p/2} \,dx 
\ge \frac{1}{N} \int_{0}^{1}   \sum_{j=0}^{N} F_j(x)   \,dx \,,
\end{align*}
and $\int_0^1 F_j(x) \,dx \asymp 1$. Hence $ \E_{\bar{\varepsilon}}(\| f_{N,\bar{\varepsilon}}\|^p_p) \gtrsim N^{p/2}$. Also, since $2/p\ge1$, with H\"older
\begin{align*}
\E_{\bar{\varepsilon}}(\| f_{N,\bar{\varepsilon}}\|^p_p) \asymp 
\int_{0}^{1} \Big(  \sum_{j=0}^{N} F_j(x)  \Big)^{p/2} \,dx 
\lesssim \Big(  \int_{0}^{1}  \sum_{j=0}^{N} F_j(x)  \,dx \Big)^{p/2} \asymp N^{p/2}\,.
\end{align*}

In case $2<p<\infty$, we again argue with H\"older
\begin{align*}
 \E_{\bar{\varepsilon}}(\| f_{N,\bar{\varepsilon}}\|^p_p) \asymp 
\int_{0}^{1} \Big(  \sum_{j=0}^{N} F_j(x)  \Big)^{p/2} \,dx 
\gtrsim \Big(  \int_{0}^{1}  \sum_{j=0}^{N} F_j(x)  \,dx \Big)^{p/2} \asymp N^{p/2}\,.
\end{align*}

Further, since $2/p<1$,
\begin{align*}
N^{-p/2} \cdot \E_{\bar{\varepsilon}}(\| f_{N,\bar{\varepsilon}}\|^p_p) \asymp 
\int_{0}^{1} \Big( \frac{1}{N} \sum_{j=0}^{N} F_j(x)  \Big)^{p/2} \,dx 
\lesssim   \frac{1}{N} \int_{0}^{1}  \sum_{j=0}^{N} F_j(x)  \,dx \asymp 1\,.
\end{align*}

Altogether, these estimates show $\E_{\bar{\varepsilon}}(\| f_{N,\bar{\varepsilon}}\|^p_p) \asymp N^{p/2}$. As a consequence, we can choose $f_{N}:=f_{N,\bar{\varepsilon}(N)}$ such that $\|f_N\|^p_p \asymp N^{p/2}$, or equivalently $\|f_N\|_p \asymp \sqrt{N}$.

ad (ii):\:  With the same wavelet system $(\psi_{j,k})_{j,k}$ as before, $L_\infty$-normalized, define
\begin{align*}
f_{N}:=\sum_{j=0}^{N} 2^{j/p} \psi_{j,k(j)} \,,
\end{align*}
where $k(j)$ is chosen such that the (spatial) support of the wavelets is mutually disjoint. Then, using again the wavelet isomorphism from Theorem~\ref{main_wav} and Remark~\ref{rem:main_wav}, we deduce
\begin{align*}
\|f_{N}\|_{B^0_{p,q}(\R)} = \|f_{N}\|_{\widetilde{B}^0_{p,q}(\R)} 
&\asymp \Big( \sum_{j=0}^{N} \Big\|  2^{j/p} \chi_{j,k(j)} \Big\|^{q}_{p} \Big)^{1/q} = \Big( \sum_{j=0}^{N} \big(  2^{j/p} 2^{-j/p} \big)^{q} \Big)^{1/q} \asymp N^{1/q} \,. 
\end{align*}
For the $L_p$-(quasi-)norm we obtain, due to the disjoint support,
\begin{align*}
\|f_{N}\|_{L_p(\R)} = \Big\|  \sum_{j=0}^{N} 2^{j/p} \psi_{j,k(j)} \Big\|_{p} = \Big(  \sum_{j=0}^{N} 2^{j} \| \psi_{j,k(j)} \|^{p}_p \Big)^{1/p} \asymp N^{1/p} \,.
\end{align*}

ad (iii):\: Finally, let $(\varphi_{j})_{j}$ be a (standard) dyadic resolution of unity, with $\varphi_0=1$ in a neighborhood of $0$ and $\varphi_1=\varphi_0(\cdot/2)-\varphi_0$, and put
\begin{align*}
f_{N}:= \mathcal{F}^{-1} \varphi_{0}(2^{-N}\cdot) \,.
\end{align*}
Then $f_{N}= 2^{N} F_1(2^{N}\cdot)$. 
For the $L_p$-(quasi-)norm we thus
compute 
\begin{align*}
\|f_{N}\|_{L_p(\R)} = 2^{N} \| F_1(2^{N}\cdot) \|_{p} \asymp 2^{N(1-1/p)} \,.
\end{align*}
Turning to the $F^0_{p,q}(\R)$-(quasi-)norm, for $N\ge2$, we calculate, writing
$\Phi_0:=\mathcal{F}^{-1} \varphi_{0}$ and $\Phi_1:=\mathcal{F}^{-1} \varphi_{1}$,
\begin{align*}
\|f_{N}\|_{F^0_{p,q}(\R)}
&=  \Big\| \Big( \sum_{j=0}^{\infty} \Big| \big(\mathcal{F}^{-1} (\varphi_j \cdot \varphi_0(2^{-N}\cdot))\big)(\cdot) \Big|^{q} \Big)^{1/q} \Big\|_{p} \\
&\ge  \Big\| \Big( \sum_{j=0}^{N-1} \Big| \big(\mathcal{F}^{-1} \varphi_j  \big)(\cdot) \Big|^{q} \Big)^{1/q} \Big\|_{p} \\
&= \Big\| \Big( |\Phi_{0}|^q + \sum_{j=1}^{N-1} \Big|  2^{j-1}  \Phi_1(2^{j-1}\cdot )   \Big|^{q} \Big)^{1/q} \Big\|_{p} \\
&= \Big\| \Big( |\Phi_{0}|^q + \sum_{j=0}^{N-2} \Big|  2^{j} \Phi_1(2^{j}\cdot ) \Big|^{q} \Big)^{1/q} \Big\|_{p} \,.
\end{align*}
Note that $\Phi_1$ has (infinitely many) vanishing moments and is thus oscillatory.
Assuming w.l.o.g. $|\Phi_1|> \delta$ on a set $I\subset[1,2)$, with $\delta>0$ being some fixed constant,
we can proceed
\begin{align*}
\|f_{N}\|_{F^0_{p,q}(\R)}
&\gtrsim \Big\| \Big( \sum_{j=0}^{N-2} \Big|  2^{j} \chi_{I}(2^{j}\cdot ) \Big|^{q} \Big)^{1/q} \Big\|_{p} \\
&= \Big\|  \sum_{j=0}^{N-2}  2^{j} \chi_{I}(2^{j}\cdot )    \Big\|_{p} 
=\Big(  \sum_{j=0}^{N-2} 2^{jp} \int_{\R} \chi_{I}(2^{j}x ) \,dx  \Big)^{1/p} \\
&\asymp \Big(  \sum_{j=0}^{N-2} 2^{j(p-1)}  \Big)^{1/p}
 \gtrsim \begin{cases}
 1 &\,, p<1\,, \\
 N^{1/p} &\,, p\ge 1\,.
 \end{cases} 
\end{align*} 
\end{proof}


\noindent
{\bf Acknowledgment.} We warmly thank Glenn Byrenheid, Nadi'ia Derevianko, Dinh D\~ung, Dorothee Haroske, Peter Oswald, Andreas Seeger, Winfried Sickel, Vladimir N.
Temlyakov, Hans Triebel, Felix Voigtl\"ander, and Stephane Jaffard for stimulating discussions and comments. This work has been partially supported by the
ANR grant \emph{AMATIS} (ANR2011 BS01 011 02) and the CNRS, Groupe de Recherche \emph{Analyse Multifractale}. In
addition we gratefully acknowledge support by the German Research Foundation (DFG) and the Emmy-Noether programme,
Ul-403/1-1.

\bibliographystyle{plain}
\bibliography{SobolevHWT}

\end{document}